\documentclass{scrartcl}
\usepackage[utf8]{inputenc}
\usepackage{amsmath}
\usepackage{amsthm}
\usepackage{amssymb}
\usepackage{mathtools}
\usepackage{dsfont}
\usepackage{enumitem}
\usepackage[USenglish]{babel}
\usepackage[left=3cm, right=2.5cm, top=2.5cm, bottom=2.5cm]{geometry}
\usepackage[normalem]{ulem}

\usepackage{tikz}

\usepackage{transparent}

\newcommand{\IN}{\mathbb{N}}
\newcommand{\IP}{\mathbb{P}}
\newcommand{\IE}{\mathbb{E}}
\newcommand{\IB}{\mathbb{B}}

\newcommand{\IR}{\mathbb{R}}
\newcommand{\IZ}{\mathbb{Z}}
\newcommand{\cA}{\mathcal{A}}

\newcommand{\cC}{\mathcal{C}}
\newcommand{\cD}{\mathcal{D}}
\newcommand{\cE}{\mathcal{E}}
\newcommand{\cF}{\mathcal{F}}
\newcommand{\cP}{\mathcal{P}}
\newcommand{\cG}{\mathcal{G}}

\newcommand{\cR}{\mathcal{R}}
\newcommand{\cM}{\mathcal{M}}

\newcommand{\cL}{\mathcal{L}}

\newcommand{\bfC}{\boldsymbol{\eta}}
\newcommand{\bfB}{\mathbf{B}}

\newcommand{\bfH}{\mathbf{H}}
\newcommand{\bfG}{\mathbf{G}}
\newcommand{\bfK}{\mathbf{K}}
\newcommand{\bfP}{\mathbf{P}}

\newcommand{\bfi}{\mathbf{i}}
\newcommand{\bfj}{\mathbf{j}}

\newcommand{\bfeta}{\boldsymbol{\eta}}
\newcommand{\seteta}{\boldsymbol{\eta}}
\newcommand{\smallbox}{\mathcal{S}}
\newcommand{\bigbox}{\mathcal{H}}
\newcommand{\bfxi}{\boldsymbol{\xi}}
\newcommand{\bfzeta}{\boldsymbol{\zeta}}

\newcommand{\dV}{V}

\newcommand{\btilde}[1]{\breve{#1}}

\newcommand{\1}{\mathds{1}}
\newcommand{\zero}{\mathbf{0}}
\newcommand{\origin}{\mathbf{0}}
\newcommand{\szero}{\underline{w}}
\newcommand{\sizero}{\underline{0}}
\newcommand{\B}{\mathds{B}}
\newcommand{\PC}{\overline{\mathbb{P}}^\xi}
\newcommand{\PO}{\overline{\mathbb{P}}^{\underline{0}}}
\newcommand{\bW}{\mathbf{W}}
\newcommand{\unaryminus}{\scalebox{0.7}[1.0]{\( - \)}}
\newcommand{\loc}{S} 
\newcommand{\PCI}{\underline{\bfK}} 

\theoremstyle{definition}
\newtheorem{definition}{Definition}[section]
\newtheorem{example}[definition]{Example}
\newtheorem{remark}[definition]{Remark}
\newtheorem{observation}[definition]{Observation}

\newtheorem*{example*}{Example}
\theoremstyle{plain}
\newtheorem{theorem}[definition]{Theorem}
\newtheorem{lemma}[definition]{Lemma}
\newtheorem{corollary}[definition]{Corollary}

\newtheorem{proposition}[definition]{Proposition}
\newtheorem{condition}[definition]{Condition}
\newtheorem{assumption}[definition]{Assumption}

\DeclareFontFamily{U}{mathx}{\hyphenchar\font45}
\DeclareFontShape{U}{mathx}{m}{n}{
	<5> <6> <7> <8> <9> <10>
	<10.95> <12> <14.4> <17.28> <20.74> <24.88>
	mathx10
}{}
\DeclareSymbolFont{mathx}{U}{mathx}{m}{n}
\DeclareFontSubstitution{U}{mathx}{m}{n}
\DeclareMathAccent{\widecheck}{0}{mathx}{"71}
\DeclareMathAccent{\wideparen}{0}{mathx}{"75}

\title{Shape Theorem for the Contact Process in a Dynamical Random Environment}

\author{Michel Reitmeie$\text{r}^*$ and Marco Seiler\footnote{Goethe University Frankfurt, Robert-Mayer-Straße 10, 60486 Frankfurt am Main, Germany, \texttt{reitmeie@math.uni-frankfurt.de}, \texttt{seiler@math.uni-frankfurt.de}}}

\begin{document}
\maketitle
\begin{abstract}
    We study the contact process in a dynamical random environment defined on the vertices and edges of the d-dimensional integer lattice $\IZ^d$. For a broad class of processes, we establish an asymptotic shape theorem for the set $\bfH_t$, which represents the vertices that have been infected up to time $t$. More precisely, we show that, similarly to the basic contact process, this asymptotic shape is characterised by a cone spanned by a convex set $U$, provided certain growth conditions are satisfied. Notably, we find that the asymptotic shape is independent of the initial configuration of the environment. Furthermore, we verify the growth conditions for various types of random environments, such as the contact process on a dynamical graph and a system with switching vertex states, where the monotonicity of the entire process is not guaranteed.

\end{abstract}
{\footnotesize\textit{2020 Mathematics Subject Classification --} Primary 60K35; Secondary 82C22, 92D30.
	}
\section{Introduction}
One of the most well-established mathematical models for studying the spread of an epidemic in a spatially structured population is the contact process, initially introduced by Harris~\cite{harris1974contact}. The population is typically represented as a graph, where vertices correspond to individuals and edges indicate which individuals are considered neighbours. Individuals are either infected or healthy, and an infected individual infects a neighbour with a certain infection rate and recovers independently with a certain recovery rate. Despite being a simplistic toy model, the contact process exhibits remarkably rich behaviour and has been extensively studied. An overview of classical results on lattices and regular trees can be found in the book by Liggett \cite{liggett1999stochastic}. More recently, the contact process on random graphs has been intensively studied. Results of this research direction are discussed in depth in the book by Valesin~\cite{valesin2024contact}. 

In reality, one can observe that population structures are not static but change over time and the infectiousness of an individual varies. Incorporating time-varying population structures or switching states of individuals is one of the most natural variations of the contact process. Therefore, we consider a contact process in a time-evolving random environment. As in the classical setup, individuals are represented by vertices and are either infected or healthy. Additionally, background states are assigned to edges and vertices, which influence the individual rates of recovery and infection. Furthermore, these background states evolve according to Markovian dynamics as time progresses. The model we consider can be seen as a generalisation of the contact process on time-evolving graphs, as studied by Seiler and Sturm~\cite{seiler2023contact}, and the contact process with switching individual activity states on the vertices, as proposed by Blath, Hermann and Reitmeier~\cite{blath2023switching}.

Although the basic contact process was introduced more than 50 years ago, techniques that make it possible to study the process in such dynamical random environments have only been developed in the last two decades. To our knowledge, the first to propose such a variant was Broman~\cite{Broman2007domination} and since then, it has become a thriving research area (see for example \cite{steif2008critical}, \cite{jacob2017scale}, \cite{linker2020contact}, \cite{baptista2024dynamical}, \cite{cardona2024dynamical}).

An important question for variants of the contact process is 
the asymptotic limiting shape of the infection area and the behaviour of the coupled region. Already Richardson~\cite{richardson1973random} studied the limiting shape of a certain stochastic growth model 50 years ago and this research direction remains highly relevant until today. Typically, one employs the Hammersley-Kingman subadditivity theory developed in \cite{hammersley1965first} and \cite{kingman1973subadditive}. For the contact process Durrett and Griffeath \cite{durrett1982several} used these techniques to prove the asymptotic shape results for sufficiently large infection rates. Later the results of Bezuidenhout and Grimmett \cite{bezuidenhout1990contact} enabled verification of the necessary estimates for the entire supercritical regime, thus providing a complete picture for the basic contact process. Since then, this question has been studied for many variations of the contact process and other growth models. For example, Garet and Marchand \cite{garet2012shape} extended the results of \cite{durrett1982several} to the contact process in a static random environment, Deshayes \cite{deshayes2014contact} extended them to the contact process with aging, and Deshayes and Siest further extended them to a certain class of additive growth processes \cite{deshayes2026asymptotic}.

The aim of this article is to study the asymptotic limiting shape of the infection region conditioned on survival for the contact process in various dynamical random environments. First, we show that our process converges to a deterministic asymptotic shape if a collection of exponential estimates regarding the growth speed of the infection region and the extinction time is satisfied. These estimates are comparable to those derived for the basic contact process. Furthermore, we show that the set that describes the limiting shape does not depend on the initial configuration of the background process. This is done in Theorem~\ref{Conjecture1}. A significant difference from many other models in the literature is that, for this result, we do not require monotonicity of the entire system with respect to a total order $\leq$. Instead, we only require a weaker assumption, which we call \textit{worst-case monotonicity}.

For the special case where the background states are independent for every vertex and edge, the worst-case monotonicity assumption can even be omitted, as shown in Theorem~\ref{thm:ShapeTheoremNonMonoton}. This is achieved by coupling the process to an auxiliary process that satisfies the necessary worst-case monotonicity, which allows us to apply Theorem~\ref{Conjecture1}. To ensure that the auxiliary process satisfies the required exponential estimates, we need to show that the two processes couple sufficiently fast.  However, this indicates that the  worst-case monotonicity assumption is still not optimal, in the sense that it is sufficient but not necessary.

After Theorem~\ref{Conjecture1} and its extension, we verify the required exponential estimates across the entire supercritical region when the process is monotone and the vertex and edge updates are independent of each other, see Theorem~\ref{thm:CPIUShapeTheorem}. A key tool  for achieving this is an adaptation and generalisation of the block construction of Bezuidenhout and Grimmett \cite{bezuidenhout1990contact} to our setup, as well as an careful implementation of a restart procedure similar to the one sketched in \cite{durrett1991} and \cite{deshayes2014contact}. In Proposition~\ref{prop:CPDREShapeTheorem}, we extend this result to worst-case monotone processes, and to processes with independent vertex and edge updates, that dominate a supercritical monotone process with independent updates.
This applies, in particular, to the case where the process is monotone and the environment is described by a spin system, see Corollary~\ref{cor:spin_system}. In this specific case, the background states of edges and vertices are no longer independent of each other.

\section{The Model}
Let $G=(V,E)$ be the $d$-dimensional integer lattice with $d\in \IN$, i.e.\
\begin{equation*}
    V=\IZ^d \quad \text{ and }\quad E=\big\{\{x,y\}\subset V: ||x-y||_1=1\big\},
\end{equation*}
where $||\cdot||_1$ is the $\ell_1$-norm. We denote by $\zero$ the origin of $\IZ^d$  and by $\Vec{E}:=\{(x,y)\in V\times V: \{x,y\}\in E\}$ the set of all directed edges.
Let $\loc$ be a finite set equipped with a partial order $\preceq$. With a slight abuse of notation we use $\preceq$ as well for the component-wise order on $\{0,1\}^V\times S^{V\cup E}$.

The \textit{contact process in a dynamical random environment} (CPDRE), denoted by $(\bfeta,\bfxi)=(\bfeta_t,\bfxi_t)_{t\geq 0}$, is a Feller process on $\{0,1\}^V\times \loc^{V\cup E}$.  We call $\bfeta=(\bfeta_t)_{t\geq 0}$ the \textit{infection process}, $\bfxi=(\bfxi_t)_{t\geq 0}$ the \textit{background process} and $\loc$ the \textit{local state space} of the background process.
We prove our main result for a general local state space $\loc$. However, in most cases, we consider the special case $\loc=[N]:=\{0,\dots,N\}$, where $N\in \IN$, equipped with the canonical total order $\leq$.

We assume that $(\bfeta,\bfxi)$ is defined for every initial configuration $(\eta,\xi)$ on the same probability space. In Section~\ref{sec:Construction} we provide an explicit Poisson construction of the joint process that, in particular, guarantees this property.
If necessary, we indicate the initial configuration $(\eta,\xi)$ by adding a superscript, that is $(\bfeta_t^{\eta,\xi},\bfxi_t^{\xi})_{t\geq 0}$ denotes the process starting in configuration $(\eta,\xi)$.
Sometimes we consider the background with an initial state distributed according to some law $\pi$, which we denote by $(\bfeta_t^{\eta,\pi},\bfxi_t^{\pi})_{t\geq 0}$. Moreover, by $\bfeta_t^{x,\xi}$ we denote the infection process where initially only the vertex $x\in V$ is infected.

To formally define the dynamics of the infection process we need two non-negative rate functions $\lambda: \loc^3\to [0,\infty)$ and $r: \loc\to [0,\infty)$. For all $x,y\in V$ with $(x,y)\in \Vec{E}$ we set
\begin{equation}\label{eq:Choice-of-the-rate-functions}
    \lambda_{(x,y)}(\xi):=\lambda(\xi(x),\xi(\{x,y\}),\xi(y)) \quad \text{ and } \quad r_{x}(\xi)=r(\xi(x)),
\end{equation}
where  $\xi \in\loc^{V\cup E}$. We call $\lambda_{(x,y)}(\xi)\geq0$ the \textit{infection rate} from $x$ to $y$ and $r_x(\xi)\geq0$ the \textit{recovery rate} of $x$ given that the environment is in state $\xi$. Note that the value of the infection rate $\lambda_{(x,y)}(\xi)$ depends only on the states of $x,y$ and $\{x,y\}$ and $r_x(\xi)$ only depends on the state of $x$. Given that $\bfxi$ is currently in state $\xi$ the transitions of the infection process currently in state $\eta$ are for all $x\in V$: 
\begin{align}\label{InfectionRatesWithBackground}
\begin{aligned}
\eta(x)&\to 1	\quad \text{ at rate } \sum_{{y:(y,x)\in \Vec{E}}} \lambda_{(y,x)}(\xi) \eta(y)\text{ and }\\
\eta(x)&\to 0 	\quad\text{ at rate } r_x(\xi).
\end{aligned}
\end{align}
By translation invariance of $\bfxi$ and the choices of the rate functions $\lambda(\,\cdot\,)$ and $r(\,\cdot\,)$ it follows immediately that the entire process $(\bfeta,\bfxi)$ is translation invariant.
\begin{remark}
    For notational convenience, we usually interpret $\bfeta_t$ and $\eta$ as subsets of $V$. This is a common notational convention and justified by the existence of the canonical bijection between $\{0,1\}^V$ and $\cP(V)$.
\end{remark}

Our main motivation comes from the following two examples which are covered by our framework in \eqref{eq:Choice-of-the-rate-functions}.
\begin{enumerate}
    \item \textit{Edge-dependent infection rates} given by
    \begin{equation*}
        r_x(\xi)\equiv r>0 \quad\text{and}\quad \lambda_{(x,y)}(\xi)=\lambda_{\xi(\{x,y\})}.
    \end{equation*}
    One interpretation is that $\xi(e)$ indicates the capacity of an edge $e$ and thus determines the infection rate. An important special case is $S=\{0,1\}$ and  
    \begin{equation}\label{eq:CPDP-rates-functions}
       \lambda_{(x,y)}(\xi)=\lambda\cdot\xi(\{x,y\})\text{ for some } \lambda>0 \quad\text{and}\quad r_x(\xi)\equiv 1.
    \end{equation}
    In this case $\bfxi_t(e)$ indicates whether the edge is present or not at a given time $t$. If an edge is present, an infection can be transmitted with rate $\lambda>0$, otherwise it is blocked. Then, the process $(\bfeta,\bfxi)$ can be seen as a \textit{contact process on a dynamical graph} (see \cite{seiler2023contact}).
    \item  \textit{Vertex-dependent infection and recovery rates}. Here the rate functions are
    \begin{align}\label{eq:CPS-rates-functions}
        \lambda_{(x,y)}(\xi)= \lambda_{\xi(x)\xi(y)}\quad \text{ and } \quad   r_{x}(\xi)= r_{\xi(x)}.
    \end{align}
    For the special case $|\loc|=2$, the \textit{contact process with switching} was introduced in \cite{blath2023switching}.
   In that model, each individual has an active and a dormant state, which influence their infectiousness and recovery rate.
    \end{enumerate}    
\begin{remark}
    Our results should still hold as long as $\lambda(\cdot)$ and $r(\cdot)$ only depend on a finite (fixed) neighbourhood. We believe that no fundamentally new ideas are necessary for this extension. However, it would make the arguments more technical and lead to even more complicated notation. Thus, we decided on \eqref{eq:Choice-of-the-rate-functions} as a minimal choice that covers the two mechanisms that motivated this work.
\end{remark}

We assume that $\bfxi$ is an (autonomous) Feller process with state space $\loc^{V\cup E}$. Since the processes $\bfxi^{\xi}$ are defined on the same probability space for all initial configurations $\xi$,  we can define the coupled region of the background at time $t$, analogously to \cite{seiler2023contact}, by
\begin{equation}\label{DefinitionCoupledRegion}
	\Psi_{t}:=\{a\in V\cup E: \bfxi_t^{\xi_1}(a)=\bfxi_t^{\xi_2}(a)\enskip \forall \xi_1,\xi_2\in \loc^{V\cup E}\}
\end{equation}
and the permanently coupled region at time $t$ through 
\begin{align}\label{DefinitionPermanetlyCoupledRegion}
	\Psi'_{t}:=\{a\in V\cup E: a\in \Psi_{s} \,\forall \, s\geq t\},
\end{align}
where $t\geq 0$. Furthermore, we impose the following assumptions on $\bfxi$.

\begin{assumption}\label{AssumptionBackground}
We assume that $\bfxi$ is a \textit{translation invariant} and \textit{finite range} Feller process, which satisfies the following properties:
	\begin{enumerate}
		\item[$(i)$] $\bfxi$ is \textnormal{ergodic}, i.e.\ there exists a unique invariant law $\pi$ such that $\bfxi^{\xi}_t\Rightarrow \pi$ as $t\to \infty$ for all $\xi\in \loc^{V\cup E}$.
		\item[$(ii)$] There exist constants $T,K,\kappa>0$ such that $\IP(a\notin \Psi'_t)<K\exp(-\kappa t)$ for every $a\in V\cup E$ and for all $t\geq T$.
		\item[$(iii)$] $\bfxi$ is a reversible Feller process.
	\end{enumerate}
\end{assumption}
The framework for the background processes $\bfxi$ is fairly general, allowing for the representation of a broad class of models. This includes, in particular, background processes that act independently on every vertex or edge. However, our framework also allows for spatial dependencies in the background. We now provide two classes of models that served as our primary motivation.
\begin{example}[Choices of Background]\label{ex:MainExamples}
\begin{enumerate}
    \item Our leading example is the \textit{contact process with independent updates} (CPIU) $(\bfeta,\bfxi)$. As the name suggests, $\bfxi=(\bfxi(a))_{a\in V\cup E}$ is an independent family of ergodic and reversible Markov chains on a finite state space. Furthermore, we assume that $\bfxi(x)\stackrel{d}{=}\bfxi(y)$ for all $x,y\in V$ and $\bfxi(e)\stackrel{d}{=}\bfxi(e')$ for all $e,e'\in E$.
    
    Since $\bfxi(a)$ for $a\in V\cup E$ has a finite state space, we get by \cite[Theorem 4.9]{levin2017Markov} that it is geometrically ergodic, i.e.\ there exist constants $A,B>0$ such that 
    \begin{equation*}
        \max_{i}|\IP(\bfxi_t(a)=i)-\pi(i)|\leq A e^{-Bt} \quad \text{ for all } \quad t\geq 0.
    \end{equation*}
    In particular, this choice of $\bfxi$ satisfies Assumption~\ref{AssumptionBackground}
    
    An important special case, which we treat in more detail later is $S=\{0,1\}$. In this case, the process is determined by four jump rates, i.e.\ for every $a\in V\cup E$ it has transitions
    \begin{equation*}
        \begin{aligned}
            \xi(a)&\to 1	\quad \text{ at rate } \alpha_V\1_{\{a\in V\}}+\alpha_E\1_{\{a\in E\}},\\
            \xi(a)&\to 0 	\quad\text{ at rate }\beta_V\1_{\{a\in V\}}+\beta_E\1_{\{a\in E\}},
        \end{aligned}
    \end{equation*}
    where $\alpha_V,\alpha_E, \beta_V,\beta_E\geq 0$.
    \item Another relevant example for $\bfxi$ which can have some spatial dependencies is a  monotone, ergodic and reversible spin system of \textit{finite range} $L\geq 0$. In this case $S=\{0,1\}$ and $\bfxi$ has a generator of the form
    \begin{align*}
        \cA_{\text{Spin}} f(\xi)=\sum_{a\in V\cup E} q(a,\xi)\big(f(\xi^{a})-f(\xi)\big),
        \end{align*} 
    where $q(a,\xi)$ is called the flip rate and $\xi^a$ denotes  the configuration flipped at $a$, i.e.\ $\xi^a(a)=1-\xi(a)$ and $\xi^a(a')=\xi(a')$ for $a\neq a'$. \textit{Finite range} here means that the rate $q(a,\xi)$ only depends on the values of $\xi$ on the edges and vertices within distance $L$ from $a$. In the context of edges, this means that the vertices at both ends need to be closer than $L$.
    
    In such systems only one vertex or edge changes its state at a time and there are no coordinated changes. This class of processes obviously contains the dynamical percolation, i.e.\ independent switching for $S=\{0,1\}$, seen in 1. and also more interesting choices, as for example several types of the \textit{ferromagnetic stochastic Ising model}.
\end{enumerate}
\end{example}

The process $(\bfeta,\bfxi)$ is not assumed to be monotone, but we need the following weaker property:
\begin{assumption}[worst-case monotonicity]\label{Ass:WS-Monotonicity}
        We call $(\bfeta,\bfxi)$ worst-case monotone if there exists a configuration $\szero \in \loc^{V\cup E}$ such that  $\bfeta^{x,\xi}_{t}\supset \bfeta^{x,\szero}_{t}$ holds almost surely for all $x\in V$, $\xi\in \loc^{V\cup E}$ and $t\geq 0$.
\end{assumption}
This property states that there exists a background configuration $\szero$ that is the worst configuration for survival of the infection. In particular, translation invariance implies that 
\begin{equation*}
    \lambda_{\min}:=\min_{a\in S^3}\lambda(a)=\lambda(w_{x},w_{\{x,y\}},w_{y})
    \quad\text{and}\quad r_{\max}:=\max_{a\in S}r(a)=r(w_{x})
\end{equation*}
for all $x,y\in V$ with $\{x,y\}\in E$.
\section{Main Results}
\subsection{The Shape Theorem}
Before we state our main result we introduce some more notation. We denote by
\begin{align*}
    \tau^{\eta,\xi}:=\inf\{t\geq 0:\bfeta^{\eta,\xi}_t= 
    \emptyset\}\quad \text{and}\quad t^{\eta,\xi}(x):=\inf\{t\geq 0: x\in \bfeta_t^{\eta,\xi}\}
\end{align*}   the \textit{extinction time} and the \textit{first hitting time} of $x$, respectively. If $\eta=\delta_z$ we also write $\tau^{z,\xi}$ and $t^{z,\xi}(x)$. We often omit $z$ in case $z=\origin$, i.e.\ $\tau^{\xi}=\tau^{\origin,\xi}$ and $t^\xi(x)=t^{\origin,\xi}(x)$. For any initial state $(\eta, \xi)$ we denote the set of all vertices which were infected at least once until time $t$ by
\begin{equation*}
    \bfH^{\eta,\xi}_t:=\bigcup_{s\leq t}\{x\in V:x\in\bfeta^{\eta,\xi}_s\}+\big[\unaryminus\tfrac{1}{2},\tfrac{1}{2}\big]^d.
\end{equation*}
If initially only the origin is infected, i.e.\ $\eta=\delta_\origin$, we write $\bfH^{\xi}_t$. 
Recall that we assumed that $(\bfeta,\bfxi)$ is defined on the same probability space for all initial configurations. Thus, we can denote by
\begin{equation*}
\bfK_{t}^{\xi}:=\{x\in V: \bfeta_t^{\origin,\xi}(x)=\bfeta_t^{V,\xi}(x) \}\quad \text{and}\quad \PCI_{t}^{\xi}:=\{x\in V: x\in \bfK_{s}^{\xi} ~~ \, \forall \, s\geq t\}+\big[\unaryminus\tfrac{1}{2},\tfrac{1}{2}\big]^d
\end{equation*}
the coupled and the permanently coupled region of the infection process $\bfeta$. Note that the expansion of adding $\big[\unaryminus\tfrac{1}{2},\tfrac{1}{2}\big]^d$ to the two sets $\bfH^{\xi}_t$ and $\PCI^{\xi}_{t}$ is only for technical reasons.

Finally, we introduce two notions of balls of radius $r>0$ around $x\in V$ by
\begin{equation*}
    \B_{r}(x):=\{y\in V:||x-y||_1\leq r \}\quad \text{and}\quad
    \B^{E}_r(x):=\big\{\{y,z\}\in E: \{y,z\}\cap\B_r(x)\neq \emptyset\big\},
\end{equation*}
where $\B_{r}(x)$ contains all vertices within distance $r$ from $x$ and $\B^{E}_r(x)$ includes all edges adjacent to such vertices. To simplify notation, we write $\B_r:=\B_r(\origin)$ and $\B^E_r:=\B^E_r(\origin)$. Now we are able to state our first main result.
\begin{theorem}[Asymptotic Shape Theorem]\label{Conjecture1}
	Let $(\bfeta,\bfxi)$ be a CPDRE which satisfies Assumption~\ref{Ass:WS-Monotonicity}. Suppose that $
    \IP(\tau^{\szero}=\infty)>0$ and there exist constants $A,B,M,c>0$ such that for all $\xi\in\loc^{V\cup E}$ and all $x\in V$
	\begin{align}
    \IP(\bfH_t^{\xi}\cap V \not\subset \B_{Mt})&\leq A \exp(-B t),\label{ConjectureAML}\\
	\IP(t< \tau^{\xi}<\infty)&\leq A \exp(-B t),\label{ConjectureSC}\\
    \IP\big(t^{\xi}(x)\geq \frac{||x||}{c}+t,\tau^{\xi}=\infty\big)&\leq A \exp(-B t),\label{ConjectureALL}\\
	\IP(\origin \notin \PCI^{\xi}_{t},\tau^{\xi}=\infty)&\leq A \exp(-B t),\label{ConjectureFC}
	\end{align}
	then there exists a bounded and convex set $U\subset \IR^d$ such that for every $\varepsilon>0$ and every $\xi$
	\begin{equation}\label{desired_result}
	\IP\big(\exists s\geq0 :  t(1-\varepsilon)U\subset (\PCI^{\xi}_t\cap \bfH_t^\xi)\subset \bfH_t^\xi\subset t(1+\varepsilon)U \,\,\forall t\geq s\big|\tau^\xi=\infty)=1.
	\end{equation}
\end{theorem}

Note that the assumptions~\eqref{ConjectureAML}--\eqref{ConjectureFC} are widely recognized in the literature pertaining to the asymptotic shape of infection models (cf. \cite{durrett1982several}). The bounds \eqref{ConjectureAML} and \eqref{ConjectureALL} are often referred to as the \textit{at most linear} and \textit{at least linear} growth, respectively. Equation \eqref{ConjectureSC} reflects the behaviour of dying out fast or surviving with high probability and \eqref{ConjectureFC} guarantees that the infection couples sufficiently fast.

In this context, the novelty of our approach lies in accommodating a dynamical random environment and requiring Assumption~\ref{Ass:WS-Monotonicity} instead of monotonicity. Moreover, to show that $U$ does not depend on the initial configuration $\xi$ and to establish the shape theorem for the case $\xi\neq\szero$, a new proof strategy is required, which differs from the existing one.

Remarkably, the shape theorem can hold without worst-case monotonicity: our next result covers CPIUs, which need not be monotone and are in some cases not even worst-case monotone, as we will see in one example in Section~\ref{subsection:examples} below.
\begin{theorem}\label{thm:ShapeTheoremNonMonoton}
     Let $(\bfeta,\bfxi)$ be a supercritical CPIU, i.e.\ $\sup_\xi\IP(\tau^{\xi}=\infty)>0$, and suppose \eqref{ConjectureAML}--\eqref{ConjectureFC} are satisfied, then \eqref{desired_result} holds true.
\end{theorem}
For the rest of this section we restrict ourselves to $S=[N]$ for some $N\in \IN$. Having established the asymptotic shape theorem for a broad class of models, we move on to verify that the sufficient conditions \eqref{ConjectureAML}--\eqref{ConjectureFC} are satisfied for several choices of the background process $\bfxi$ and the rate functions $\lambda$ and $r$.

First, we consider our leading example CPIU. For this process, we can directly verify \eqref{ConjectureAML}--\eqref{ConjectureFC}. Unfortunately, in order to do so, we additionally require monotonicity of the joint process $(\bfeta,\bfxi)$. Thus, we assume that $\bfxi$ is monotone and choose $\lambda$ and $r$ such that $(\bfeta,\bfxi)$ is monotone. We can then show that, conditioned on survival, the asymptotic shape always exists.
\begin{theorem}\label{thm:CPIUShapeTheorem}
    Let $(\bfeta,\bfxi)$ be a monotone supercritical CPIU, then \eqref{desired_result} holds.
\end{theorem}
This result can be generalised to a worst-case monotone CPDRE or a non-monotone CPIU in the following way.
\begin{proposition}\label{prop:CPDREShapeTheorem}
     Let $(\bfeta,\bfxi)$ be a worst-case monotone CPDRE or a CPIU. If $(\bfeta,\bfxi)$ can be coupled with a supercritical monotone CPIU $(\underline{\bfeta},\underline{\bfxi})$ such that $(\underline{\bfeta}_t,\underline{\bfxi}_t)\leq (\bfeta_t,\bfxi_t)$ holds for all $t\geq 0$, then \eqref{desired_result} follows.
\end{proposition}
Unfortunately, we can only prove that there exists an asymptotic shape for a CPIU $(\bfeta,\bfxi)$ if we can couple it from below with a supercritical monotone CPIU $(\underline{\bfeta},\underline{\bfxi})$ as in the above proposition. However, if the rate functions of $(\bfeta,\bfxi)$ are $\lambda(\cdot)$ and $r(\cdot)$, with $r(\cdot)$ increasing, and a CPIU  $(\underline{\bfeta},\underline{\bfxi})$ with infection rates
\begin{equation}\label{eq:coupling-rates-for-lower-bound}
    \underline{\lambda}_{(x,y)}(\xi)=\min_{\bfi\in \loc^3}\lambda(\bfi)\quad\text{for all }\xi~\text{ and }~(x,y)\in \vec{E}
\end{equation}
and the same recovery rates as $(\bfeta,\bfxi)$ is supercritical, then this coupling trivially exists. If $r(\cdot)$ is not increasing, one can simply relabel the background states to ensure monotonicity.
\begin{remark}
It is not trivial to determine whether a CPIU is indeed supercritical for a given parameter configuration. Thus, if one is interested in determining a concrete parameter regime where an asymptotic shape exists, one can couple the CPIU with a basic contact process. Apart from the trivial coupling, one can use the coupling techniques given in \cite[Theorem~3.7]{blath2023switching} or \cite[Proposition~2.2]{linker2020contact}, which are based on the techniques developed by Broman in \cite{Broman2007domination}. Of course, one cannot hope to obtain the full parameter regime with this kind of approach.
\end{remark}
Next we apply our results in the situation where $\bfxi$ is a monotone spin system on $\{0,1\}^{V\cup E}$ of finite range $L\geq 0$, as described in Example~\ref{ex:MainExamples}~(2).
In order to conclude that \eqref{desired_result} holds we will need a comparison argument with a (monotone) CPIU on $\{0,1\}^{V\cup E}$. Therefore, let
\begin{equation*}
    \underline{\alpha}_V:=\min_{\xi:\xi(x)=0}q(x,\xi)\quad\text{and}\quad \underline{\beta}_V:=\max_{\xi:\xi(x)=1}q(x,\xi)
\end{equation*}
be the minimal up-flip and the maximal down-flip rates for vertices and accordingly $\underline{\alpha}_E$ and $\underline{\beta}_E$ the minimal up-flip and maximal down-flip rates for edges.
\begin{corollary}\label{cor:spin_system}
	Let $N=1$. Suppose $(\bfeta,\bfxi)$ is a monotone CPDRE and $\bfxi$ a spin system with finite range $L\geq 0$ and 
    flip rate $q(\cdot,\cdot)$. Let $(\underline{\bfeta},\underline{\bfxi})$ denote a CPIU with the same infection and recovery rate functions as $(\bfeta,\bfxi)$ and  rates $\underline{\alpha}_V,\underline{\alpha}_E, \underline{\beta}_V,\underline{\beta}_E>0$ for the background process. If $\IP\big(\underline{\bfeta_t}^{\origin,\szero}\neq \emptyset~\forall t\geq 0\big)>0$, then $(\bfeta,\bfxi)$ satisfies the conditions \eqref{ConjectureAML}--\eqref{ConjectureFC} of Theorem~\ref{Conjecture1} and \eqref{desired_result} holds.
    \end{corollary}
\begin{proof}
    By definition $(\underline{\bfeta},\underline{\bfxi})$ is monotone and supercritical. Moreover, $(\bfeta,\bfxi)$ satisfies condition \eqref{ConjectureTEC}. Therefore, by Proposition~\ref{prop:CPDREShapeTheorem} it suffices to  find a coupling such that $(\bfeta,\bfxi)$ dominates  $(\underline{\bfeta},\underline{\bfxi})$ to prove \eqref{desired_result}. This coupling, however, is straightforward and given in \cite[Proposition~2.10]{seiler2023contact}.
\end{proof}
\subsection{Application to examples}\label{subsection:examples}
In this subsection, we apply the results provided in the previous subsection to several examples, including processes that have been studied in the literature. We consider five examples and provide parameter regimes for which \eqref{desired_result} holds, i.e.\ there exists a bounded and convex subset $U\subset \IR^d$ such that, for every $\varepsilon>0$ and every $\xi\in [N]^{V\cup E}$,
\begin{equation*}
\IP\big(\exists s\geq0 : t(1-\varepsilon)U\subset (\PCI^{\xi}_t\cap \bfH_t^\xi)\subset\bfH_t^{\xi}\subset t(1+\varepsilon)U ,,\forall t\geq s\big| \tau^{\xi}=\infty\big)=1.
\end{equation*}

\begin{remark}
The examples we consider have either edge- or vertex-dependent rates $\lambda$ and $r$, but not both at the same time. Thus, for every example, we can reduce the state space of the background process $\bfxi$ to either $[N]^V$ or $[N]^E$.
\end{remark}

\paragraph{Dynamical percolation:} The so-called \textit{contact process on a dynamical percolation} (CPDP) was first introduced by Linker and Remenik in \cite{linker2020contact}. Let $N=1$ and choose $\lambda(\cdot)$ and $r(\cdot)$ as in \eqref{eq:CPDP-rates-functions}. The background process $\bfxi$ currently in state $\xi$ has, for every $e\in E$, the transitions
    \begin{align*}
        \begin{aligned}
            \xi(e)&\to 1	\quad \text{ at rate } \alpha_E,\\
            \xi(e)&\to 0 	\quad\text{ at rate }\beta_E,
        \end{aligned}
    \end{align*}
    where $\alpha_E, \beta_E> 0$. The CPDP is monotone for every parameter choice and $\bfxi$ satisfies Assumption~\ref{AssumptionBackground}. Thus, as a direct consequence of Theorem~\ref{thm:CPIUShapeTheorem}, \eqref{desired_result} holds throughout the entire supercritical parameter regime.
\paragraph{Ising background:}
    The second example is a \textit{contact process on a ferromagnetic Ising model}, i.e. we assume again $N=1$ and choose $\lambda(\cdot)$ and $r(\cdot)$ as in \eqref{eq:CPDP-rates-functions}. Here, $\bfxi$ is a \emph{ferromagnetic stochastic Ising model} on the edge set $E$ with inverse temperature $J$. Denote by
    \begin{equation*}
        m_{e,i}(\xi):=|\{e': |e'\cap e|=1, \xi(e')=i\}|
    \end{equation*}
    the number of adjacent edges in state $i\in\{0,1\}$. Then the transitions of $\bfxi$ are
		\begin{align*}
		\xi(e)&\to 1	\quad \text{ at rate } \nu\cdot \Big(1-\tanh\big(2J\big( 2d- m_{e,1}(\xi)-1\big)\big)\Big) \text{ and }\\
		\xi(e)&\to 0	\quad \,\,\,\,\text{ at rate } \nu\cdot \Big(1-\tanh\big(2J\big( 2d- m_{e,0}(\xi)-1\big)\big)\Big),
		\end{align*}
	where $0<J<\tfrac{1}{4}\log\big(\frac{d}{d-1})$ and $\nu>0$. This process is monotone and satisfies Assumption~\ref{AssumptionBackground}, as already remarked in \cite[Example~1.2]{seiler2023contact}. Note that $\nu$ only acts as a time rescaling of the background process.

    Thus, using Corollary~\ref{cor:spin_system}, we can verify the existence of an asymptotic shape, i.e.\ \eqref{desired_result}, as long as the CPDP with the same infection and recovery rates $\lambda$ and $r$ and background jump rates
    \begin{equation*}
        \underline{\alpha}_V=\nu \Big(1-\tanh\big(2J\big( 2d-1\big)\big)\Big)\,\,\text{ and } \,\, \underline{\beta}_V=\nu \Big(1-\tanh\big(-2J\big( 2d-1\big)\big)\Big)
    \end{equation*}
    is supercritical. Using \cite[Proposition 2.2]{linker2020contact} we obtain a rate criterion that ensures that the CPDP is supercritical, which is the case if $\alpha_V>\lambda_c(d)$ and
    \begin{equation*}
        \lambda\geq \lambda_c(d)\bigg(\frac{2\nu-\lambda_c(d)}{\alpha_V-\lambda_c(d)}\bigg),
    \end{equation*}
    where $\lambda_c(d)$ is the critical value of the basic contact process on $\IZ^d$ and we used that $\tanh(x)+\tanh(-x)=0$.

\paragraph{Monotone vertex switching:} The last monotone example we consider is the \textit{contact process with (independent) vertex switching} (CPS) as introduced by Blath, Hermann and Reitmeier \cite{blath2023switching}. Again, let $N=1$, but choose $\lambda(\cdot)$ and $r(\cdot)$ as in \eqref{eq:CPS-rates-functions}. Additionally, we assume that
    \begin{equation*}
        \lambda_{11}\geq \max\{\lambda_{10}, \lambda_{01}\}\geq \min\{\lambda_{10}, \lambda_{01}\} \geq\lambda_{00}\geq 0 \quad \text{ and } \quad  r_1\leq r_0.
    \end{equation*}
    The background process $\bfxi$ currently in state $\xi$ has, for every $x\in  V$, the transitions
    \begin{align*}
        \begin{aligned}
            \xi(x)&\to 1	\quad \text{ at rate } \alpha_V,\\
            \xi(x)&\to 0 	\quad\text{ at rate }\beta_V,
        \end{aligned}
    \end{align*}
    where $\alpha_V, \beta_V> 0$. This yields a monotone process. Thus, we can directly apply Theorem~\ref{thm:CPIUShapeTheorem} to show the existence of an asymptotic shape, i.e.\ \eqref{desired_result}, throughout the entire supercritical regime.

\paragraph{Non-monotone vertex switching:} 
    We now consider a non-monotone CPS by assuming that
    \begin{equation*}
        \lambda_{11}=\lambda_{10}=:\lambda_1 \geq \lambda_0:= \lambda_{01}=\lambda_{00} \quad \text{ and } \quad  r_1\geq r_0.
    \end{equation*}
    In this case, a vertex in state $1$ is more infectious but also has a higher recovery rate, than a vertex in state $0$. It is not difficult to check that this process is not a monotone Markov process. Thus, in this situation, neither Theorem~\ref{thm:CPIUShapeTheorem} nor Corollary~\ref{cor:spin_system} is applicable. We would like to make use of Proposition~\ref{prop:CPDREShapeTheorem} and therefore need to couple the process with a monotone CPS $(\underline{\bfeta},\underline{\bfxi})$. This coupling can be obtained by choosing $\underline{\lambda}(\cdot)$ as in \eqref{eq:coupling-rates-for-lower-bound}, i.e.\ $\underline{\lambda}_{ij}= \lambda_{\min}$ for all $i,j\in\{0,1\}$. Note that $(\underline{\bfeta},\underline{\bfxi})$ is supercritical if
    \begin{equation*}
        \lambda_{\min}> \lambda_c(d)r_{\max},
    \end{equation*}
    which follows immediately by coupling with a basic CP with infection rate $\lambda_{\min}$ and recovery rate $r_{\max}$.

\paragraph{Worst-case monotone example:} 
    In this last example, we consider a CPDRE that is not monotone with respect to the total order $\leq$ but is worst-case monotone. Furthermore, the background $\bfxi$ has spatial dependence, so that we need the worst-case monotonicity assumption in order to apply our main result.

    We consider $N=2$ and $d=1$. Thus, the graph is $\IZ$, and the background state of a vertex can take one of three possible values.
    We assume that $\lambda$ and $r$ are chosen as in \eqref{eq:CPS-rates-functions}, where we additionally impose that the infection rates satisfy $\lambda_{ji}:=\lambda_{j}$ for all $i,j\in\{0,1,2\}$, i.e.\ the infection rate is determined by the state of the infected vertex. Furthermore, we assume that
    \begin{equation*}
        \lambda_1\geq \lambda_2\geq \lambda_0\quad \text{ and } \quad  r_0\geq r_1\geq r_2.
    \end{equation*}
    This means that state $1$ is more infectious than state $2$, but it also has a higher recovery rate than state $2$, so neither state is more advantageous for the infection than the other. On the other hand, state $0$ is clearly the worst state from the perspective of the infection.

    We define for any $\xi\in \{0,1,2\}^{V}$
    \begin{equation*}
        n_{i,x}(\xi):=|\{ y: ||y-x||_1=1, \xi(y)=i\}|\quad\text{with }i\in\{1,2\},~x\in V.
    \end{equation*}
    If $\bfxi$ is in state $\xi$ it has the following transitions
    \begin{align*}
        \begin{aligned}
            \xi(x)&\to 2	\quad \text{ at rate } \gamma e^{Jn_{2,x}(\xi)}\big(\1_{\{\xi(x)=0\}} + e^{2J}\1_{\{\xi(x)=1\}}\big),\\
            \xi(x)&\to 1	\quad \text{ at rate } \gamma e^{Jn_{1,x}(\xi)}\big(\1_{\{\xi(x)=0\}}+ e^{2J}\1_{\{\xi(x)=2\}}\big),\\
            \xi(x)&\to 0 	\quad\text{ at rate }  \gamma.
        \end{aligned}
    \end{align*}
    where $\gamma\geq 0$, $J\in[0,0.3]$. Now we need to verify that this model actually satisfies Assumptions \ref{AssumptionBackground} and \ref{Ass:WS-Monotonicity}:
    \begin{itemize}
        \item Reversibility can be checked by guessing the non-normalised equilibrium and then verifying the detailed balance equations. In this case, the non-normalised weights are $\kappa_0(\xi)=1, \kappa_1(\xi)=e^{Jn_1(\xi)}$ and $\kappa_2(\xi)=e^{Jn_2(\xi)}$.
        \item Ergodicity as well as Assumption \ref{AssumptionBackground}(ii) follow by an application of \cite[Chapter 1, Lemma 4.1]{liggett1985interacting} and \cite[Remark~1.3]{seiler2023contact}. These results state that the process is ergodic and satisfies Assumption \ref{AssumptionBackground}(ii) if $1>e^{3J}(e^J-1)$, which is satisfied for $J$ small enough, in particular for $J\leq 0.3$.
        \item Before we verify that $(\bfeta,\bfxi)$ is worst-case monotone, we want to mention that $\bfxi$ is not monotone with respect to the total order $\leq$. However, it is monotone with respect to the weaker partial order $\preceq$, where $0\preceq1$, $0\preceq2$ and $1$ and $2$ are incomparable. Furthermore, the all-zero configuration $\sizero$ satisfies $\sizero\preceq \xi$ for all $\xi$, and thus Proposition~\ref{prop:Monotonicity} implies that $(\bfeta,\bfxi)$ is worst-case monotone.
    \end{itemize}
   
   Finally, we can couple the process $(\bfeta,\bfxi)$ with a CPIU $(\underline{\bfeta},\underline{\bfxi})$, where $\underline{\lambda}_{ji}:=\underline{\lambda}_{j}$,
    \begin{equation*}
        \underline{\lambda}_1= \underline{\lambda}_2:=\min\{\lambda_2,\lambda_1\},\quad \underline{\lambda}_0:= \lambda_0,\quad  \underline{r}_0:= r_0 \text{ and } \quad  \underline{r}_1=\underline{r}_2:=\max\{r_1,r_2\}.
    \end{equation*}
    Moreover, $\underline{\bfxi}$ has transitions
     \begin{align*}
        \begin{aligned}
            \xi(x)&\to i\in \{0,1,2\}	\quad\text{ at rate }  \gamma.
        \end{aligned}
    \end{align*}
    Here, we use the fact that $J=0$ minimises the rate of leaving state $0$. Since the infection process $\underline{\bfeta}$ does not distinguish between the background states $1$ and $2$, it has the same survival probability as a monotone CPS with $N=1$, where $\alpha_V=2\gamma$, $\beta_V=\gamma$ and the infection and recovery rates are chosen accordingly. Thus, if this CPS is supercritical, we can use Proposition~\ref{prop:CPDREShapeTheorem} to obtain an asymptotic shape for $(\bfeta,\bfxi)$.

\subsection{Discussion and future research directions:}
As we already pointed out, worst-case monotonicity turns out to be sufficient, but not necessary, to obtain the asymptotic shape. Thus, it would be interesting to determine the optimal condition required to prove Theorem~\ref{Conjecture1}. The worst-case monotonicity is motivated primarily by technical considerations, or more precisely by the proof technique.
However, it appears to us that in the setup of the CPDRE a (deterministic) asymptotic shape should exist if the effect of the initial state of the background process on the long-time behaviour is negligible.
This is exactly the idea of Assumption~\ref{AssumptionBackground}(ii), which can be read as a fast-mixing assumption on $\bfxi$. This effect can be seen rather explicitly in the proof of \cite[Theorem~2.1]{seiler2023contact}, where it is shown that, under this assumption, the influence of the initial state of $\bfxi$ has no impact on the phase transition. This principle is also implicitly used in the proof of Theorem~\ref{thm:ShapeTheoremNonMonoton}. However, it is not clear how one can use this heuristic idea directly in the restarting procedure, which is a pathwise argument.

The process $(\bfeta, \bfxi)$ is, in general, not a monotone Markov process. However, the infection process $\bfeta$ itself is always additive (see Proposition~\ref{prop:Monotonicity}). Consequently, one might be inclined to believe that the techniques developed in \cite{bezuidenhout1990contact} could be extended to the process $(\bfeta, \bfxi)$ if vertex and edge updates are independent. Unfortunately, in Section~\ref{sec:BlockConstruction}, we were able to adapt these techniques only under the assumption that the entire system is monotone, as some of the technical results rely on this property. To be precise, we need the block events defined in Section~\ref{sec:BlockConstruction} to be positively correlated, which is implied by monotonicity of the whole system. Nevertheless, we have the impression that these blocks should also be positively correlated for several non-monotone CPDREs such as the CPS. Thus, these results should still hold without the monotonicity assumption, at least in some cases. To prove this, it would be necessary to show that the system is almost monotone or asymptotically monotone on macroscopic scales, which ensures that the good blocks are positively correlated. However, the precise meaning of this remains somewhat unclear and still needs to be rigorously defined.

A natural next question would be to study the effect of the evolving environment on the expansion speed of the infection area. As a first step, one could analyse the behaviour of a CPDP, where vertex and edge updates are independent of each other. In this case, it is more convenient to consider
\begin{equation*}
    \alpha_V=v_Vp_V,\,\, \beta_V=v_V(1-p_V),\,\, \alpha_E=v_Ep_E\,\, \text{ and }\,\, \beta_E= v_E(1-p_E)
\end{equation*}
where $v_V,v_E>0$ and $p_V,p_E\in (0,1)$. In this parametrisation, the constants $v_V,v_E$ indicate the update speed of vertices and edges, respectively, while $p_V$ and $p_E$ are the opening probabilities. Now two similar questions arise. First, how does the expansion speed vary for different infection rates $\lambda$ and opening probabilities $p_V$ and $ p_E$ if $\lambda p_V$ and $\lambda p_E$ remain constant? Second, how is the expansion speed affected by a faster update speed, i.e. as one increases $v_V$ or $v_E$, while keeping all other parameters constant?\\

\textbf{Outline of the article.} The remaining paper is structured as follows.  In Section~\ref{sec:Construction}, an explicit Poisson construction for the CPDRE and a graphical interpretation are given. This is followed by Section~\ref{Sec:Preliminaries}, where we briefly discuss and state some results proven in \cite{seiler2023contact} and \cite{blath2023switching}, which we will need in the proof sections. Section~\ref{sec:RestartingAndEssentialHitting} is devoted to the proofs of Theorem~\ref{Conjecture1} and Theorem~\ref{thm:ShapeTheoremNonMonoton}. Section~\ref{sec:ShowingAsympShapeCPDRE} is dedicated to showing Theorem~\ref{thm:CPIUShapeTheorem} and Proposition~\ref{prop:CPDREShapeTheorem},
i.e.\ verifying the conditions \eqref{ConjectureAML}--\eqref{ConjectureFC} for the respective situations.

\section{Construction of the Process}\label{sec:Construction}

In this section, we provide an explicit construction of the CPDRE. This is done via a so-called random mapping representation, which is a Poissonian construction of a jump process on a general product space $S^\Lambda$ and was introduced by Sturm and Swart in \cite{sturm2018monotone}. Here $\Lambda$ is usually called the \textit{lattice} and $S$ the \textit{local state space}. See \cite[Section~2\&4]{swart2022course} for all details of this type of Poisson construction and especially Section~4 for the case where $\Lambda$ is given by an infinite graph $G$. Note that this random mapping representation is closely related to the well-known graphical representation discussed in Liggett's book \cite{liggett1999stochastic} and is, in fact,  an equivalent construction for some additive systems, see \cite[Subsection~6.2\&6.3]{swart2022course}.

To carry out the construction we first choose a countable set $\cM$ of maps $m:S^\Lambda \to S^\Lambda$, which describe every possible transformation of the system and corresponding jump rates $(h_m)_{m\in \cM}$. Then the process is defined via a Poisson point process $\Delta$ on $\cM\times \IR$ with intensity measure $h_m\mathsf{d}t$ in the way described below. Heuristically, one chooses an initial configuration and then orders all Poisson points $(m,s)$ according to the arrival times, i.e.\ the second component. Then all maps $m$ are applied successively until a given time $t$.

To make the construction rigorous we introduce some further notation.
For any map $m:S^{\Lambda}\to S^{\Lambda}$ and $a \in \Lambda$ we define the map $m[a]: S^{\Lambda}\to S$ by $m[a](\zeta):= m(\zeta)(a)$.  Further, we denote by
\begin{equation*}
    \cD(m):=
        \{a\in \Lambda :  \exists \zeta \text{ s.t. } m[a] (\zeta) \neq \zeta(a)\},
\end{equation*}
the set which contains all lattice points, whose state could possibly be changed by $m$. 

Furthermore, for $a\in \Lambda$  we call a point $a_1\in \Lambda$ \textit{$m[a]$-relevant} if there exist $ \zeta_1,\zeta_2\in S^{\Lambda}$ such that $m[a](\zeta_1)\neq m[a](\zeta_2)$ and  $\zeta_1(a_2)=\zeta_2(a_2)$ for all $a_2\neq a_1$. Denote the set of all relevant lattice points for the map $m[a]$ by
\begin{equation*}
    \cR(m[a]):=\{ a'\in  \Lambda : a' \text{ is } m[a]\text{-relevant}\}.
\end{equation*}
In order to guarantee  that the construction is well-defined and the resulting process is a Feller process, we impose the following assumptions on $\cM$ and $(h_m)_{m\in \cM}$:
\begin{enumerate}
    \item $|\cD(m)|<\infty$ for all $m\in \cM$,
    \item $m$ is a continuous map with respect to the product topology,
    \item the rates satisfy \begin{equation}\label{totalratebound}
    \sup_{a\in \Lambda}\sum_{m\in \cM:\, \cD(m)\ni a }h_m(|\cR(m[a])|+1)<\infty.
\end{equation}
\end{enumerate}
Note that continuity of $m$ implies that $|\cR(m[a])|<\infty$ for all $a\in S$, see \cite[Lemma~4.13]{swart2022course}. 
In \cite{swart2022course} it is shown that if \eqref{totalratebound} is satisfied, then the constructed  process is a Feller process $\bfzeta=(\bfzeta_t)_{t\geq 0}$ with state space $S^{\Lambda}$, which has a generator of the form
\begin{equation*}
    \cA f(\zeta)=\sum_{m\in \cM} h_m\big(f(m(\zeta))-f(\zeta)\big).
\end{equation*}
Before we begin with the actual construction of the CPDRE, we will introduce a concept that is, under some circumstances, slightly stronger than the monotonicity of the Feller process $\bfzeta$, or rather monotonicity of its associated semigroup.		
\begin{definition}
	We say that the process $\bfzeta$ is \textit{monotonically representable} with respect to the partial order $\preceq$ on $\loc$ if $\bfzeta$ can be constructed using some set $\mathcal{M}$ of monotone maps, that is, for all $m \in \mathcal{M}$ we have
	\begin{equation*}
	    m(\zeta_1) \preceq m(\zeta_2) \quad \text{for all } \zeta_1, \zeta_2 \in \loc^{V \cup E} \text{ with } \zeta_1 \preceq \zeta_2.
	\end{equation*}
	Here, $\preceq$ denotes the componentwise order on $\loc^{V \cup E}$ induced by the partial order $\preceq$.
\end{definition}
If the process $\bfzeta$ is monotonically representable, then $\bfzeta$ is also a monotone Feller process (see \cite[Lemma~5.3]{swart2022course}) and our mapping construction implies that  $\bfxi_t^{\xi}\preceq \bfxi_t^{\xi'}$ holds for all $t\geq 0$ whenever $\xi\preceq \xi'$. Note that, in general, not every monotone Markov process is monotonically representable. However, for processes on a finite and totally ordered state space the two concepts are in fact equivalent.
\subsection{Assumptions on the Background Process}\label{sec:AssBackground}
We assume that there exists an explicit random mapping representation of the background process $(\bfxi_t)_{t\geq 0}$ on $\loc^{V\cup E}$ and denote the countable set of maps  $m:\loc^{V\cup E} \to \loc^{V\cup E}$ by $\cM_{BG}^*$. The Poisson point process to construct the process is denoted by $\Delta^{BG}$ on $\cM_{BG}^*\times \IR$ with intensity measure $h_m\mathsf{d}t$, where $(h_m)_{m\in \cM_{BG}^*}$ are the corresponding jump rates, which satisfy \eqref{totalratebound}. The corresponding generator is denoted by
\begin{align}\label{eq:generator_bg}
    \cA_{\textnormal{BG}} f(\zeta)=\sum_{m\in \cM} h_m\big(f(m(\zeta))-f(\zeta)\big).
\end{align}
Let $R_i$ be the reflection at the $i$th coordinate, that is
\begin{equation*}
    R_i(x)=(x_1\dots,-x_i,\dots,x_d)  \text{ and } R_i(\{x,y\})=\{R_i(x),R_i(y)\}
\end{equation*}
for any  $i\in \{1,\dots,d\}$, $x\in V$ and every $\{x,y\}\in E$. We denote by $T_z$ with $z\in V$ the spatial shift on the lattice $\IZ^d$, i.e.\ $T_z(x)=x-z$ for $z\in V$ and $T_z(\{x,y\})=\{x-z,y-z\}$ for $\{x,y\}\in E$. Now the reflection $R_i$ and the spatial shift $T_z$ act on $\xi$ such that $R_i\xi(a)=R_i(\xi(a))$ and $T_z\xi(a)=\xi(T_z(a))$ for every $a\in V\cup E$. Moreover, on maps $m\in \cM$ these operators are defined as 
\begin{equation*}
    R_im(a)(\xi)=m(a)(R_{i}\xi) \quad \text{and} \quad  T_zm(a)(\xi)=m(a)(T_{z}\xi).
\end{equation*}
We impose some assumptions on the maps $m\in\cM_{BG}^*$ that ensure that the constructed process is indeed monotonically representable, translation invariant, and of finite range. Recall that $\preceq$ is the component-wise total order on $\loc^{V\cup E}$.
\begin{assumption}\label{Ass:BackgroundMaps}
    We assume that $\cM_{BG}^*$ and $(h_m)_{m\in \cM_{BG}^*}$ satisfies the following properties:
    \begin{enumerate}
	   \item For every map $m\in \cM_{BG}^*$ and every reflection $R_i$ or translation $T_z$, there exists a map $m'\in \cM_{BG}^*$ such that $R_im=m'$ or respectively $T_zm=m'$ and $h_m=h_{m'}$.
	   \item There exists $L\geq 0$ such that for all $m\in \cM_{BG}^*$ there exists $x\in V$ and
       \begin{equation*}
           \cD(m)\cup \bigcup_{a\in V\cup E}\cR(m[a])\subset \B_L(x)\cup \B^{E}_{L}(x).
       \end{equation*}
    \end{enumerate}    
\end{assumption}
\begin{remark}\label{remark:bg_finite_range}
    Note that the first assumption guarantees that $\bfxi$ is \textit{symmetric} and \textit{translation invariant}. The second implies that changes of $\bfxi$ are of \textit{finite range}, i.e.\ there exists an $L\geq 0$ such that a change in the state of $a\in V\cup E$ depends only on the states $b\in V\cup E$ which are less than distance $L$ away. In the context of edges, this means that the vertices at both ends need to be closer than $L$.
\end{remark}

Before we proceed, we will show that the monotone examples which we consider as background processes $\bfxi$ (see Example~\ref{ex:MainExamples}) are indeed monotonically representable.

\begin{example}
Let us consider our main examples stated in Example~\ref{ex:MainExamples} for $S=[N]$ equipped with the canonical total order $\leq$.
\begin{enumerate}
    \item If $\bfxi=(\bfxi(a))_{a\in V\cup E}$ is a family of independent monotone Markov processes on $[N]^{V\cup E}$, then \cite[Proposition~12]{sturm2018monotone} shows that each $\bfxi(a)$ is monotonically representable, since $[N]$ is totally ordered. One can then extend the maps $m:[N]\to [N]$ in these Poisson constructions to maps on $[N]^{V\cup E}$.
    \item In the special case $N=1$, we can straightforwardly state this construction. Let us define the two maps $\mathbf{up}_{a}$ and $\mathbf{down}_{a}$ for $a\in V\cup E$, by setting
    \begin{align*}
        \mathbf{up}_{a}(\xi)(a'):=
        \begin{cases}
            1& \text{ if } a=a'\\
            \xi(a_1)& \text{ otherwise},
        \end{cases}\quad \text{ and }\quad 
        \mathbf{down}_{a}(\xi)(a'):=
        \begin{cases}
            0& \text{ if } a=a'\\
            \xi(a_1)& \text{ otherwise},
        \end{cases}
    \end{align*}
    for all $\xi\in \{0,1\}^{V\cup E}$ and $a'\in V\cup E$. Now we construct our leading example, the CPDP, by choosing
    \begin{equation*}
        \cM_{BG}^*:=\{\mathbf{up}_{a}: a\in V\cup E\}\cup \{\mathbf{down}_{a}: a\in V\cup E\}.
    \end{equation*}
    Furthermore, we set $h_{\mathbf{up}_{x}}=\alpha_V$ and $h_{\mathbf{down}_{x}}=\beta_V$ for all $x\in V$ and  $h_{\mathbf{up}_{e}}=\alpha_E$ and $h_{\mathbf{down}_{e}}=\beta_E$ for all $e\in E$. By plugging in these maps and rates, one can see that the generator \eqref{eq:generator_bg} corresponds to a system of independent vertex and edge updates with the correct transition rates.
\item A random mapping representation for a nearest-neighbour ferromagnetic Ising model can be found in \cite[Section~4.6]{swart2022course}. For nearest-neighbour spin systems, this construction can be adapted if they are monotone and translation invariant.
\end{enumerate}

\end{example}

\subsection{Random Mapping Representation of the CPDRE}\label{subsec:RandomMapping}
We will now construct the CPDRE $(\bfeta,\bfxi)$. Therefore we need to extend the random mapping representation of $(\bfxi_t)_{t\geq 0}$. Let us first extend the maps given in $\cM_{BG}^*$ in a natural way, that is, for $m\in \cM_{BG}^*$ define the map $m^B:\{0,1\}^V\times \loc^{V \cup E}\to \{0,1\}^V\times \loc^{V\cup E}$ as $m^B(\eta,\xi)=(\eta,m(\xi))$. We denote by $\cM_{BG}:=\{m^B: m\in \cM_{BG}^*\}$ the collection of these maps and set the corresponding rates to be $h_{m^B}:=h_{m}$. 

Suppose $|\loc|=N$, let $\{\bfi_k\}_{1\leq k\leq N^3}$ be an enumeration of the elements in $\loc^3$ which is increasing with respect to the infection rate $\lambda$, i.e.\ $\{\bfi_k\}_{1\leq k\leq N^3}=\loc^3$ and $\lambda(\bfi_{k-1})\leq \lambda(\bfi_{k})$ for all $k\leq N^3$. Moreover, let $F:\loc^3\to \{1,\dots,N^3\}$ be the function defined by $F(\bfi_k)=k$.
Now we define maps $\mathbf{inf}^{*}_{\bfi,(x,y)}:\{0,1\}^V\times \loc^{V\cup E}\to \{0,1\}^V$, where $\bfi\in \loc^3$  and $(x,y)\in \vec{E}$ by setting
\begin{align*}
\mathbf{inf}^{*}_{\bfi,(x,y)}(\eta,\xi)(z)&:=
\begin{cases}
1& \text{if } F\big((\xi(x),\xi(\{x,y\}),\xi(y))\big)
\geq F(\bfi), \eta(x)=1 \text{ and } y=z, \\
\eta(z)& \text{otherwise},
\end{cases}
\end{align*}
for every $z\in V$. In words, the map $\mathbf{inf}^{*}_{\bfi,(x,y)}$ changes the configuration $\eta$ at site $y$ to $1$, if the neighbouring site $x$ is infected, i.e.\ $\eta(x)=1$, and the background at $(x,\{x,y\},y)$ is in a state greater or equal than $\bfi$ according to the order given by $F$.
Analogously, let $\{\bfj_k\}_{k\leq N}=S$ be an enumeration such that $r(\bfj_{k-1})\geq r(\bfj_{k})$ for all $ k \leq N$ and $G:S\to \{1,\dots,N\}$ be the function defined by $G(\bfj_k)=k$. For $\bfj\in S$ and $x\in V$ we define $\mathbf{rec}^{*}_{\bfj,x}:\{0,1\}^V\times \loc^{V\cup E}\to \{0,1\}^V$ by
\begin{align*}
\mathbf{rec}^{*}_{\bfj,x}(\eta,\xi)(z)&:=\begin{cases}
0& \text{if }  G(\xi(x))\leq G(\bfj) \text{ and } z=x, \\
\eta(z)& \text{otherwise}.
\end{cases}
\end{align*}

We again extend the maps to the full state space $\{0,1\}^V\times \loc^{V\cup E}$ by setting 
\begin{align*}
    \mathbf{inf}_{\bfi,(x,y)}(\eta,\xi):=(\mathbf{inf}^{*}_{\bfi,(x,y)}(\eta,\xi),\xi) \quad \text{ and }\quad \mathbf{rec}_{\bfj,x}(\eta,\xi):=(\mathbf{rec}^{*}_{\bfj,x}(\eta,\xi),\xi).
\end{align*}
Given our infection and recovery maps we define the corresponding rates as
\begin{equation*}
    h_{\mathbf{inf}_{\bfi_1,(x,y)}}:=\lambda(\bfi_1)\quad\text{and}\quad h_{\mathbf{inf}_{\bfi_{k},(x,y)}}:=\lambda(\bfi_k)-\lambda(\bfi_{k-1})\quad\text{for }k\geq 2,
\end{equation*}
as well as
\begin{equation*}
    h_{\mathbf{rec}_{\bfj_{N},x}}:=r(\bfj_{N})\quad\text{and}\quad h_{\mathbf{rec}_{\bfj_{k},x}}:=r(\bfj_{k-1})-r(\bfj_{k})\quad\text{for }k\leq N.
\end{equation*}
We define
\begin{equation*}
    \cM^{\bfi}_{\inf}:=\{\mathbf{inf}_{\bfi,(x,y)}: (x,y)\in \Vec{E}\}\quad\text{and}\quad\cM^{\bfj}_{\text{rec}}:=\{\mathbf{rec}_{\bfj,x}: x\in V\}
\end{equation*}
and denote the collections of these maps by
\begin{equation*}
    \cM_{CP}:=\bigcup\limits_{\bfi\in \loc^3}\cM^{\bfi}_{\inf}\cup\bigcup\limits_{\bfj\in \loc}\cM^{\bfj}_{\text{rec}}.
\end{equation*}
Let $\Delta^{\inf,\bfi}$ and $\Delta^{\text{rec},\bfj}$ be Poisson point processes on $\cM^{\bfi}_{\inf}\times \IR$ and $\cM^{\bfj}_{\text{rec}}\times \IR$, respectively, with the corresponding intensity measures $h_{\mathbf{inf}_{\bfi,x,y}}\mathsf{d}t$ and $h_{\mathbf{rec}_{\bfj,x}}\mathsf{d}t$ and set
\begin{equation*}
    \Delta:=\bigcup_{\bfi\in \loc^3}\Delta^{\inf,\bfi}\cup \bigcup_{\bfj\in \loc} \Delta^{\text{rec},\bfj}\cup \Delta^B.
\end{equation*}
Clearly, $\Delta$ is a Poisson point process on $\cM\times \IR$ with $\cM:=\cM_{CP}\cup \cM_{BG}$. Moreover, $|\cD(m)|<\infty$ is true for all maps $m\in \cM$ and every map is also continuous with respect to the product topology. Since $\cM$ together with $(h_m)_{m\in \cM}$ satisfy \eqref{totalratebound} the random mapping representation yields a Feller process $(\bfeta,\bfxi)$ with generator
\begin{equation*}
    \cA f(\eta,\xi)=\sum_{m\in \cM_{CP}} h_m\big(f(m(\eta,\xi))-f(\eta,\xi)\big)+\cA_{\textnormal{BG}}f(\eta,\,\cdot\,)(\xi),
\end{equation*}
where $\cA_{\textnormal{BG}}$ was defined in \eqref{eq:generator_bg}. Plugging in the maps and rates shows that this process $(\bfeta,\bfxi)$ has the same transition rates as described in \eqref{InfectionRatesWithBackground}.

As a direct consequence of the random mapping representation, we get a the following monotonicity criterion. 
\begin{proposition}\label{prop:Monotonicity}
    The infection process $\bfeta$ is additive in the initial infection configuration, i.e.\ for $\eta_1,\eta_2\in\{0,1\}^V$ it holds 
    \begin{equation*}
        \bfeta_t^{\eta_1\vee\eta_2,\xi}=\bfeta_t^{\eta_1,\xi} \vee\bfeta_t^{\eta_2,\xi}\quad \text{for all }t\geq 0, ~\xi\in\loc^{V\cup E}.
    \end{equation*}
     Furthermore, if $\bfxi$ is monotonically representable and $\lambda(\,\cdot\,)$ is an increasing and $r(\,\cdot\,)$ a decreasing function with respect to the component wise canonical order $\preceq$ , then the CPDRE $(\bfeta,\bfxi)$ is monotonically representable, and thus
    \begin{equation*}
        (\bfeta_t^{\eta_1,\xi_1},\bfxi_t^{\xi_1})\preceq (\bfeta_t^{\eta_2,\xi_2},\bfxi_t^{\xi_2})\quad\text{holds for all }t\geq 0\quad\text{a.s.}~\text{ if }(\eta_1,\xi_1)\preceq(\eta_2,\xi_2).
    \end{equation*}
    In particular, if there exists $w\in \loc$ such that  $w\preceq \bfj$ for all $\bfj\in \loc$, then $(\bfeta,\bfxi)$ also satisfies Assumption~\ref{Ass:WS-Monotonicity}, i.e.\ is worst-case monotone.
\end{proposition}
\begin{proof}
    The additivity in the infection configuration is a direct consequence of the construction, since the $\mathbf{inf}(\cdot, \xi)$ and $\mathbf{rec}(\cdot, \xi)$ maps are additive for every fixed $\xi$. Note that a map $m$ is additive if $m(\eta_1\vee \eta_2)=m(\eta_1)\vee m(\eta_2)$.
    The assumptions on the rate functions imply that all $\mathbf{inf}$ and $\mathbf{rec}$ maps are monotone with respect to pointwise order $\preceq$. Now the statement follows directly from the fact that $\cM$ contains only monotone maps (see \cite[Lemma~5.3]{swart2022course}).
    Finally, by assumption, the constant sequence $\szero$ satisfies $\szero\preceq \xi$ for all $\xi\in \loc^{V\cup E}$, and thus Assumption~\ref{Ass:WS-Monotonicity} is a direct consequence of this monotone coupling.
\end{proof}

As previously mentioned, the random mapping representation is closely related to the classical graphical representation. In our case, we can also define an infection path given the background process $\bfxi$, and thus obtain a graphical interpretation for the spread of the infection. One interprets a point $(\mathbf{inf}_{\bfi,(x,y)},t)$ as an infection arrow pointing from $x$ to $y$, which is only usable if the background $\bfxi$ around $\{x,y\}$ is currently in some state $\bfi'=\bfxi_t(x,\{x,y\},y)$ with $F(\bfi')\geq F(\bfi)$. In words, every infection arrow corresponds to some type $\bfi\in \loc^3$ and can only be used if the background at $(x,\{x,y\},y)$ is in state $\bfi$ or in a higher state with respect to the order given by $F$.
Similarly, $(\mathbf{rec}_{\bfj,x},t)$ is identified as a recovery event, which is only usable if $G(\bfxi_t(x))\leq G(\bfj)$.

\begin{definition}[$\xi$-infection path]\label{def:InfectionPath}
\label{InfectionPath}
    Let  $(y, s)$ and $(x, t)$ with $s < t$ be space-time points and  $\bfxi$ the background process starting with initial configuration $\xi$. We say that there is a \textit{$\xi$-infection path} from $(y, s)$ to $(x, t)$ if there is a sequence
	of times $s = t_0 < t_1 < \dots < t_n \leq t_{n+1} = t$ and space points $y = x_0,x_1,\dots, x_n = x$ such that for all  $k \in\{ 0, \dots , n\}$ we have
	$( \mathbf{inf}_{\bfi,(x_{k},x_{k+1})},t_k)\in \Delta$ for some $\bfi$ with $F(\bfi)\leq F(\bfxi_{t_k}(x_{k},\{x_{k},x_{k+1}\},x_{k+1})) $ 	and $\Delta\cap\{ (\mathbf{rec}_{\bfj,x_k},t): G(\bfxi_t(x_k))\leq G(\bfj) , t\in[t_k , t_{k+1} )\}\big)=\emptyset$. 
    We write $(y, s)\stackrel{\xi}{\longrightarrow} (x, t)$ if there exists a $\xi$-infection path. 
\end{definition}

Now, we can also characterise the infection process $\bfeta$ with initial configuration $(\eta,\xi)$ via these infection paths. We use the background process $(\bfxi_t)_{t\geq 0}$ with $\bfxi_0=\xi$ and then set $\bfeta_0:=\eta\in \{0,1\}^\dV$ as well as
\begin{equation}\label{DefinitionCPDLP}
	\bfeta^{\eta,\xi}_t(x):=\begin{cases}
	    1 &\text{ if there exists } y\in V~\text{with}~\bfeta_0(y)=1 \text{ and } (y,0)\stackrel{\xi}{\longrightarrow} (x,t)\\
        0 &\text{ otherwise}.   
	\end{cases}
\end{equation}

We conclude this section by introducing additional notation required for Section~\ref{sec:RestartingAndEssentialHitting}, where we extensively utilise the self-similarity properties of $\IZ^d$, alongside the translation invariance and symmetry of the process. Consequently, the concepts of shifts in both spatial and temporal directions are fundamental. We abuse notation and write for $\Delta\subset \cM\times \IR$,
\begin{equation*}
    T_x(\Delta)=\{( T_x m,t): (m,t)\in \Delta \}\subset \cM\times \IR.
\end{equation*}
Analogously for a temporal shift $\theta_s$ with $s\geq 0$, i.e.\ $\theta_s(t)=t+s$, we write
\begin{equation*}
\theta_s(\Delta)=\{( m,\theta_s(t)): (m,t)\in \Delta \}\subset \cM\times \IR.
\end{equation*}
Furthermore, we set $\cF_t=\sigma(\Delta\cap [0,t])$ for every $t\geq 0$  and  $\cF:=\cF_{\infty}$.

\section{Basic Properties for the CPDRE}\label{Sec:Preliminaries}
\subsection{Independence of Criticality of the Initial State}\label{Sec:CritValPermanetlyCoupled}
One of the main results in \cite{seiler2023contact} is that the critical value $\lambda_c$ is independent of the initial configuration $(\eta,\xi)\in \{0,1\}^V\times\loc^{V\cup E}$ as long as the initial configuration for the infection process is finite, i.e.\ $|\eta|<\infty$. This result can be extended to our setting and can be proven in an analogous manner. Therefore, we do not give a proof in full detail, but provide a sketch of the arguments. Note that we will also need some of the auxiliary objects and results later on.
\begin{theorem}[{\cite[Theorem~2.1]{seiler2023contact}}]\label{thm:CritValueIndependence}
	Suppose Assumption~\ref{AssumptionBackground} (i) and (ii) are satisfied. If there exists a configuration $(\eta,\xi)\in \{0,1\}^V\times\loc^{V\cup E}$ with $|\eta|<\infty$ such that $\IP(\tau^{\eta,\xi}=\infty)>0$, then it follows that $\IP(\tau^{\eta,\xi}=\infty)>0$ for all $(\eta,\xi)\in \{0,1\}^V\times\loc^{V\cup E}$ with $|\eta|>0.$
\end{theorem}
\begin{proof}[Sketch of proof]
    First, we introduce $(\widetilde{\bfeta}_t)_{t\geq 0}$, which is a contact process with infection rate $\lambda_{\max}:=\max_{\bfi\in \loc^3}\lambda(\bfi)$ and no recoveries. In the literature this process is also called first passage percolation or Richardson model. This process is defined on the same graphical representation as the CPDRE. In the construction we use every infection point, i.e.\ we use the PPP 
    \begin{equation*}
        \widetilde{\Delta}^{\inf}_{\{x,y\}}:=\bigcup_{\bfi\in\loc^3} \Delta^{\inf,\bfi}_{(x,y)}
    \end{equation*}
    where $\{x,y\}\in \vec{E}$, and ignore all recovery events and all restrictions from the background. With this construction we couple $\widetilde{\bfeta}$ with the infection process $\bfeta$ such that if $\bfeta_0\leq \widetilde{\bfeta}_0$, then $\bfeta_t\leq \widetilde{\bfeta}_t$ holds for all $t\geq 0$. We think of $\widetilde{\bfeta}$ as a maximal infection process. 
    
    Let us assume that the finite range of $\bfxi$ is $L\geq 0$. A key object of the proof is the set
    \begin{equation}\label{eq:PermantlyCoupledNeighbourhood}
        \Phi_{t}:=\{x\in V: x,y,\{x,y\} \in \Psi'_{t} \,\text{ for all } \, y\in \B_L(x)\},
    \end{equation}
    that contains all vertices which are  permanently coupled and for which all adjacent edges, as well as all vertices within distance $L$, are also permanently coupled. The difference here compared to \cite{seiler2023contact} is that we need to incorporate that not only the background state of all adjacent edges but also all neighbouring vertices have already become permanently coupled.

    Let us briefly summarize the proof strategy. Since $|\eta|<\infty$ there exists $M>0$ such that 
    \begin{equation}\label{eq:GrowthboundMaxInf}
        \IP(\exists s\geq 0 : \widetilde{\bfeta}^{\eta}_t\subset \B_{M t} ~\forall\, t\geq s)=1,
    \end{equation}
    where $M$ does not depend on the choice of $\eta$. This is a well known result in the context of first passage percolation on $\IZ^d$. For a proof, see, for example, \cite[Lemma~5.2]{seiler2023contact}. 

    On the other hand for any $M>0$ it holds that 
    \begin{equation}\label{eq:GrowthboundPermanetlyCoupledBack}
        \IP(\exists s\geq 0 : \B_{M t}\subset \Phi_t ~\forall t\geq s)=1.
    \end{equation}
    This can be shown in almost the exact same way as \cite[Proposition~5.3]{seiler2023contact}, since the number of neighbouring vertices is obviously of the same order as of adjacent edges.
    
    Now putting \eqref{eq:GrowthboundMaxInf} and \eqref{eq:GrowthboundPermanetlyCoupledBack} together yields that 
    \begin{equation}\label{eq:InfectionContainedInCouplingRegion}
        \IP(\exists s\geq 0 :  \widetilde{\bfeta}^{\eta}_t\subset \B_{Mt} \subset \Phi_t ~\forall\, t\geq s)=1.
    \end{equation}
    In words, there exists an almost surely finite time $s\geq 0$ such that all vertices $x$ that have been infected until time $t$ will be contained in $\Phi_t$ for all $t\geq s$, regardless of the initial configuration $(\eta,\xi)$. Thus, the set of infected sites is fully contained in the permanently coupled region from $s$ onwards.

    Now the claim can be shown analogously as \cite[Proposition 5.5]{seiler2023contact}. The idea is to first consider $\bfeta_0=\delta_{\zero}$, i.e.\ that only the origin $\origin$ is initially infected. Then, one defines two auxiliary process $\underline{\bfeta}$ and $\overline{\bfeta}$ with $\underline{\bfeta}_0 =\bfeta_0=\overline{\bfeta}_0$. The process $\underline{\bfeta}$ is only allowed to spread after time $s$ and for the process $\overline{\bfeta}$ vertices can only recover after time $s$. Afterwards both processes use again the exact same graphical representation. Heuristically speaking $\underline{\bfeta}$ can be seen as the worst-case scenario for the infection until time $s$ and $\overline{\bfeta}$ represents the best case. With the help of \eqref{eq:InfectionContainedInCouplingRegion} one then can show that
    \begin{equation*}
        \IP(\underline{\bfeta}_t\neq \szero \,\forall\, t\geq 0)>0 \Leftrightarrow \IP(\overline{\bfeta}_t\neq \szero \,\forall\, t\geq 0)>0.
    \end{equation*}
    Broadly speaking, one shows that an advantage that only arises during a finite time period $[0,s]$ cannot determine whether survival is possible or not.

    This implies that if we start from one initially infected vertex, positivity of the survival probability does not depend on the starting configuration of the background $\bfxi$. 
\end{proof}
In fact, since we consider $G$ to be the $d$-dimensional integer lattice we can improve \eqref{eq:InfectionContainedInCouplingRegion} with respect to the convergence speed.
    \begin{lemma}\label{lem:InfContainedExponetialSpeed}
     Let $\eta\subset V$ be finite. Then there exist $M,A,B>0$ such that 
     \begin{equation*}
          \IP(  \widetilde{\bfeta}^\eta_t\subset\B_{Mt}\subset \Phi_t \,\forall\, t\geq s)>1- Ae^{-Bs} \quad \text{for all} \quad s\geq 0.
     \end{equation*}
    \end{lemma}
    \begin{proof}
        The proof follows exactly as in \cite[cf. 5.2, 5.3 and 5.4]{seiler2023contact} exploiting that  $G$ is the $d$-dimensional integer lattice.
    \end{proof}

\subsection{Upper Invariant Law and Duality}\label{InvariantLawAndDuality}
In this section, we follow the works of \cite{seiler2023contact} and \cite{blath2023switching}, where most results have already been proven in special cases and the proofs also apply to our context.
\begin{lemma}\label{lem:ExistenceUpperInv}
    There exists a probability measure $\nu$ such that $(\bfeta^{V,\pi}_t,\bfxi^{\pi}_t)\Rightarrow\nu$ as $t\to \infty$. Furthermore, $\nu$ is the upper invariant measure of the CPDRE $(\bfeta,\bfxi)$, i.e.\ if $\mu$ is another invariant distribution, then $\mu\preceq_d \nu$, where $\preceq_d$ denotes the stochastic order.
\end{lemma}
\begin{proof}
    This can be shown analogously as in \cite[Theorem~2.4]{blath2023switching}.
\end{proof}

We will briefly show that if the background is started stationary, i.e.\ $\bfxi_0\sim \pi$, then the CPDRE $(\bfeta_,\bfxi)$ is dual in distributional sense to a CPDRE $(\widecheck{\bfeta},\widecheck{\bfxi})$ with mirrored infection rates, i.e.\ $y$ infects $x$ with rate $\lambda_{(x,y)}(\cdot)$ instead of $\lambda_{(y,x)}(\cdot)$ and $\widecheck{\lambda}_{(x,y)}=\lambda_{(y,x)}$. Thus $\widecheck{\bfxi}$ is distributed as $\bfxi$ and, if $\widecheck{\bfxi}$ is currently in state $\xi$ the transitions of $\widecheck{\bfeta}$ currently in state $\eta$ are for all $x\in V$, 
\begin{equation}\label{InfectionRatesForDualProcess}
    \begin{aligned}
   \eta(x)&\to 1	\quad \text{ at rate } &&\sum_{{y:(y,x)\in \Vec{E}}} \lambda_{(x,y)}(\xi) \eta(y)\quad\text{ and }\\
    \eta(x)&\to 0 	\quad\text{ at rate } &&r_x(\xi).
    \end{aligned}
\end{equation}

\begin{proposition}\label{prop:DualityStationary}
	Let $\eta, \eta'\in \cP(V)$, $\bfxi$ be reversible and $\bfxi_0\sim \pi$. Then it holds that 
	\begin{equation}\label{eq:DualityStationary}
		\IP(\bfeta^{\eta,\pi}_{t}\cap \eta' \neq \emptyset)=\IP(\widecheck{\bfeta}^{\eta',\pi}_{t}\cap \eta \neq \emptyset)\quad\text{for all }t\geq 0.
	\end{equation}
    In particular if $\lambda_{(x,y)}(\cdot)=\lambda_{(y,x)}(\cdot)$ for all $(x,y)\in \vec{E}$, then \eqref{eq:DualityStationary} is a self duality relation.
\end{proposition}
\begin{example}
    Let us briefly discuss the resulting duality relations in two of our leading examples.
\begin{enumerate}
    \item For the contact process on a dynamical graph the duality relation \eqref{eq:DualityStationary} is indeed a self-duality as already shown in \cite[Propostion~6.1]{seiler2023contact}.
    \item The contact process with switching is not always self-dual. In \cite[Theorem 2.3]{blath2023switching} it is shown that the dual process is again a contact process with switching, but the rates $\lambda_{10}$ and $\lambda_{01}$ are being swapped. Thus, if $\lambda_{10}=\lambda_{01}$, then the process is self-dual otherwise not.
    \end{enumerate}
\end{example}

To prove Proposition~\ref{prop:DualityStationary} we proceed analogously as in \cite[Section~6.1]{seiler2023contact}. For fixed $t>0$ we define $\widehat{\bfxi}^{\xi,t}_s:=\bfxi^{\xi}_{(t-s)}$  for $s\in [0,t]$ and we fix a realisation of the background $\bfxi^{\xi}$ in the time interval $[0,t]$ by conditioning on the background.  We then define a dual process $(\widehat{\bfeta}^{\eta',\xi,t}_s)_{0\leq s\leq t}$ with $\widehat{\bfeta}^{\eta',\xi,t}_0=\eta'\subset V$ by reversing the time flow and starting at the fixed time $t>0$. This process 
is conditionally dual to $(\bfeta^{\eta}_{s})_{s\leq t}$ in the sense that
\begin{equation}\label{eq:ConditinalDuality}
    \IP(\bfeta^{\eta,\xi}_{t}\cap \eta' \neq \emptyset |\cG)=\IP(\bfeta^{\eta,\xi}_{s}\cap \widehat{\bfeta}^{\eta',\xi,t}_{t-s}\neq \emptyset |\cG)=\IP(\eta \cap \widehat{\bfeta}^{\eta',\xi,t}_{t} \neq \emptyset |\cG) 
\end{equation}
holds almost surely for all $s\leq t$, where  $\cG:=\sigma(\bfxi_s:0\leq s\leq t)$ is the $\sigma$-algebra generated by the background process until time $t$.

More precisely, we define $\widehat{\bfeta}$ analogously to $\bfeta$ with the help of the graphical representation using the same infection and recovery events just backwards in time and the direction of the infection is reversed, i.e.\
\begin{equation*}
	(u,\mathbf{inf}_{\bfi,(x,y)})\to (t-u,\mathbf{inf}_{\bfi,(y,x)}) \text{ and } (u,\mathbf{rec}_{j,x}) \to (t-u,\mathbf{rec}_{j,x}),
\end{equation*}
Obviously, $(\widehat{\bfeta},\widehat{\bfxi})$ is in general not a CPDRE. However, by assumption we know that $\bfxi$ is reversible. This implies, in particular, if we start the process stationary, then $(\widehat{\bfeta}_s^{\pi,t},\widehat{\bfxi}_s^{\pi,t})_{s\leq t}$ has the same distribution as a CPDRE with mirrored infection rates.

\begin{proof}[Proof of Proposition~\ref{prop:DualityStationary}]
By observing that $(\widehat{\bfeta}_s^{\pi,t},\widehat{\bfxi}_s^{\pi,t})_{s\leq t}$ and $(\widecheck{\bfeta}^\pi_s,\widecheck{\bfxi}^\pi_s)_{s\leq t}$ have the same distribution and averaging \eqref{eq:ConditinalDuality} with $\bfxi_0\sim\pi$ we obtain a duality relation if the background is initially stationary. See \cite[Proposition 6.1]{seiler2023contact} for a detailed proof.
\end{proof}

Let $\widecheck{\tau}^{\eta}$ denote the extinction time of $\widecheck{\bfeta}^{\eta}$ and $\widecheck{\nu}$ the associated upper invariant law of the process $(\widecheck{\bfeta},\widecheck{\bfxi})$. A direct consequences of Proposition~\ref{prop:DualityStationary} is the following.
\begin{corollary}\label{cor:FiniteVSInfinteSurv}
    Let $\bfxi$ be reversible, then it holds that
    \begin{equation*}
        \IP(\tau^{\zero}=\infty)>0 \Leftrightarrow \widecheck{\nu}\neq \delta_{\emptyset}\otimes \pi \quad and \quad \IP(\widecheck{\tau}^{\zero}=\infty)>0 \Leftrightarrow \nu\neq \delta_{\emptyset}\otimes \pi.
    \end{equation*}
\end{corollary}
\begin{proof}
   This can be proven in the same way as in \cite[Proposition~6.3.]{seiler2023contact} by using \eqref{eq:DualityStationary} and Theorem~\ref{thm:CritValueIndependence}.
\end{proof}

One technical problem is that, unlike in the classical case, the infection processes $(\bfeta_s)_{s\geq t}$ and $(\widehat{\bfeta}^{2t}_s)_{s\leq t}$ are not independent, even though they are defined in disjoint parts of the graphical representation, since we conditioned on the background process in the construction of $\widehat{\bfeta}$.
Therefore, in the following, we introduce another coupling with an auxiliary process $(\widecheck{\bfeta}_{s}^{t/2,2t})_{s\leq t}$ that has this property and will be crucial in later chapters.

Let $(\bfxi^{s/2}_{r} )_{r\geq 0}$ denote the process which is coupled with the original background $(\bfxi_{r})_{r\geq s/2}$ in such a way that it starts at time $s/2$ with an initial distribution $\pi$ which is independent from $(\bfxi^{\xi}_r)_{r\leq s/2}$ and from time $s/2$ onward it uses the same graphical representation as $(\bfxi_r)_{r\geq s/2}$. Let $\widecheck{\bfxi}^{s/2,t+s}_{r}:=\bfxi^{s/2}_{t+s/2-r}$, then $(\widecheck{\bfxi}^{s/2,t+s}_{r} )_{r\leq t+s/2}$ has the same dynamics as the background process $\bfxi$.

Now let $(\widecheck{\bfeta}_{r}^{\eta',s/2,t+s})_{r\leq t+s/2}$ be a process coupled to $\widehat{\bfeta}^{\eta',\xi,t+s}$ by using the same time-reversed infection arrows and recovery symbols from time $t+s$ back to $s/2$, but the environment $(\widecheck{\bfxi}^{s/2,t+s}_r )_{r\leq t+ s/2}$ instead of $(\widehat{\bfxi}^{\xi,t+s}_r)_{r\leq t+ s/2}$.

The following results have already been shown in \cite{seiler2023contact} for special cases of dynamical environments on  arbitrary underlying graphs. Since we consider the $d$-dimensional integer lattices as underlying graph we can improve the result, in the sense that the speed of convergence is in fact exponentially fast, which is again crucial for several proofs. 
\begin{lemma}\label{lem:ControlDualProcess}
	There exist constants $A,B>0$ such that for all $x\in V$ and
    $\xi\in \loc^{V\cup E}$ it holds that
	\begin{equation*}
		\IP(\widehat{\bfeta}_{s}^{x,\xi,2t}=\widecheck{\bfeta}_{s}^{x,t/2,2t}\,\, \forall s\leq t)>1-Ae^{-Bt} \text{ for all } t>0.
	\end{equation*}
\end{lemma}
\begin{proof}
    The proof is a simple modification of the proofs in \cite[Lemma~6.14~and~Lemma~6.16]{seiler2023contact}.
     By translation invariance of the graphical construction we can assume $x=\origin$ and we observe the following set inclusion
    \begin{equation*}
	    \begin{aligned}
	        \{\widehat{\bfeta}_{s}^{\origin,\xi,2t}\subset \B_{M t},\widecheck{\bfeta}_{s}^{\origin,t/2,2t}&\subset \B_{M t} \,\, \forall s\leq t\}\cap \bigcap_{s\geq t/2}\{\bfxi^{t/2}_s(a)=\bfxi^{\xi}_{s+t/2}(a) \,\, \forall a\in  \B_{M s}\cup\B^E_{Ms}\} \\
		    \subseteq &\,\,\{\widehat{\bfeta}_{s}^{\origin,\xi,2t}=\widecheck{\bfeta}_{s}^{\origin,t/2,2t}\,\, \forall s\leq t\}.
	    \end{aligned}
	\end{equation*}
    Let $A,B,M>0$ be the constants of Lemma~\ref{lem:InfContainedExponetialSpeed} for the initial set $\eta=\{\origin\}$. Then we have
    \begin{equation*}
		\IP(\widehat{\bfeta}_{s}^{\origin,\xi,t}\subset \B_{M t},\widecheck{\bfeta}_{s}^{\origin,t/2,2t}\subset \B_{M t} \,\, \forall s\leq t)\geq \IP(  \widetilde{\bfeta}^\origin_t\subset\B_{Mt}) >1-Ae^{-Bt}\quad\text{for all }t\geq 0.
	\end{equation*}
      Controlling the probability
    \begin{align*}
		\IP\bigg(\bigcap_{s\geq \frac{t}{2}}\{\bfxi^{t/2}_s(a)=\bfxi^{\xi}_{s+\frac{t}{2}}(a) \,\, \forall a\in  \B_{M s}\cup\B^E_{Ms}\} \bigg)&\geq \IP(\B_{Ms}\subset \Phi_s\circ \theta_{t/2} \,\forall\, s\geq t/2)\\
        &\geq \IP(\B_{Ms}\subset \Phi_s\,\forall\, s\geq t/2)
        >1-Ae^{-Bt/2}
	\end{align*}
    and redefining the constants $A$ and $B$ proves the claim.
\end{proof}
\begin{remark}
    Note, if we add the superscript $s/2$ and $t+s$ to $\widecheck{\bfeta}^{\eta' ,s/2,t+s}$ we want to emphasise the coupling with $\bfeta^{\eta}$ on the time interval $[s/2,t+s]$. However, since $\bfxi$ is reversible and $\widecheck{\bfxi}_{s/2}^{s/2,t+s}\sim \pi$, i.e.\ the background is stationary, we can extend the process to all times $u\geq t+s/2$ such that it still has the transitions \eqref{InfectionRatesForDualProcess}. In this case we denote it by $(\widecheck{\bfeta},\widecheck{\bfxi})$ and drop the superscript $s/2$ and $t+s$.
\end{remark}

\section{Conditions for Asymptotic Shape}\label{sec:RestartingAndEssentialHitting}
In this section we prove Theorem~\ref{Conjecture1}. Thus, we assume throughout the section that the CPDRE $(\bfeta,\bfxi)$ is supercritical, worst-case monotone (Assumption~\ref{Ass:WS-Monotonicity}) and satisfies \eqref{ConjectureAML}--\eqref{ConjectureFC}. Clearly, on the event of extinction $\{\tau^{\origin,\xi}<\infty\}$ we have almost surely $t(1-\varepsilon)U\not \subset  \bfH_t^\xi$ for all non-empty and convex $U$ and $t$ sufficiently large. Therefore it only makes sense to formulate the asymptotic shape theorem under the measure
\begin{equation*}
\PC(\,\cdot\,):=\IP(\,\cdot\,|\tau^{\origin,\xi}=\infty).
\end{equation*}
Under our assumptions we show that the first hitting times $t^\xi(x)$ satisfy 
\begin{equation}\label{eq:conv_hitting}
\lim_{n\to \infty}\frac{t^\xi(nx)}{n}=\mu(x) \quad \text{for all } x\in V~\text{ $\PC$-a.s.,}
\end{equation}
where $\mu$ is a norm on $V$. This can be used to show the desired result \eqref{desired_result} for the unit ball w.r.t.\ $\mu$ denoted by $\B_{\mu}$.

Our approach to prove the asymptotic shape theorem is based on the ergodic theory for subadditive processes, which can be traced back to  Hammersley and Welsh \cite{hammersley1965first}, Kingman \cite{kingman1973subadditive} and an extension of Kingman's paper by Hammersley and Kesten \cite{Hammersley1974postulates}. However, this machinery requires subadditivity, integrability and stationarity for the increments of the hitting times $t^\xi(x)$, which we do not have. It is clear that the hitting time is not integrable under $\IP$ since it can be infinite under the event of extinction.  Under $\PC$ this is the case, however, under this measure the hitting times are not subadditive and neither do they have stationary increments. We are not the first to face this problem, and thus there is a well established approach which goes back to Durrett and Griffeath \cite{durrett1982several} and was further refined by Garet and Marchand \cite{garet2012shape}.
Therefore, we adapt the \textit{essential hitting time} $\sigma^\xi(x)$ introduced by Garet and Marchand \cite{garet2012shape}.

We verify that this sequence of essential hitting times satisfies the requirements for the special case $\xi=\szero$. 
Then, we can apply an extension of Kingman's subadditive ergodic theorem to get the desired convergence \eqref{eq:conv_hitting} for $\sigma^{\szero}$ and deduce an asymptotic shape for the set of \textit{essentially hit} vertices
\begin{equation*}
    \bfG_t^{\szero}:=\{x\in \IZ^d:\sigma^{\szero}(x)\leq t\}+\big[\unaryminus\tfrac{1}{2},\tfrac{1}{2}\big]^d.
\end{equation*}
To deduce the result for the first hitting time $t^{\szero}(x)$ we then only need to control the difference between $\sigma^{\szero}(x)$ and $t^{\szero}(x)$.

The described approach is an adaptation of the work by Garet and Marchand \cite{garet2012shape} to our situation. In our class of models the infection process $\bfeta$ is still defined via a percolative structure, but $(\bfeta,\bfxi)$ is not necessarily monotone, however, it satisfies our assumption of worst-case monotonicity. This property is needed to \textit{glue infection paths together} in the sense that if $x$ is infected from $\origin$ at time $t^\xi(x)$ and the process restarted at $(x,t^\xi(x))$ with initial background $\szero$ reaches $y$ within time $\hat{t}(y)$, then there also exists a $\xi$-infection path $(\origin,0)\stackrel{\xi}{\longrightarrow} (x+y,t^\xi(x)+\hat{t}(y))$.
This path gluing property is crucial to prove that the essential hitting times are approximately subadditive. Most of the statements which we take from \cite{garet2012shape} can be proven in almost the same way, but we need to generalise some of the results. This is also the reason why we do not provide the proofs of most results in full detail, but only for generalised results where we deem it necessary. 

After verifying that \eqref{desired_result} is valid for $\xi=\szero$, we proceed in two steps to show the result for arbitrary $\xi$. Even though we cannot control the difference between $\sigma^{\xi}(x)$ and $t^{\xi}(x)$ for arbitrary $\xi$, we still obtain that the lower bound for the asymptotic shape of $\bfH^{\xi}$ is the same as for $\bfH^{\szero}$, that is
\begin{equation*}
    \PC\big(\exists s\geq0 :  t(1-\varepsilon)\bfB_{\mu}\subset \bfH_t^{\xi}\,\,\forall t\geq s\big)=1.
\end{equation*}
It seems to us that these techniques can only provide the lower bound. In order to obtain the matching upper bound, we need to use a completely different and novel approach. Therefore, we switch gears and use the fact that we already know that \eqref{desired_result} is true for $\xi=\szero$. Then, we heavily exploit the coupling properties of $\bfxi^{\xi}$. This gives us that the upper bound for the asymptotic shape of $\bfH^{\xi}_t$ must be the same as for $\bfH^{\szero}$ and the proof is complete.

As explained above this proof strategy is tailor made for non-permanent growth models like the basic contact process. Therefore, we first work under the additional assumption that our process is non-permanent in the sense that
\begin{equation*}
\max_{j\in\loc_\pi} r(j)>0,
\qquad
\text{where}\qquad
\loc_\pi:=\{j\in\loc:\pi(\{\xi\in\loc^{V\cup E}:\xi(\origin)=j\})>0\}
\end{equation*}
is the support of the one-site marginal at the origin. One can easily verify that this implies the existence of constants $A,B>0$ such that, for all $\xi\in \loc^{V\cup E}$
\begin{equation}
\IP(\origin\in \bfeta_s^{\origin,\xi}~\forall s\in[0,t])\leq A \exp(-B t)\quad\text{for all }t\geq0.
\label{ConjectureTEC}
\end{equation}
Our class of processes also contains permanent models, where $\max_{j\in\loc_\pi} r(j)=0$. In that case the proof of the asymptotic shape theorem is significantly easier, since one can work directly with the first hitting time and does not need to define the essential hitting time. However, some of the proofs concerning the essential hitting times below --- in particular, the proof of Lemma~\ref{Lemma:exp_control_N} --- no longer work. Therefore, we provide at the end a short direct proof of Theorem~\ref{Conjecture1} for these kinds of models, noting that similar results are also given in \cite{cox1981limit} and that this is not a genuine novelty.

\subsection{Definition and First Properties of the Essential Hitting Time}
Let us define the essential hitting time $\sigma^\xi(x)$ of $x$ similar as in \cite{garet2012shape} via a sequence of stopping times $v^\xi_i(x)$, $l^\xi_i(x)$ and $u^\xi_i(x)$ with $v^\xi_0(x)=l^\xi_0(x)=u^\xi_0(x)=0$ as follows:
\begin{enumerate}
    \item Suppose $v^\xi_{k-1}(x)$ is already defined. Set $l^\xi_{k}(x)=\inf\{t\geq v^\xi_{k-1}(x):x\notin\bfeta_{t}^\xi\}$. Note that if $v^\xi_{k-1}(x)$ is finite, then $l^\xi_k(x)$ is finite almost surely by \eqref{ConjectureTEC}.
    \item Suppose $l_k^\xi(x)$ is already defined. Set $u^\xi_{k}(x)=\inf\{t\geq l^\xi_k(x):x\in \bfeta_t^{\xi}\}$. In particular, if $l^\xi_{k}(x)$ is finite, then $u^\xi_{k}(x)$ is the first time after $v^\xi_{k-1}(x)$ when the site $x$ is re-infected again. 
    \item Suppose $u^\xi_k(x)$ is defined, then $v^\xi_k(x)=u^\xi_k(x)+\tau^{x,\szero}\circ \theta_{u^\xi_k(x)}$, where the second summand is the lifetime of the process starting with configuration $\big(\delta_x,\szero\big)$ at time $u^\xi_k(x)$. Recall that $\theta_{u^\xi_k(x)}$ describes the time shift of the graphical construction, i.e.\ a shift of $-u^\xi_k(x)$ for all arrival times given by the Poisson Processes from our construction. 
\end{enumerate}
Now set 
\begin{equation}
	K^\xi(x)=\min\{n\geq 0 :v^\xi_n(x)=\infty~\text{or}~u^\xi_{n+1}(x)=\infty\}
\end{equation}
and $\sigma^\xi(x)=u^\xi_{K^\xi(x)}(x)$.
Thus, at time $\sigma^\xi(x)$ the site $x$ becomes infected and one of the following happens: either there exists a $\szero$-infection path from the space-time point $(x,\sigma^\xi(x))$ up to infinity, or after all $\szero$-infection paths starting from $(x,\sigma^\xi(x))$ have died, the site $x$ is never infected again.
Moreover, $K^\xi(x)$ is the number of iterations we need to find the essential hitting time. Figure \ref{figure:ess_hitting_time} illustrates the definition of the essential hitting time once again.\\
\begin{figure}[t]
\centering
						\scalebox{1}{
							\begin{tikzpicture}[darkstyle/.style={circle,inner sep=0pt}]
							
							\draw[->,thick](0,0)--(0,6.3)node[left,xshift=-0.2cm, yshift=-0.2cm] {time};
							
							\draw[-,thick](0,0)--(6,0)node[right,xshift=-3cm, yshift=-1 cm] {$\IZ^d$};
							\node () at (0,-0.3){$0$};
							\draw[fill=black](0,0) circle(1pt);
							\draw[fill=black](4,0) circle(1pt);
							\node () at (4,-0.3){$x$};
							
							\draw [black] plot [smooth] coordinates { (0,0) (1,0.3) (2,0.5) (3,1.3) (4,1.5)};
							\draw[fill=black](0,1.5) circle(1pt)node[left]{$u^\xi_1(x)$};
							\draw[dashed](0,1.5) --(4,1.5);
							
							\draw [blue,thick] plot [smooth] coordinates { (4,1.5) (3.5,1.75) (3,2) };
							\draw [blue,thick] plot [smooth] coordinates { (4,1.5) (4.5,2.6) (5,3)};
							
							\draw[fill=blue](0,3) circle(1pt)node[left]{$v^\xi_1(x)$};
							\draw[dashed](0,3) --(5,3);
							
							\draw [black] plot [smooth] coordinates { (1,0.3) (2,3) (4,3.8)};
							\draw[fill=black](0,3.8) circle(1pt)node[left]{$u^\xi_2(x)$};
							\draw[dashed](0,3.8) --(4,3.8);
							
							\draw [blue,thick] plot [smooth] coordinates { (4,3.8) (5,4.4) (4,4.6)};
							\draw[fill=blue](0,4.6) circle(1pt)node[left]{$v^\xi_2(x)$};
							\draw[dashed](0,4.6) --(4,4.6);
							
							\draw [black] plot [smooth] coordinates { (0,0) (2,4) (4,5.5)};
							\draw[fill=black](0,5.5) circle(1pt)node[left]{$u^\xi_3(x)$};
							\draw[dashed](0,5.5) --(4,5.5);
							
							\draw [black] plot [smooth] coordinates { (0,0) (2,4) (4,5.5)};
							\draw [red] plot [smooth] coordinates { (4,5.5) (5,6)}node[right]{$\infty$};

							\end{tikzpicture}}
                    \caption{Construction of the essential hitting times: The black paths are the $\xi$-infection paths starting from the origin. The blue and red ones correspond to $\szero$-infection paths starting from space-time points $(x,u^\xi_i(x))$. Only the red path is of infinite length. In the specific example $K^\xi(x)=3$, $\sigma^\xi(x)=u^\xi_3(x)$ and $l_i^\xi(x)=v_{i-1}^\xi(x)$ for $i=2,3$.}
                    \label{figure:ess_hitting_time}
					\end{figure}

We define the space-time shift $\hat{\theta}_x^\xi$ for subsets $\omega\subset \cM\times \IR$  as
\begin{align*}
    \hat{\theta}^\xi_x(\omega):=\begin{cases}
        T_x\circ \theta_{\sigma^\xi(x)}(\omega)\quad&\text{if }\sigma^\xi(x)<\infty,\\
        T_x(\omega) &\text{otherwise.}
    \end{cases}
\end{align*}

\begin{remark}
In contrast to \cite{garet2012shape} and \cite{deshayes2014contact}, we do not directly define $u_{k+1}(x)=\inf\{t\geq v_k(x): x\in \bfeta_t^{\origin}\}$ but instead work with the definition involving $l_k(x)$. This is because we need to ensure that the site $x$ gets infected at time $u_{k+1}(x)$, i.e.\ that an infection event/arrow pointing at $x$ occurs at that time, which is not necessarily guaranteed by the definitions in \cite{garet2012shape} and \cite{deshayes2014contact}. In those definitions, it can happen that $x$ is infected at time $u_{k+1}(x)=v_k(x)$ but there is no infection event at that time.
\end{remark}

First note that $K^{\xi}$ is almost surely finite and has sub-geometric tail-probabilities. Moreover, conditioned on survival, at time $(x,\sigma(x))$ there exists a $\szero$-infection path up to infinity. This two facts are shown in the two subsequent lemmas.
\begin{lemma}\label{lemma:K_as_finite}
	Let $\rho:=\IP(\tau^{\origin,\szero}=\infty)>0$. Then we have for all  $x\in\IZ^d$ and $n\in\IN$ that
	\begin{equation*}
		\IP(K^\xi(x)>n)\leq (1-\rho)^n.
	\end{equation*}
\end{lemma}
\begin{proof}
    This can be proven analogously as \cite[Lemma~6]{garet2012shape}.
\end{proof}
\begin{lemma}\label{Lemma:equivalence_K_2}
	For every $x\in V$ and every background configuration $\xi$ we have almost surely
	\begin{align*}
		(K^\xi(x)=k\text{ and }\tau^{\origin,\xi}=\infty)\Leftrightarrow (u^\xi_k(x)<\infty\text{ and }v^\xi_k(x)=\infty)
	\end{align*}
\end{lemma}
\begin{proof}[Proof (adapted from {\cite[Lemma~7]{garet2012shape}})] We first show the implication from left to right. Let  $x\in V$, $k\in \IN$, and observe that
    \begin{align}
		\IP(&\tau^{\origin,\xi}=\infty,v^\xi_k(x)<\infty,u^\xi_{k+1}(x)=\infty)\notag\\
        &=\IP(\tau^{\origin,\xi}=\infty,v^\xi_k(x)<\infty,l^\xi_{k+1}(x)=\infty)\notag\\
        &+\IP(\tau^{\origin,\xi}=\infty,l^\xi_{k+1}(x)<\infty,u^\xi_{k+1}(x)=\infty),\notag
	\end{align} 
where the first term is zero due to \eqref{ConjectureTEC}. For the second term we apply the strong Markov-property at time $l^\xi_{k+1}(x)$ to obtain
	\begin{align*}
		\IP(&\tau^{\origin,\xi}=\infty,l^\xi_{k+1}(x)<\infty,u^\xi_{k+1}(x)=\infty|\cF_{l^\xi_{k+1}(x)})\\
        &=\1_{\{l^\xi_{k+1}(x)<\infty\}}\IP(\tau^{\cdot}=\infty,t^{\cdot}(x)=\infty)\circ (\bfeta_{l_{k+1}^\xi(x)}^{\origin,\xi},\bfxi_{l_{k+1}^\xi(x)}^{\origin,\xi}).
	\end{align*}
	Since $\bfeta_{l_{k+1}^\xi(x)}^{\origin,\xi}$ is almost surely finite if $l^\xi_{k+1}(x)<\infty$, we can use property \eqref{ConjectureALL} to deduce that the last probability is zero which yields the implication. For the reverse one observe, that the right-hand side directly implies $K^\xi(x)=k$. Moreover, the graphical construction implies a coupling which guarantees that $\tau^{\origin,\xi}=\infty$. Note that it is crucial that, at $u_{k}^{\xi}(x)$, we restart the process with the background in the worst state $\szero$ and Assumption~\ref{Ass:WS-Monotonicity} holds, as otherwise the survival of the original process would not be guaranteed.
\end{proof}
\begin{remark}
    By definition of $\sigma^\xi(x)$ we also have
    	\begin{align*}
		(K^\xi(x)=k\text{ and }\tau^{\origin,\xi}=\infty)\Leftrightarrow (\sigma^\xi(x)=u^\xi_k(x)\text{ and }\tau^{\origin,\xi}=\infty)
	\end{align*}
\end{remark}
We now state the ergodic theorem for subadditive processes from Deshayes \cite{deshayes2014contact}, which we will apply later. This result is based on an extension of Kingman's work \cite{kingman1973subadditive} by Hammersley and Kesten \cite{Hammersley1974postulates}, and its proof mainly follows the arguments in \cite{garet2012shape}.
\begin{theorem}[{\cite[Theorem 8.2]{deshayes2014contact}}]\label{theorem:deshayes_as_shape}
    Let $(\Omega,\cF,\IP)$ be a probability space. Let $(s(x))_{x\in \IZ^d}$ be random variables with finite second moments and suppose that $s(x)$ and $s(-x)$ have the same distribution, for every $x\in\IZ^d$.
    Let $(v(y,x))_{y,x\in \IZ^d}$ and $(r(x,y))_{x,y\in \IZ^d}$ be collections of random variables such that:
    \begin{enumerate}[label=(AS\arabic*),ref=AS\arabic*]
        \item $\forall x,y \in\IZ^d$, $s(x+y)\leq s(x)+v(y,x)+r(x,y)$ with $v(y,x)$ having the same distribution as $s(y)$, and being independent of $s(x)$. \label{prop:1}
        \item $\forall x,y\in \IZ^d$, $\exists C_{x,y}$ and $\alpha_{x,y}<2$ such that $\IE[r(nx,py)^2]\leq C_{x,y}(n+p)^{\alpha_{x,y}}$.\label{prop:2}
        \item  $\exists C>0$ such that $\forall x\in \IZ^d$, $\IP(s(nx)>Cn||x||)\to 0$ as $n\to \infty.$\label{prop:3}
        \item  $\exists K>0$ such that for $\forall\varepsilon>0$, $\IP$- almost surely $\exists M$ such that
        \begin{equation*}
            (||x||>M~\text{and}~||x-y||\leq K\varepsilon||x||)\Rightarrow||s(x)-s(y)||\leq \varepsilon ||x||.\label{prop:4}
        \end{equation*}
        \item  $\exists c>0$ such that $\forall x\in \IZ^d$, $\IP(s(nx)<cn ||x||)\to 0$, as $n\to \infty$.\label{prop:5}
    \end{enumerate}
    Then there exists $\mu:\IZ^d\to\IR^+$ such that 
    \begin{equation*}
        \lim_{||x||\to\infty}\frac{s(x)-\mu(x)}{||x||}=0~\text{almost surely and in }L^2.
    \end{equation*}
    Moreover, $\mu$ can be extended to a norm on $\IR^d$ and we have the following asymptotic shape theorem: For all $\varepsilon>0$, $\IP$-almost surely, for all $t$ large enough,
    \begin{equation*}
        (1-\varepsilon)\B_\mu\subset\frac{\hat{G}_t}{t}\subset(1+\varepsilon)\B_\mu,
    \end{equation*}
    where $\hat{G}_t:=\{x\in \IZ^d:s(x)\leq t\}+\big[\unaryminus\tfrac{1}{2},\tfrac{1}{2}\big]^d$ and $\B_\mu$ is the unit ball for $\mu$.
\end{theorem}
\begin{remark}
	In contrast to \cite[Theorem~8.2]{deshayes2014contact}, the assumptions \eqref{prop:1} and \eqref{prop:4} are slightly different. However, the proof presented there applies exactly to our setup, so we believe that the difference is only due to a typo in \eqref{prop:4} and some inaccuracy in \eqref{prop:1}.
\end{remark}

Clearly, we would like to apply the theorem to the essential hitting times $\sigma^\xi(x)$. However, we can only do this in case $\xi=\szero$ otherwise $\sigma^\xi$ lacks some invariance property which is crucial to show assumption \eqref{prop:1}. 
\subsection{Invariance of $\sigma^{\szero}$} 
We will establish an invariance property of the shift $\hat{\theta}_x^{\szero}$ under $\PO$,  which guarantees the stationarity property of the essential hitting times that we need.

\begin{lemma}\label{Lemma:invariance_2}
	Let $x\in V$, $A$ in the $\sigma$-algebra generated by $\sigma^\xi(x)$ and $E\in \cF$ any measurable event with respect to the Poisson construction. Then
	\begin{align}\label{eq:invariance}
		\IP(A\cap (\hat{\theta}^\xi_x)^{-1}(E)|\tau^{\origin,\xi}=\infty)=\IP(A|\tau^{\origin,\xi}=\infty)\IP(E|\tau^{\origin,\szero}=\infty).
	\end{align}
\end{lemma}
\begin{proof}[Proof (adapted from {\cite[Lemma~8]{garet2012shape}})]
	We first check that for any $k\in\IN$ the equation
	\begin{align*}
	   \IP(A\cap (\hat{\theta}^\xi_x)^{-1}(E)\cap\{K^\xi(x)=k\}|\tau^{\origin,\xi}=\infty)=\IP(A\cap\{K^\xi(x)=k\}|\tau^{\origin,\xi}=\infty)\IP(E|\tau^{\origin,\szero}=\infty)
	\end{align*}
    holds. Let $A'\subseteq \IR$ be a Borel set such that $A=\{\omega:\sigma^\xi(x)(\omega)\in A'\}$. As in \cite{garet2012shape} we use the fact that the essential hitting time (which is not a stopping time itself) is constructed via stopping times, to derive
	\begin{align}
		\IP(&A\cap(\hat{\theta}^\xi_x)^{-1}(E)\cap\{K^\xi(x)=k\}\cap\{\tau^{\origin,\xi}=\infty\})\notag\\
		&=\IP(A,(\hat{\theta}^\xi_x)^{-1}(E),u_k^\xi(x)<\infty,v^\xi_k(x)=\infty)\label{eq:1}\\
		&=\IP(\sigma^\xi(x)\in A',T_x\circ \theta_{\sigma^\xi(x)}\in E,u^\xi_k(x)<\infty,v^\xi_k(x)=\infty)\label{eq:2}\\
        &=\IP(u^\xi_k(x)\in A',T_x\circ \theta_{u^\xi_k(x)}\in E,u^\xi_k(x)<\infty,\tau^{x,\szero}\circ \theta_{u^\xi_k(x)}=\infty)\label{eq:3}\\
        &=\IE\big[\1_{\{u^\xi_k(x)\in A',u^\xi_k(x)<\infty\}}\IP(T_x\circ \theta_{u^\xi_k(x)}\in E,\tau^{x,\szero}\circ \theta_{u^\xi_k(x)}=\infty|\cF_{u^\xi_k(x)})\big]\notag \\
        &=\IP(u^\xi_k(x)\in A',u^\xi_k(x)<\infty)\IP(E,\tau^{\origin,\szero}=\infty),\label{eq:4}
	\end{align}
 where \eqref{eq:1} follows by Lemma \ref{Lemma:equivalence_K_2}, \eqref{eq:2} by definition of $A'$ and $\hat{\theta}_x^\xi$, \eqref{eq:3} by definition of $v^\xi_k(x)$ and the fact that $\sigma^\xi(x)=u^\xi_k(x)$ if $u^\xi_k(x)<\infty$ and $v^\xi_k(x)=\infty$, and \eqref{eq:4} by the strong Markov-property and the temporal and spatial invariance of the Poisson construction. Dividing by $\IP(\tau^{\origin,\xi}=\infty)>0$ yields
\begin{align*}
	\IP(A\cap &(\hat{\theta}^\xi_x)^{-1}(E)\cap\{K^\xi(x)=k\}|\tau^{\origin,\xi}=\infty)\\
    &=\IP(u_k^\xi(x)\in A',u_k^\xi(x)<\infty)\frac{\IP(\tau^{\origin,\szero}=\infty)}{\IP(\tau^{\origin,\xi}=\infty)}\IP(E|\tau^{\origin,\szero}=\infty).
\end{align*}
Plugging in $E=\Omega$ one easily verifies
\begin{align*}
    \IP(u_k^\xi(x)\in A',u_k^\xi(x)<\infty)\frac{\IP(\tau^{\origin,\szero}=\infty)}{\IP(\tau^{\origin,\xi}=\infty)}=\IP(A\cap\{K^\xi(x)=k\}|\tau^{\origin,\xi}=\infty)
\end{align*} which shows the desired equation and finishes the proof of \eqref{eq:invariance}.
\end{proof}

\begin{corollary}\label{Corollary:invariance_2}
	Let $x,y\in V$ with $x\neq 0$, then:
	\begin{enumerate}	
	\item We have $\IP(\cdot|\tau^{\origin,\szero}=\infty)=\IP(\cdot\circ \hat{\theta}^\xi_y|\tau^{\origin,\xi}=\infty)$.
    \item $\sigma^\xi(y)\circ\hat{\theta}^\xi_x$ and $\sigma^\xi(x)$ are independent. Moreover, $\sigma^\xi(y)\circ\hat{\theta}^\xi_x$ and $\sigma^\xi(y)\circ\hat{\theta}^\xi_\origin$ are identically distributed under $\IP(\cdot|\tau^{\origin,\xi}=\infty)$.
    \item The variables $\sigma^{\szero}(x)\circ\hat{\theta}^\xi_\origin$ and $\sigma^{\szero}(y)\circ\hat{\theta}^{\szero}_x\circ\hat{\theta}^\xi_\origin$ are independent and $\sigma^{\szero}(y)\circ\hat{\theta}^\xi_\origin$ is identically distributed as $\sigma^{\szero}(y)\circ\hat{\theta}^{\szero}_x\circ\hat{\theta}^\xi_\origin$ with respect to the law $\PC$.
    \item The random variables $(\sigma^\xi(x)\circ(\hat{\theta^\xi_x})^j)_{j\geq 0}$ are independent under $\IP(\cdot|\tau^{\origin,\xi}=\infty)$.
	\end{enumerate}
\end{corollary}
\begin{proof}[Proof (adapted from {\cite[Corollary~9]{garet2012shape}})]
    The first point follows from Lemma~\ref{Lemma:invariance_2} by taking $A=\Omega$, which yields
    \begin{equation}\label{eq:ZeroShift}
        \IP((\hat{\theta}^\xi_{y})^{-1}(E)|\tau^{\origin,\xi}=\infty)=\IP(E|\tau^{\origin,\szero}=\infty)\quad\text{for all }E\in \cF.
    \end{equation}For the second, take two Borel sets $A'$ and $B'$ and define $A:=\{\sigma^\xi(x)\in A'\}$ as well as $B:=\{\sigma^\xi(y)\in B'\}$. Now we apply Lemma \ref{Lemma:invariance_2} to the sets $A$ and $B$ to derive
    \begin{align}\label{eq:IndepedenceAndIdentical}
        \IP(\sigma^\xi(x)\in A', \sigma^\xi(y)\circ\hat{\theta}_x^\xi\in B'|\tau^{\origin,\xi}=\infty)&=\IP(A\cap (\hat{\theta}^\xi_x)^{-1}(B)|\tau^{\origin,\xi}=\infty)\notag\\
        &=\IP(A|\tau^{\origin,\xi}=\infty)\IP(B|\tau^{\origin,\szero}=\infty)\\
        &=\IP(\sigma^\xi(x)\in A'|\tau^{\origin,\xi}=\infty) \IP(\sigma^\xi(y) \in B'|\tau^{\origin,\szero}=\infty).\notag
    \end{align}
    By the first property it follows that
    \begin{align*}
        \IP(\sigma^\xi(y) \in B'|\tau^{\origin,\szero}=\infty)=\IP( \sigma^\xi(y)\circ\hat{\theta}_x^\xi\in B'|\tau^{\origin,\xi}=\infty)
    \end{align*}
    which concludes the proof of independence. Moreover, this implies together with \eqref{eq:ZeroShift} that $\sigma^\xi(y)\circ\hat{\theta}^\xi_x$ and $\sigma^\xi(y)\circ\hat{\theta}^\xi_\origin$ are identically distributed.

    For the third point we see that by \eqref{eq:ZeroShift} it follows immediately that independence of $\sigma^{\szero}(x)\circ\hat{\theta}^\xi_\origin$ and $\sigma^{\szero}(y)\circ\hat{\theta}^{\szero}_x\circ\hat{\theta}^\xi_\origin$ under $\PC$ translates to independence of $\sigma^{\szero}(x)$ and $\sigma^{\szero}(y)\circ\hat{\theta}^{\szero}_x$ under $\PO$.
    
    In order to see that $\sigma^{\szero}(y)\circ\hat{\theta}^\xi_\origin$ and $\sigma^{\szero}(y)\circ\hat{\theta}^{\szero}_x\circ\hat{\theta}^\xi_\origin$ are identically distributed  with respect to the law $\PC$ we again note that this corresponds to $\sigma^{\szero}(y)$ and $\sigma^{\szero}(y)\circ\hat{\theta}^{\szero}_{x}$ being identically distributed with respect to $\PO$, which can be seen by \eqref{eq:ZeroShift}. But since $\sigma^{\szero}(y)\circ\hat{\theta}^{\szero}_{\origin}=\sigma^{\szero}(y)$ this follows by the second point.
    The last point follows analogously to \cite[Corollary 9]{garet2012shape}.
\end{proof}
The first property of Corollary~\ref{Corollary:invariance_2}  shows that $\PC$ is only invariant under the shift $\hat{\theta}^\xi_x$ if $\xi=\szero$. 
At the same time it tells us the correct time we need to consider in case $\xi\neq \szero$. This is the shift of the essential hitting times $\sigma^{\szero}(x)$ by $\hat{\theta}^\xi_0$, which we denote by
\begin{equation*}
    s^\xi(x):=\sigma^{\szero}(x)\circ \hat{\theta}^\xi_{\origin}.
\end{equation*}
However, to later compare the \textit{shifted essential hitting time $s^\xi(x)$} with the first hitting time $t^\xi(x)$, we cannot neglect the time shift and therefore have to consider the time  
\begin{equation*}
	\tilde{s}^\xi(x):=s^{\xi}(x)+\sigma^{\xi}(\origin),
\end{equation*}
for which we have $t^\xi(x)\leq \tilde{s}^\xi(x)$ by construction and by our worst-case monotonicity assumption.

We later apply Theorem~\ref{theorem:deshayes_as_shape} to this shifted essential hitting time $s^\xi(x)$ (note that $s^{\szero}=\sigma^{\szero}$) and deduce an asymptotic shape for
\begin{equation*}
	\bfG_t^\xi:=\{x\in \IZ^d:s^\xi(x)\leq t\}+\big[\unaryminus\tfrac{1}{2},\tfrac{1}{2}\big]^d.
\end{equation*}

For fixed $\xi$ let $s(x):=s^\xi(x)$, $v(y,x):= \sigma^{\szero}(y)\circ \hat{\theta}^{\szero}_x\circ \hat{\theta}_0^\xi$ and $r(x,y)=r^\xi(x,y)^+:=[s(x+y)-(s(x)+v(x,y))]^+$. Clearly, the inequality of assumption~\eqref{prop:1} is fulfilled by definition of the variables. The fact that $v(y,x)$ and $s(y)$ have the same distribution follows by Corollary~\ref{Corollary:invariance_2} as well as the independence of $s(x)$ and $v(y,x)$. Hence assumption \eqref{prop:1} is fulfilled.
Out of the remaining assumptions the second moment bound \eqref{prop:2} for the error term
\begin{equation*}
    r^\xi(x,y)^+=[s^\xi(x+y)-(s^\xi(x)+\sigma^{\szero}(y)\circ \hat{\theta}^{\szero}_x\circ \hat{\theta}_0^\xi)]^+
\end{equation*} (with respect to subadditivity of $s^\xi$) requires the most work. It is comparatively less difficult to show the remaining conditions such as finite second moment of $s^\xi(x)$ and \eqref{prop:3} - \eqref{prop:4}.
\begin{remark}\label{remark:initial_cond}
    By definition and Corollary~\ref{Corollary:invariance_2} the shifted essential hitting time $s^\xi$ and the error term $r^\xi$ have the same distributions under $\PC$ as $\sigma^{\szero}$ and $r^{\szero}$ under $\PO$.
\end{remark}

Note that for our final result we still need to control the difference between the shifted essential hitting time $s^\xi(x)$ and the first hitting time $t^\xi(x)$, which we only achieve in case $\xi=\szero$. Both problems -- controlling the difference $s^{\szero}-t^{\szero}$ and the error term $r^{\szero}$ -- are of similar nature (see Figure \ref{figure:2_problems}) and can be solved by controlling the differences of $v_i^{\szero}(x)-u_i^{\szero}(x)$ and those of $u_{i+1}^{\szero}(x)-v_i^{\szero}(x)$ separately for all $1\leq i< K^{\szero}(x)$.
    \begin{figure}[t]
    \centering
			\scalebox{1}{
				\begin{tikzpicture}[darkstyle/.style={circle,inner sep=0pt}]
				\draw[->,thick](0,0)--(0,6.3)node[left,xshift=-0.2cm, yshift=-0.2cm] {time};
				
				\draw[-,thick](0,0)--(3,0)node[] {};
				\node () at (0,-0.3){$0$};
				\draw[fill=black](0,0) circle(1pt);
				\draw[fill=black](2,0) circle(1pt);
				\node () at (2,-0.3){$x+y$};
				
				\draw [black] plot [smooth] coordinates { (0,0) (0.5,0.3) (1,0.5) (1.5,1.3) (2,1.5)};
				\draw[fill=black](0,1.5) circle(1pt)node[left]{$u_1^{\szero}(x+y)$};
				\draw[dashed](0,1.5) --(2,1.5);
				
				\draw [black,thick] plot [smooth] coordinates {(0,0) (1,1.5) (2,2.5)};
				\draw[fill=black](0,2.5) circle(1pt)node[left]{$\sigma^{\szero}(x)+\sigma^{\szero}(y)\circ \hat{\theta}^{\szero}_x$};
				\draw [red] plot [smooth] coordinates { (2,2.5) (2.5,2.7) (3,3.4)}node[right]{$\infty$};
				\draw[dashed](0,2.5) --(2,2.5);
				
				\draw [blue,thick] plot [smooth] coordinates { (2,1.5) (1.75,1.75) (1.6,1.8) };
				\draw [blue,thick] plot [smooth] coordinates { (2,1.5) (2.5,2.1) (3,3)};
				
				\draw[fill=blue](0,3) circle(1pt)node[left]{$v_1^{\szero}(x+y)$};
				\draw[dashed](0,3) --(3,3);
				
				\draw [black] plot [smooth] coordinates { (0.5,0.3) (1,3) (2,3.8)};
				\draw[fill=black](0,3.8) circle(1pt)node[left]{$u_2^{\szero}(x+y)$};
				\draw[dashed](0,3.8) --(2,3.8);
				
				\draw [blue,thick] plot [smooth] coordinates { (2,3.8) (2.5,4.4) (2,4.6)};
				\draw[fill=blue](0,4.6) circle(1pt)node[left]{$v_2^{\szero}(x+y)$};
				\draw[dashed](0,4.6) --(2,4.6);
				
				\draw [black] plot [smooth] coordinates { (0,0) (1,4) (2,5.5)};
				\draw[dashed](0,5.5) --(2,5.5);
				
				\draw [black] plot [smooth] coordinates { (0,0) (1,4) (2,5.5)};
				\draw [red] plot [smooth] coordinates { (2,5.5) (2.5,6)}node[right]{$\infty$};
				\draw[fill=black](0,5.5) circle(1pt)node[left]{$\sigma^{\szero}(x+y)$};
				\draw[dashed](0,5.5) --(2,5.5);
				\draw[<->,thick,purple] (2,2.5) -- (2,5.5)node[right,yshift=-0.5cm]{$r^{\szero}(x,y)$};

				\begin{scope}[shift={(5.9,0)}]
				\draw[->,thick](0,0)--(0,6.3)node[left,xshift=-0.2cm, yshift=-0.2cm] {time};
				
				\draw[-,thick](0,0)--(3,0)node[] {};
				\node () at (0,-0.3){$0$};
				\draw[fill=black](0,0) circle(1pt);
				\draw[fill=black](2,0) circle(1pt);
				\node () at (2,-0.3){$x+y$};
				
				\draw [black] plot [smooth] coordinates { (0,0) (0.5,0.3) (1,0.5) (1.5,1.3) (2,1.5)};
				\draw[fill=black](0,1.5) circle(1pt)node[left]{$u_1^{\szero}(x+y)$};
				\draw[dashed](0,1.5) --(2,1.5);
				
				\draw [blue,thick] plot [smooth] coordinates { (2,1.5) (1.75,1.75) (1.6,1.8) };
				\draw [blue,thick] plot [smooth] coordinates { (2,1.5) (2.5,2.1) (3,3)};
				
				\draw[fill=blue](0,3) circle(1pt)node[left]{$v_1^{\szero}(x+y)$};
				\draw[dashed](0,3) --(3,3);
				
				\draw [black] plot [smooth] coordinates { (0.5,0.3) (1,3) (2,3.8)};
				\draw[fill=black](0,3.8) circle(1pt)node[left]{$u_2^{\szero}(x+y)$};
				\draw[dashed](0,3.8) --(2,3.8);
				
				\draw [blue,thick] plot [smooth] coordinates { (2,3.8) (2.5,4.4) (2,4.6)};
				\draw[fill=blue](0,4.6) circle(1pt)node[left]{$v_2^{\szero}(x+y)$};
				\draw[dashed](0,4.6) --(2,4.6);
				
				\draw [black] plot [smooth] coordinates { (0,0) (1,4) (2,5.5)};
				\draw[dashed](0,5.5) --(2,5.5);
				
				\draw [black] plot [smooth] coordinates { (0,0) (1,4) (2,5.5)};
				\draw [red] plot [smooth] coordinates { (2,5.5) (2.5,6)}node[right]{$\infty$};
				\draw[fill=black](0,5.5) circle(1pt)node[left]{$\sigma^{\szero}(x+y)$};
				\draw[dashed](0,5.5) --(2,5.5);
				
				\draw[<->,thick,purple] (2,1.5) -- (2,5.5)node[right,yshift=-0.5cm]{$\sigma^{\szero}(x+y)-t^{\szero}(x+y)$};
				\end{scope}

				\end{tikzpicture}}
                \caption{A comparison between both quantities which we want to control. On the left the (positive) error term $r^{\szero}(x,y)$ for the subadditivity of $\sigma^{\szero}(x)$, on the right the difference between the essential hitting time $\sigma^{\szero}$ and the first hitting time $t^{\szero}=u_1^{\szero}$.}
                \label{figure:2_problems}
	\end{figure}

\subsection{Controlling $v_i-u_i$ and $u_{i+1}-v_i$}
From now on, we sometimes omit the initial configuration 
$\xi$ of the background process from our notation when it is clear from the context, in particular in the proofs. 
In a first step we control the difference  $v_i(x)-u_i(x)$ uniformly for all $i<K(x)$.
\begin{lemma}\label{Lemma:exp_control_diff_vu}
    There exist constants $A$ and $B$ such that for all $x\in V$ and $\xi\in \loc^{V\cup E}$
    \begin{equation*}
        \PC(\exists i< K^\xi(x)~\text{such that}~v_i^\xi(x)-u_i^\xi(x)>t)\leq A\exp{(-B t)}\quad\text{for all }t\geq 0.
    \end{equation*}
\end{lemma}
\begin{proof} This can be proven analogously to \cite[Lemma 12]{garet2012shape} using the strong Markov property.\qedhere
\end{proof}

Next we want to bound the reinfection times $u_{i+1}(x)-v_i(x)$ in a similar fashion as in \cite{garet2012shape}. Thus we define 
\begin{equation*}
    \gamma:=3M\big(1+c^{-1}\big)>3
\end{equation*} where $M$ and $c$ are the constants chosen in \eqref{ConjectureAML} and \eqref{ConjectureALL}, respectively, where we assume w.l.o.g.\ that $M>1$.

We denote by $\Delta_x^{}[0,t)=(\{\mathbf{rec}_{\bfj,x}: \bfj\in \loc\}\times[0,t))\cap\Delta$ all potential recovery  events that arrive at site $x$ within the time interval $[0,t)$ and analogously $\Delta^{}_{(x,y)}[0,t):=(\{\mathbf{inf}_{\bfi,(x,y)}: \bfi\in \loc^3\}\times[0,t))\cap\Delta$ is the random set of all infection events from $x$ to $y$ up to time $t$.
Moreover, we set
$\Delta_x^{}:=\Delta_x^{}[0,\infty)$ and $\Delta^{}_{(x,y)}:=\Delta^{}_{(x,y)}[0,\infty)$. 
We  consider the \textit{effective} recovery events arriving at $x$ if we start with a given initial background configuration $\xi$, say $\Delta_x^\xi[0,t)$, with 
\begin{equation*}
    \Delta^\xi_x[0,t)=\{(\mathbf{rec}_{\bfj,x},s)\in\Delta_x[0,t): G(\bfxi^{\xi}_s(x))\leq G(\bfj)\}.
\end{equation*}
For $x,y\in V$ and $t>0$ we define the event of bad growth for $x$ from $(y,0)$ at scale $t$ as
\begin{align*}
    E^{y,\xi}(x,t):=&\{\Delta_y^\xi[0,t/2)=\emptyset\}\cup\{\bfH_t^{y,\xi}\cap V\not\subset \B_{Mt}(y)\}\cup\{t/2<\tau^{y,\szero}<\infty\}\\
    &\cup\{t/2<\tau^{y,\xi}<\infty\}\cup\{\tau^{y,\xi}=\infty,\inf\{s\geq 2t:x\in \bfeta_s^{y,\xi}\}>\kappa t\}.
\end{align*}
\begin{remark}
    The definition may look a little bit odd at first glance, in particular  one may ask if we really need both events $\{t/2<\tau^{y,\szero}<\infty\}$ and $\{t/2<\tau^{y,\xi}<\infty\}$. Indeed we require both events, which becomes apparent in the proof of the subsequent Lemma~\ref{Lemma:bound_u}. 
\end{remark}
We count the number of such events at every recovery or infection symbol (and at time $0$ and $L$) in a space-time box of radius $M t+2$ and height $L$ around $x$:
\begin{equation*}
    N_L^\xi(x,t):=\sum_{y\in \B_{Mt+2}(x)}\int_0^L\1_{E^{y,\bfxi^{\xi}_s}(x,t)}\circ\theta_s~d\bigg(\delta_0+\delta_L+\Delta_y+\sum_{e:y\in e}\Delta_e\bigg)(s).
\end{equation*}
\begin{observation}\label{observation:distance_recovery}
    If $N_{L}^{\bfxi_s}(x,t)\circ \theta_s=0$ then  for every $y\in \B_{Mt+2}(x)$ from time $s$ up to time $s+L+t/2$ the effective recovery symbols arriving for $y$ have distances less than $t/2$.
\end{observation}
We now state a Lemma which gives a useful bound on $u_{i+1}(x)-u_i(x)$ if no bad growth event occurs:
\begin{lemma}\label{Lemma:bound_u}
    If $N_{L+t/2}^{\bfxi^\xi_s}(x,t)\circ \theta_s=0$ and $s+t\leq u_i^\xi(x)\leq s+L$, then $v_i^\xi(x)=\infty$ or $u_{i+1}^\xi(x)-u_i^ \xi(x)\leq \kappa t$.
\end{lemma}
\begin{proof}[Proof (adapted from {\cite[Lemma~13]{garet2012shape}})]By definition, there exists a $\szero$-infection path from $(\zero,0)$ to $(x,u_i(x))$, which we call $g_i:[0,u_i(x)]\to V$ with $g_i(0)=\zero$ and $g_i(u_i(x))=x$. Moreover, at time $u_i(x)$, an infection event occurs at $x$, since, by construction, $v_{i-1}(x)\neq u_i(x)$. (Note that here we need the modified definition of $u_i(x)$ compared to \cite{garet2012shape} and \cite{deshayes2014contact}.)
Together with the assumptions $s+t\leq u_i(x)\leq s+L$ and $N_{L+t/2}^{\bfxi_s}(x,t)\circ \theta_s=0$ we can conclude $\1_{E^{x,\bfxi_{u_i(x)}}(x,t)}\circ\theta_{u_i(x)}=0$ which guarantees $\tau^{x,\szero}\circ\theta_{u_i(x)}=\infty$ or $\tau^{x,\szero}\circ\theta_{u_i(x)}\leq t/2$. This is implied by the third set in the definition of $E^{y,\xi}$. The first case implies $v_i(x)=\infty$ and we are done, the second one implies $v_i(x)-u_i(x)\leq t/2$.

Let us define $x_0=g_i(u_i(x)-t)$ as the vertex that the infection path $g_i$ visits at time $u_i(x)-t$. Next, we will prove that $x_0\in \B_{Mt+2}(x)$. For the sake of contradiction, assume $x_0\notin \B_{Mt+2}(x)$ and let $t_1$ be the first time after $u_i(x)-t$ where the infection path enters  $\B_{Mt+2}(x)$ at a vertex, say $x_1$. In particular, this implies $||x-x_1||= Mt+1$. 
Again, by our assumptions, we can conclude that the event $E^{x_1,\bfxi_{t_1}}(x,t)\circ\theta_{t_1}$ does not occur, because $t_1$ is a possible infection time for $x_1$. Hence, $\bfH_t^{x_1,\bfxi_{t_1}}\circ \theta_{t_1}\subset B_{Mt}(x_1)$, which implies that the infection of $x$ from $x_1$ needs more time than $t$, i.e.\ $t_1+t<u_i(x)$, which contradicts $t_1\geq u_i(x)-t$.
Thus, $x_0\in \B_{Mt+2}(x)$.

Let $t_2$ be the first time after $u_i(x)-t$ where some infection or recovery symbol at $x_0$ occurs.  By Observation \ref{observation:distance_recovery} we have $t_2-(u_i(x)-t)< t/2$. Clearly, the infection does not die at time $t_2$, and we set $x_2=g_i(t_2)$. This yields that the infection of $x$ from $(x_2,t_2)$ survived more than $t/2$ time units and $\tau^{x_2,\bfxi_{t_2}}>t/2$. Moreover, $x_2\in \B_{Mt+2}(x)$, by exactly the same argument as for $x_0$. Exploiting once again the non-occurrence of $E^{x_2,\bfxi_{t_2}}(x,t)\circ\theta_{t_2}$, we obtain $\tau^{x_2,\bfxi_{t_2}}=\infty$, which is implied by the fourth set in the definition of $E^{y,\xi}$, and 
\begin{equation*}
    \inf\{s\geq 2t:x\in \bfeta_s^{x_2,\bfxi_{t_2}}\}\circ \theta_{t_2}\leq \gamma t.
\end{equation*}
Together with $x_2\in \bfeta_{t_2}^\origin$ we can conclude that there exists some $t_3\in[t_2+2t,t_2+\gamma t]$ with $x\in \bfeta_{t_3}^\origin$. Now, we can bound
\begin{equation*}
    t_3\geq t_2 +2t\geq u_i(x)-t +2t\geq u_i(x)+t\geq v_i(x)+t/2,
\end{equation*}
where we used $v_i(x)-u_i(x)\leq t/2 $ in the last inequality. Moreover, $v_i(x)-u_i(x)\leq t/2 $ also implies that
$v_i(x)\leq s +L +t/2$, which yields together with our Observation \ref{observation:distance_recovery} that there exists some $t_4\in[v_i(x), v_i(x)+t/2]$ with $x\notin\bfeta^\origin_{t_4}$ and, in particular,
\begin{equation*}
    l_i(x)\leq v_i(x)+t/2\leq t_3. 
\end{equation*}
Therefore, by definition $u_{i+1}(x)\leq t_3$, and finally
\begin{equation*}
    u_{i+1}(x)-u_{i}(x)\leq t_3-u_i(x)\leq t_2+\gamma t -u_i(x) \leq \gamma t.\qedhere
\end{equation*}
\end{proof}
We bound the probability that no event of bad growth occurs in a space-time box.
\begin{lemma}\label{Lemma:exp_control_N}
    There exists $A,B>0$ such that for every $\xi$ we have
    \begin{equation*}
        \IP(N_L^{\xi}(x,t)\geq 1)\leq A (1+L)\exp(-Bt)\quad\text{for all } t>0,~x\in V\text{ and }L>0.
    \end{equation*}
\end{lemma}
\begin{proof}
    We first show that there exists constants $A,B>0$ such that for every $\xi$, $x\in V$ and $y\in \B_{Mt+2}(x)$ we have
    \begin{equation}\label{eq:bad_growths_estimate}
        \IP(E^{y,\xi}(x,t))\leq A\exp(-Bt)\quad\text{for all }t\geq0.
    \end{equation}
    First note that 
    \begin{equation*}
        \IP(\Delta_y^\xi[0,t/2)=\emptyset)=\IP(y\in \bfeta^{y,\xi}_s~\forall s\in[0,t/2)),
    \end{equation*}
    which can be controlled uniformly in $\xi$ by \eqref{ConjectureTEC}.
    Moreover, the probabilities of the events $\{\bfH_t^{y,\xi}\not\subset B_{Mt}(y)\}$, $\{t/2<\tau^{y,\szero}<\infty\}$ and $\{t/2<\tau^{y,\xi}<\infty\}$ can be controlled by \eqref{ConjectureAML} and \eqref{ConjectureSC}. The control of the probability of the last term $\{\tau^{y,\xi}=\infty,\inf\{s\geq 2t:x\in \bfeta_s^{y,\xi}\}>\gamma t\}$ can be obtained as in \cite[Lemma 14]{garet2012shape}, which yields \eqref{eq:bad_growths_estimate}. 

    As in the proof of \cite[Lemma 14]{garet2012shape}, we use that $\Delta_y+\sum_{e:y\in e}\Delta_e$ is a Poisson point process with intensity $r_{\max}+ 2d \lambda_{\max}$ for every $y\in V$, where $\lambda_{\max}$ and $r_{\max}$ denote the maximal infection and recovery rates, to show that 
    \begin{equation*}
        \IE[N_L^{\xi}(x,t)]\leq |\B_{Mt+2}|(2+L(r_{\max}+ 2d \lambda_{\max}))A\exp(-Bt).
    \end{equation*}
    Finally, the claim follows by the Markov inequality and the polynomial bound on the cardinality of $\B_{Mt+2}$.
\end{proof}
\begin{lemma}\label{lemma:set_inklusion}
    For any initial time $s>0$, any scale $t>0$ and every $x\in V$ the following inclusion holds almost surely:
    \begin{align}
        \{\tau^{\origin,\xi}=\infty\}\cap\{\exists &z\in \B_{Mt+2}(x): \tau^{z,\bfxi^\xi_s}\circ \theta_s=\infty, z\in \bfeta_s^{\origin,\xi}\}\label{event1}\\
        &\cap\{N^{\bfxi^{\xi}_s}_{K^\xi(x)\kappa t}(x,t)\circ \theta_s=0\}\label{event2}\\
        &\cap \bigcap_{1\leq i< K^\xi(x)}\{v_i^\xi(x)-u_i^\xi(x)<t/2\}\label{event3}\\
        &\subseteq \{\tau^{\origin,\xi}=\infty\}\cap \{\sigma^\xi(x)\leq s +K^\xi(x)\kappa t\}.\label{event4}
    \end{align}
\end{lemma}
\begin{proof}[Proof (adapted from {\cite[Lemma~15]{garet2012shape}})]
    \begin{equation*}
        i_0:=\max\{i:u_i(x)\leq s+t\}.
    \end{equation*}
    If for every finite $u_i(x)$ it holds that $u_i(x)\leq s+t$, we are done since then $\sigma(x)=u_{K(x)}(x)\leq s+t\leq s+K(x)\kappa t$. Thus, we can assume that $u_{i_0+1}(x)$ is finite and in particular $v_{i_0}(x)<\infty$ as well as $i_0<K(x)$. Therefore, the event in \eqref{event3} gives us  $v_{i_0}(x)-u_{i_0}(x)<t/2$, which implies $v_{i_0}(x)< s+\frac{3}{2}t$ by definition of $i_0$. Moreover, the event in \eqref{event2} ensures the existence of an effective recovery symbol for site $x$ between the times $v_{i_0}(x)$ and $v_{i_0}(x)+ t/2< s+2t$. Thus, in particular $l_{i_0+1}(x)\leq s+2t$.
    
    Let  $z\in \B_{Mt+2}$ be the \textit{source point} satisfying \eqref{event1}. The event in \eqref{event2} implies that $E^{z,\bfxi_s}(x,t)\circ \theta_s$ does not occur. Together with $\tau^{z,\bfxi_s}\circ \theta_s=\infty$, this yields that
    \begin{equation*}
        \inf\{u\geq 2t: x\in \bfeta_u^{z,\bfxi_s}\}\circ \theta_s\leq \kappa t
    \end{equation*}
    holds, which, along with the observation about $l_{i_0+1}(x)$, leads to $u_{i_0+1}(x)\leq s+\kappa t$.
    Note, by definition of $i_0$ we have for all $j\geq 1$ that $u_{i_0+j}(x)> s+t$. Moreover, we have for all $i<K(x)$ that $v_i(x)<\infty$ by definition of $K(x)$. Applying Lemma~\ref{Lemma:bound_u} recursively and using the event in \eqref{event2}, we obtain
    \begin{equation*}
        u_{i_0+j}\leq s +j\kappa t\quad\text{for all }j\in\{1,\dots,K(x)-i_0\}.
    \end{equation*}
    In particular, we have $\sigma(x)=u_{K(x)}(x)\leq s+(K(x)-i_0)\kappa t\leq s+K(x)\kappa t$.
\end{proof}
We directly apply this lemma to derive a technical result, which we will need later.
\begin{corollary}\label{corollary:ess_expmoments}
	There exist constants $A,B>0$ such that for all $\xi\in\loc^{V\cup E}$ we have that
	\begin{equation*}
		\PC(\sigma^\xi(\origin)>t)\leq A\exp{(-B\sqrt{t})}\quad\text{for all }t\geq 0.
	\end{equation*}
\end{corollary}
	\begin{proof}
		Let $\xi$ be arbitrary and note that for $t\geq 1$
		\begin{equation*}
			\PC(\sigma^\xi(\origin)>2t^2)\leq \PC(\sigma^\xi(\origin)>t+K^\xi(\origin)\kappa t, K^\xi(\origin)\kappa\leq t)+\PC(K^\xi(\origin)\kappa >t).
		\end{equation*}
		The second term can be bounded by $A\exp{(-Bt)}$ for some positive constants $A,B>0$ appealing to Lemma~\ref{lemma:K_as_finite}. For the first term we apply the contraposition of Lemma~\ref{lemma:set_inklusion} with initial time $t>0$ and scale $t>0$ at $\origin$ to derive
		\begin{align*}
			 \PC(\sigma^\xi(\origin)>t+K^\xi(\origin)\kappa t, K^\xi(\origin)\kappa\leq t)&\leq \PC\big(\forall z\in \B_{Mt+2}\cap\bfeta_t^{\origin,\xi}: \tau^{z,\bfxi_t}\circ \theta_t < \infty\big)\\
			 &+\PC\big(N^{\bfxi_t}_{t^2}(\origin,t)\circ \theta_t\geq1\big)\\
			 &+\PC\bigg(\bigcup_{1\leq i< K^\xi(\origin)}\{v_i^\xi(\origin)-u_i^\xi(\origin)\geq t/2\}\bigg).
		\end{align*}
		The second and third term can be controlled by Lemma~\ref{Lemma:exp_control_N} and Lemma~\ref{Lemma:exp_control_diff_vu}, respectively. For the first term, observe that
		\begin{equation*}
			\{\tau^{\xi,\origin}=\infty\}\cap \{\forall z\in \B_{Mt+2}\cap\bfeta_t^{\origin,\xi}: \tau^{z,\bfxi_t}\circ \theta_t < \infty\}\subset \{\bfH_t\not\subseteq\B_{Mt+2}\}
		\end{equation*}
		and the desired bound follows by \eqref{ConjectureAML}.
	\end{proof}

Finally, we can control the error term $r^\xi(x,y)$, which describes the \textit{lack of subadditivity}, by estimating the tail of its distribution and bounding its moments.
\begin{lemma}\label{Lemma_lackSubadd}
    There exists constants $A,B>0$ such that for every $x,y\in V$ and all $\xi$
    \begin{equation}\label{eq:exp_bound_r}
        \PC(r^\xi(x,y)\geq t)\leq A\exp(-B\sqrt{t}) \quad \text{ for all } \quad t\geq 0.
    \end{equation}
\end{lemma}
\begin{proof} According to Remark~\ref{remark:initial_cond} it suffices to consider the case $\xi=\szero$. This can be proven analogously to the proof of \cite[Theorem 2]{garet2012shape} using the estimates from Lemma~\ref{Lemma:exp_control_diff_vu} and Lemma~\ref{Lemma:exp_control_N} and the contraposition of Lemma~\ref{lemma:set_inklusion} at a source point $x+y$, which is infected at time $s:= \sigma^{\szero}(x)+\sigma^{\szero}(y)\circ\hat{\theta}^{\szero}_x$. Therefore, it is crucial that this source point $x+y$ at $s$ exists, in \cite{garet2012shape} this follows by monotonicity of the process. However, we do not necessarily have this property, and thus we need to guarantee its existence by other means. This is exactly where the worst-case monotonicity (Assumption~\ref{Ass:WS-Monotonicity}) comes again into play, which does provide this.
\end{proof}
\begin{corollary}\label{corollary:p-moment_rest_term}
    For any $p\geq 1$ there exists $M_p>0$ such that
    \begin{equation*}
        {\IE}^{\xi}[(r^{\xi}(x,y)^+)^p]\leq M_p\quad\text{for all }x,y\in V~\text{and }\xi.
    \end{equation*}
\end{corollary}
\begin{proof}Follows from Lemma~\ref{Lemma_lackSubadd} exactly as in \cite[Corollary 16]{garet2012shape} with Fubini-Tonelli.
\end{proof}
\subsection{Control of $\sigma^{\szero}(x)-t^{\szero}(x)$}
We now want to bound the difference between the essential hitting time $\sigma^{\szero}$ and the first hitting time $t^{ \szero}(x)$. We tackle the problem in a similar fashion as before, by first defining new boxes of bad growth and counting such boxes. However, the control will not be uniform in $x$ but depend on the distance of $x$ from the origin. This comes from the fact that we do not have a source point of infinite progeny close to $x$ as we had for the lack of subadditivity (cmp. Figure \ref{figure:2_problems}). For $x,y\in V$, $t>0$ and $L>0$ let
    \begin{align*}
        \hat{E}^{y,\xi}(t)&=\bigg\{\tau^{y,\xi}<\infty,\bigcup_{s\geq 0}\bfH_s^{y,\xi}\not\subset B_{Mt}(y)\bigg\},\\
        \hat{N}_L(x,t)&=\sum_{y\in B_{Mt+1}(x)}\int_0^L\1_{\hat{E}^{y,\bfxi_s}(t)}\circ\theta_s~d\bigg(\sum_{e:y\in e}\Delta_e\bigg)(s).
    \end{align*}

  One easily verifies:
  \begin{corollary}\label{corollary:exp_bound_N_hat}
      There exist constants $A,B>0$ such that
        \begin{equation}\label{eq:n_hat_exp}
        \PO(\hat{N}_L(x,t)\geq 1)\leq A (1+L)\exp{(-Bt)}\quad\text{for all }x\in V,~t\geq 0.
    \end{equation} 
  \end{corollary}
  \begin{proof}
 Analogous to the proof of Lemma~\ref{Lemma:exp_control_N}, one first shows the existence of constants $A',B'>0$ such that
    \begin{equation*}
        \IP(\hat{E}^{y,\xi}(t))\leq A' \exp{(-B't)}\quad\text{for all }y\in V,~t\geq 0 ~\text{and}~\xi.
    \end{equation*}
    Then, the result follows by estimating the expected value of $\hat{N}_L(x,t)$, where the strong Markov property is required.
    \end{proof}
    
\begin{lemma}\label{lem:limit_diff_hit_ess}
    We have $\PC$-almost surely
    \begin{equation*}
        \lim_{||x||\to\infty}\frac{|s^\xi(x)-t^{\origin,\szero}(x)\circ\hat{\theta}^\xi_0|}{||x||}=0.
    \end{equation*}
\end{lemma}
\begin{proof}
    Analogous to \cite[Proposition 17]{garet2012shape} one can show that there exist constants $A,B,\alpha>0$ such that for every $x\in V$ and $u>0$
    \begin{equation*}
        \PO(\sigma(x)\geq t(x)+K(x)(\alpha\log(1+||x||)+u))\leq A\exp{(-B u)}.
    \end{equation*}
    Then one can straightforwardly adapt the proof of \cite[Lemma 18]{garet2012shape} to show that for every $p>0$ there exists some constant $C_p>0$ such that
    \begin{equation*}
        \bar{\IE}(|\sigma^{\szero}(x)-t^{\origin,\szero}(x)|^p)\leq C_p(\log(1+||x||))^p.
    \end{equation*}
    Finally, the statement of the lemma can be shown for $\xi=\szero$ as in \cite[Corollary 19]{garet2012shape}. For general $\xi$ the result follows by applying Corollary~\ref{Corollary:invariance_2}. 
\end{proof}
\subsection{Proof of Theorem~\ref{Conjecture1}}
From our previous results we can now deduce the missing requirements to apply Theorem~\ref{theorem:deshayes_as_shape}.
\begin{corollary}\label{corollary:ALL_essential}
    There exist constants $A,B,C>0$ such that for all $t>0$ and all $x\in V$
    \begin{equation*}
        \PC(s^\xi(x)>C ||x||+t)\leq A\exp(-B\sqrt{t}).
    \end{equation*}
\end{corollary}
\begin{proof}
    The case $\xi=\szero$ follows as in \cite[Corollary 20]{garet2012shape}. For the general case the claim follows by Corollary~\ref{Corollary:invariance_2}(2) since $s^\xi$ under $\PC$ is distributed as $\sigma^{\szero}=s^{\szero}$ under $\PO$.
\end{proof}
We further have that all moments of the essential hitting time exist.
\begin{corollary}\label{corollary:p-moment}
    For every $p\geq 1$ there exists a constant $C'_p>0$ such that
    \begin{equation*}
        \Bar{\IE}^\xi[s^\xi(x)^p]\leq C'_p(1+||x||)^p.
    \end{equation*}
\end{corollary}
\begin{proof}
    This is again an immediate consequences of Corollary \ref{corollary:ALL_essential} and can be deduced by applying the Minkowski inequality and Fubini-Tonelli. See \cite[Corollary~21]{garet2012shape}.
\end{proof}

Lastly, before proving our theorem, we clarify that assumption \eqref{prop:4} also holds. Let $C$ be the constant from Corollary \ref{corollary:ALL_essential}.
\begin{lemma}\label{Lemma:uniform_continuity_essh}
     For every $\varepsilon>0$, $\PC$-almost surely, there exists some $R>0$ such that 
    \begin{equation}\label{eq:continuity_like}
        \big(||x||\geq R\text{ and }||x-y||\leq\varepsilon ||x||\big)\Rightarrow |s^{\xi}(x)-s^{\xi}(y)|\leq 3C\varepsilon||x||\quad\forall x,y\in \IZ^d.
    \end{equation}
\end{lemma}
\begin{proof} Analogously as in \cite[Lemma 26]{garet2012shape} it can be shown that \eqref{eq:continuity_like} holds for $\xi=\szero$. Then the claim follows by Corollary~\ref{Corollary:invariance_2}(2).
\end{proof}

Having prepared all ingredients needed we start by showing our asymptotic shape result for the case $\xi=\szero$.
\begin{theorem}\label{ShapeTheoremForZeroConfig}
     There exists a norm $\mu:\IZ^d\to\IR^+$ such that for every $\varepsilon>0$
\begin{equation*}
	\PO\big(\exists s\geq0 :  t(1-\varepsilon)\IB_{\mu}\subset \bfH_t^{\szero}\subset t(1+\varepsilon)\IB_{\mu} \,\,\forall t\geq s\big)=1,
\end{equation*}
where $\B_\mu$ is the unit ball with respect to $\mu$.
\end{theorem}
\begin{proof}We first show the result for $\sigma(\,\cdot\,)$, i.e.\ that there exists a norm $\mu:\IZ^d\to\IR^+$ such that 
    \begin{equation*}
	\PO\big(\exists s\geq0 :  t(1-\varepsilon)\IB_{\mu}\subset \bfG_t^{\szero}\subset t(1+\varepsilon)\IB_{\mu} \,\,\forall t\geq s\big)=1,
\end{equation*}
with $\bfG_t^{\szero}:=\{x\in \IZ^d:\sigma^{\szero}(x)\leq t\}+\big[\unaryminus\tfrac{1}{2},\tfrac{1}{2}\big]^d$. Thus, we only have to check that the assumptions of Theorem \ref{theorem:deshayes_as_shape} are satisfied under the probability measure $\PO$ for the essential hitting times $(\sigma^{\szero}(x))_{x\in \IZ^d}$, the error terms $(r^{\szero}(x,y)^+)_{x,y\in \IZ^d}$ and $(v(x,y))_{x,y\in \IZ^d}$ with $v(y,x)=\sigma^{\szero}(y)\circ\hat{\theta}^{\szero}_x$. Clearly, by symmetry, $\sigma^{\szero}(x)$ and $\sigma^{\szero}(-x)$ have the same distribution and Corollary \ref{corollary:p-moment} guarantees finite second moment of the essential hitting times. The first hypothesis \eqref{prop:1} follows by the definition of $r^{\szero}(x,y)$ and Corollary \ref{Corollary:invariance_2}. The second property follows by Corollary \ref{corollary:p-moment_rest_term}. The at least linear growth of hypothesis \eqref{prop:3} is given by Corollary \ref{corollary:ALL_essential} and the at most linear growth (hypothesis \eqref{prop:5}) follows by the fact that $t^{\origin,\szero}(x)\leq\sigma^{\szero}(x)$ and the at most linear growth of the hitting time $t$, i.e\ assumption~\eqref{ConjectureAML}. Finally, hypothesis \eqref{prop:4} follows by Lemma \ref{Lemma:uniform_continuity_essh}.

Having established the shape theorem for the essential hitting time $\sigma^{\szero}(x)$ it follows for the first hitting time $t^{\origin,\szero}(x)$ by Lemma~\ref{lem:limit_diff_hit_ess} that
    \begin{equation}\label{eq:ConvOfHittingTimes}
        \lim_{||x||\to\infty}\frac{|\mu(x)-t^{\origin,\szero}(x)|}{||x||}=0 \qquad \PO\text{-almost surely.}
    \end{equation}
    Since by definition $\sigma^{\szero}(x)\geq t^{\origin,\szero}(x)$ it only remains to show that the upper inclusion is true. We show this by contradiction. 
    
    Assume that there exists an increasing sequence of times $(t_n)_{n\geq 1}$, with $t_n \to\infty$ and $t_{n}^{-1}\bfH_{t_n}\not\subset (1 + \varepsilon)\B_{\mu}$. This means that for any $n\geq 1$ there exists a $x_n\in V$ such that $t^{\origin,\szero}(x) \leq t_n$ and $\mu(x_n)>(1 + \varepsilon)t_n$. This implies that $\mu(x_n)> (1 + \varepsilon)t^{\origin,\szero}(x_n)$
    for all $n\geq 1$. Furthermore, the second inequality $\mu(x_n)>(1 + \varepsilon)t_n$ implies that $||x_n||\to \infty$, since $\mu$ is a norm. This results in a contradiction of \eqref{eq:ConvOfHittingTimes}.
\end{proof}
\begin{corollary}\label{corollar:ConvergenceForAllBackgrounds} For all $\xi\in\loc^{V\cup E}$ and all $x\in V$ it holds that
    \begin{equation}\label{ConvergenceForAllBackgrounds}
         \lim_{n\to\infty}\frac{s^{\xi}(nx)}{n}=\lim_{n\to\infty}\frac{\sigma^{\szero}(nx)\circ\hat{\theta}^{\xi}_{\origin}}{n}=\mu(x)
    \end{equation}
    almost surely with respect to $\PC$. This implies in particular that for every $\varepsilon>0$,
    \begin{equation*}
        \PC\big(\exists t_0\geq0 :  t(1-\varepsilon)\IB_{\mu}\subset \bfH_t^{\xi}\,\,\forall t\geq t_0\big)=1.
    \end{equation*} 
\end{corollary}
\begin{proof}
    We again use Theorem~\ref{theorem:deshayes_as_shape} with the shifted essential hitting time $s^\xi(x)$, the error terms $r^\xi(x,y)^+$ and $v(y,x)=\sigma^{\szero}(y)\circ\hat{\theta}^{\szero}_x\circ\hat{\theta}^\xi_\origin$. The assumptions \eqref{prop:1}--\eqref{prop:4} as well as the symmetry and finite second moments of $s^\xi(x)$ follow as in Theorem~\ref{ShapeTheoremForZeroConfig}. 
    Hence, we can deduce the asymptotic shape for $\bfG_t^\xi$ and the existence of some norm $\mu^\xi$ with
    \begin{equation*}
         \lim_{n\to\infty}\frac{s^{\xi}(nx)}{n}=\lim_{n\to\infty}\frac{\sigma^{\szero}(nx)\circ\hat{\theta}^{\xi}_{\origin}}{n}=\mu^\xi(x)
    \end{equation*} $\PC$-almost surely. However, by Theorem~\ref{theorem:deshayes_as_shape}, the convergence also holds in $L^1$, which implies that $\mu^{\xi}=\mu$ for all $\xi$.
    
    Finally, to prove the second claim, let $\varepsilon\in(0,1)$ be arbitrary, but fixed and set $\tilde{\varepsilon}=\varepsilon/2$. Recall the definition $\tilde{s}^\xi(x):=s^\xi(x)+\sigma^\xi(\origin)$ and consider the set 
    \begin{equation*}
	\tilde{\bfG}_t^{\xi}:=\{x\in V:\tilde{s}^{\xi}(x)\leq t\}+\big[\unaryminus\tfrac{1}{2},\tfrac{1}{2}\big]^d.
    \end{equation*}
    For $t\geq \sigma^\xi(\origin)$, we have $x\in \tilde{\bfG}_t^{\xi}$ if and only if $x\in {\bfG}_{t-\sigma^\xi(\origin)}^{\xi}$ by definition. Moreover, by construction we have $t^{\origin,\xi}\leq \tilde{s}^{\xi}(x)$, which implies that $\tilde{\bfG}_t^\xi\subseteq \bfH_t^\xi$. Note that by Corollary~\ref{corollary:ess_expmoments} we have $\PC(\exists t_0>0: \sigma^\xi(\origin)<t_0)=1$ implying that
    \begin{align*}
    	1&=\PC(\exists t_0>0: (1-\tilde \varepsilon)t\B_\mu\subset \bfG_t^\xi~\forall t\geq t_0, \sigma^\xi(\origin)<t_0)\\
    	&\leq \PC(\exists t_0>0: (1-\tilde \varepsilon)(t-\sigma^\xi(\origin))\B_\mu\subset \bfG_{t-\sigma^\xi(\origin)}^\xi\subset\bfH_t^\xi~\forall t\geq 2t_0, \sigma^\xi(\origin)<t_0)\\
    	&\leq \PC(\exists t_0>0: (1-\tilde \varepsilon)^2 t\B_\mu\subset \bfH_{t}^\xi~\forall t\geq \frac{ 2t_0}{\tilde{\varepsilon}}, \sigma^\xi(\origin)<t_0)\\
    	&\leq \PC(\exists t_0>0: (1- \varepsilon) t\B_\mu\subset \bfH_{t}^\xi~\forall t\geq t_0).\qedhere
    \end{align*}
\end{proof}
\begin{proposition}\label{ub_all_bg}
    Let $\mu:\IZ^d\to\IR^+$ be given by \eqref{ConvergenceForAllBackgrounds}. Then it follows that for every $\xi\in\loc^{V\cup E}$ and every $\varepsilon>0$ it holds that
    \begin{equation*}
	    \IP\big(\exists s\geq0 :  \bfH_t^{\xi}\subset t(1+\varepsilon)\B_\mu \,\,\forall t\geq s\big)=1.
    \end{equation*}
\end{proposition}
\begin{proof}
    We proceed in two steps. First, we show that the claim is true if the background is started stationary, i.e.\ for every $\varepsilon>0$ we have
    \begin{equation}\label{eq:upperbound_stat}
            \IP\big(\exists t_0\geq0 :  \bfH_t^{\pi}\subset t(1+\varepsilon)\B_\mu \,\,\forall t\geq t_0\big)=1.
    \end{equation}
    Afterwards, we use the stationary case as a reference point to show the claim for arbitrary initial background  configurations $\xi\in \loc^{V\cup E}$.

    To show \eqref{eq:upperbound_stat} let $\varepsilon>0$ be arbitrary but fixed. For every $s\geq 0$ let $C_s:=\{\Delta_\origin[0,s)=\emptyset\}$ denote the event that no recovery occurs up to time $s$ for the origin. Furthermore, let
       \begin{equation*}
            A(s,\varepsilon):=\{ \exists t_0: \bfH_t^{\xi^{\szero}_s}\circ \theta_s\subset t(1+\varepsilon)\B_\mu~\forall t\geq t_0\}.
    \end{equation*}
    Note that $\bfH_t^{\xi^{\szero}_s}\circ \theta_s$ is the set of ever infected sites of the restarted process at time $s$ with initial configuration $(\{\origin\},\bfxi_s^{\szero})$ and $A(s,\varepsilon)$ is the event that this set is eventually contained within the cone $t(1+\varepsilon)\B_{\mu}$ starting with its tip at $s$. We already know by Theorem~\ref{ShapeTheoremForZeroConfig} that $\IP(A(0,\varepsilon^*))=1$ for all $\varepsilon^*>0$, and thus, in particular, $\IP(A(0,\varepsilon/2))=1$. 

    On $C_s$ the origin is infected at time $s$, which implies that $C_s\cap A(0,\varepsilon/2)\subset A(s,\varepsilon)$. This can be shown using the graphical representation: if the origin is infected at time $s$, then the set of infected vertices of the process restarted at the origin at time $s$ (with initial background $\bfxi^{\szero}_s$) is always a subset of the infected vertices of the original process. 
    
    This fact now implies that
    \begin{equation*}
        \IP(A(s,\varepsilon)\cap C_s)=\IP(A(s,\varepsilon)\cap C_s\cap A(0,\varepsilon/2))=\IP(C_s\cap A(0,\varepsilon/2))=\IP(C_s).
    \end{equation*}
    Furthermore, $C_s$ and $A(s,\varepsilon)$ are independent, as they depend on disjoint parts of the graphical representation, which implies that
    \begin{equation*}
        \IP(A(s,\varepsilon))\cdot\IP(C_s)=\IP(A(s,\varepsilon)\cap C_s)=\IP(C_s).
    \end{equation*}
    Since $\IP(C_s)\in (0,1]$ for all $s>0$ it follows that $\IP(A(s,\varepsilon))=1$ for all $s>0$.

    Recall from Section~\ref{Sec:CritValPermanetlyCoupled} the maximal infection process $(\widetilde{\bfeta}^\eta_t)_{t\geq 0}$ whose infection rate is $\lambda_{max}$ and has initial state $\eta\subset V$. This processes has no recoveries and is coupled with the infection process via the graphical representation such that $\bfeta_t\subset \widetilde{\bfeta}_t$ for all $t>0$ if $\bfeta_0\subset \widetilde{\bfeta}_0$. Furthermore, recall the process $(\Phi_{t})_{t\geq0}$ from \eqref{eq:PermantlyCoupledNeighbourhood}. Now let $M>0$ be chosen as in Lemma~\ref{lem:InfContainedExponetialSpeed}
    and define the event
    \begin{equation*}
    D_s:=\{\widetilde{\bfeta}^\origin_t\subset \B_{Mt}\subset \Phi_t~\forall t\geq s\}.
    \end{equation*} 
    By Lemma~\ref{lem:InfContainedExponetialSpeed} we know, in particular, that $\IP(D_s)\to 1$ as $s\to \infty$
    and by definition of $D_s$ it follows that 
    \begin{equation*}
         \IP(\{ \exists t_0\geq 0: \bfH_t^{\pi}\circ \theta_s\subset t(1+\varepsilon)\B_\mu~\forall t\geq t_0\}\cap D_s)=\IP(A(s,\varepsilon)\cap D_s).
    \end{equation*}
Now this allows us to conclude for all $s\geq 0$ that
    \begin{align*}
    \IP\big(\exists t_0\geq0 :  \bfH_t^{\pi}\subset t(1+\varepsilon)\B_\mu \,\,\forall t\geq t_0\big)&=\IP\big( \exists t_0: \bfH_t^{\pi}\circ \theta_s\subset t(1+\varepsilon)\B_\mu~\forall t\geq t_0\big)\\
    &\geq \IP\big(\{ \exists t_0: \bfH_t^{\pi}\circ \theta_s\subset t(1+\varepsilon)\B_\mu~\forall t\geq t_0\}\cap D_s \big)\\
    &= \IP(A(s,\varepsilon)\cap D_s )=\IP(D_s),
\end{align*}
where we used the invariance of the graphical construction in the first equality. Taking the limit proves \eqref{eq:upperbound_stat}. 

To finish the proof fix for some arbitrary initial configuration $\xi\in\loc^{V\cup E}$ and $\varepsilon>0$. We define the event
    \begin{equation*}
        A^\xi(\varepsilon):=\{\exists t_0: \bfH_t^{\xi}\subset (1+\varepsilon) t \B_\mu ~\forall t\geq t_0\}
    \end{equation*}
and aim to show $\IP(A^\xi(\varepsilon))=1$.
Note that by the coupling via the graphical representation we have $\bfH_{t+s}^{\origin,\xi}\subseteq \bfH_t^{\widetilde{\bfeta}_s,\bfxi_s}\circ \theta_s$ for all $s,t\geq 0$ and therefore
\begin{equation}\label{eq:SplittingInMultipleStationaryCones}
    \begin{aligned}
        A^\xi(\varepsilon)\cap D_s&=\{\exists t_0: \bfH_{t+s}^{\xi}\subseteq (1+\varepsilon)(t+s)~\B_\mu ~\forall t\geq t_0\}\cap D_s\\
        &\supseteq \{\exists t_0: \bfH_t^{\widetilde{\bfeta}_s,\bfxi_s}\circ \theta_s\subseteq (1+\varepsilon)(t+s)~\B_\mu ~\forall t\geq t_0\}\cap D_s\\
        &\supseteq \{\exists t_0: \bfH_t^{\B_{Ms},\bfxi_s}\circ \theta_s\subseteq (1+\varepsilon)(t+s)~\B_\mu ~\forall t\geq t_0\}\cap D_s\\
        &\supseteq \{\exists t_0: \bfH_t^{x,\pi}\circ \theta_s\subseteq (1+\varepsilon/2)t~\B_\mu\circ T_x ~\forall t\geq t_0,\forall x\in \B_{Ms}\}\cap D_s.
    \end{aligned}
\end{equation}
For the last inclusion we used additivity and the fact that on $D_s$ the infection never leaves the permanently coupled region after time $s$. Thus, the initial background configuration $\xi$ at time $0$ does not have any influence on the behaviour of the process restarted at time $s$.

We already showed that \eqref{eq:upperbound_stat} holds for $\varepsilon/2$. Therefore, by translation invariance and additivity of the infection process we get that
\begin{equation*}
    \IP(\exists t_0: \bfH_t^{x,\pi}\circ \theta_s\subseteq (1+\varepsilon/2)t~\B_\mu\circ T_x ~\forall t\geq t_0,\forall x\in \B_{Ms})=1.
\end{equation*}
Using this fact and \eqref{eq:SplittingInMultipleStationaryCones} we can conclude that
\begin{equation*}
    \IP(A^\xi(\varepsilon))\geq \IP(A^\xi(\varepsilon)\cap D_s)=\IP(D_s).
\end{equation*}
Exploiting again the fact that $\IP(D_s)\to 1$ as $s\to \infty$, we obtain $\IP(A^\xi(\varepsilon))=1$, which completes the proof.
\end{proof}
\begin{proof}[Proof of Theorem~\ref{Conjecture1}]
Note, Corollary  \ref{corollar:ConvergenceForAllBackgrounds} and Proposition \ref{ub_all_bg} already imply that 
       \begin{equation*}
	   \PC\big(\exists s\geq0 :  t(1-\varepsilon)\B_{\mu}\subset \bfH_t^{\xi}\subset t(1+\varepsilon)\B_\mu \,\,\forall t\geq s\big)=1,
    \end{equation*} where $\B_\mu$ is the unit ball with respect to the norm $\mu$ defined in \eqref{ConvergenceForAllBackgrounds}. We therefore only need to show the lower bound for $\bfH_t^\xi\cap \PCI^{\xi}_t$ in a similar way as in \cite{garet2012shape}. 
Thus, let $\tilde{\bfG}_t^\xi$ be defined as before and set
		\begin{equation*}
			\hat{t}^\xi(x):=\inf\{t\geq 0: x\in \PCI^{\xi}_t\cap \tilde{\bfG}_t^\xi\}.
		\end{equation*}
		We first show that 
		\begin{equation}\label{eq:diff_coupled_times}
			\lim_{||x||\to\infty }\frac{|\hat{t}^\xi(x)-\mu(x)|}{||x||}=0\quad\text{$\PC$-almost surely}.
		\end{equation}
		By Corollary~\ref{corollar:ConvergenceForAllBackgrounds} and the fact that $\sigma^\xi(\origin)$ is $\PC$-almost surely finite, it suffices to prove that
		\begin{equation*}
			\lim_{||x||\to\infty }\frac{|\tilde{s}^\xi(x)-\hat{t}^\xi(x)|}{||x||}=0\quad\text{$\PC$-almost surely}.
		\end{equation*}
		By definition, we have $\tilde{s}^\xi(x)\leq \hat{t}^\xi(x)$
		 and therefore showing
		\begin{equation}\label{eq:ExponetialControlIntersectionTime}
			\PC\big(\hat{t}^\xi(x)-\tilde{s}^\xi(x)>t\big)\leq A\exp({-Bt})\quad\forall x\in V,t\geq 0,
		\end{equation} 
		for some constants $A,B>0$ is sufficient to prove \eqref{eq:diff_coupled_times}. 
		
		Let $x\in \IZ^d$ and $\xi\in \loc^{V\cup E}$ be fixed. We  first show that 
		\begin{equation*}
			\bfK_{\tilde{s}^{\xi}(x)+t}^\xi\supset x+\bfK_t^{\xi_x}\circ \hat{\theta}^{\szero}_x\circ \hat{\theta}_0^\xi,
		\end{equation*}
		where $\xi_x:=T_x(\bfxi_{\tilde{s}^{\xi}(x)})$. Let $z\in x+\bfK_t^{\xi_x}\circ \hat{\theta}^{\szero}_x\circ \hat{\theta}_0^\xi$. We first consider the case $z\notin \bfeta_{\tilde{s}^{\xi}(x)+t}^{\IZ^d,\xi}$. Clearly, by additivity $z\notin \bfeta_{\tilde{s}^{\xi}(x)+t}^{\origin,\xi}$ and therefore $z\in \bfK_{\tilde{s}^{\xi}(x)+t}^\xi$. We now consider $z\in \bfeta_{\tilde{s}^\xi(x)+t}^{\IZ^d,\xi}$. 
		The set inclusion $\bfeta_{\tilde{s}^\xi(x)+t}^{\IZ^d,\xi}\subset x+\bfeta_{t}^{\IZ^d,\xi_x}\circ \hat{\theta}^{\szero}_x\circ \hat{\theta}_0^\xi $ holds trivially for $t=0$ and the coupling of our graphical construction ensures that it holds for all $t\geq 0.$
		This, together with our initial assumption, implies that $z\in x+ \bfeta_{t}^{\origin,\xi_x}\circ \hat{\theta}^{\szero}_x\circ \hat{\theta}_0^\xi$. By definition of the space-time shifts and our graphical construction we furthermore get $z\in \bfeta_{\tilde{s}^\xi(x)+t}^{\origin,\xi}$ which proves the claim. Fixing $s\geq 0$ we therefore have
		\begin{equation}\label{eq:perm_couple}
			\bigcap_{t\geq s}\bfK_{\tilde{s}^\xi(x)+t}^\xi\supset x+ \bigcap_{t\geq s}\bfK_t^{\xi_x}\circ \hat{\theta}^{\szero}_x\circ \hat{\theta}_0^\xi
		\end{equation} and in particular
		$\PCI^{\xi}_{\tilde{s}^\xi(x)+s}\supset x+\PCI^{\xi_x}_s\circ \hat{\theta}^{\szero}_x\circ \hat{\theta}_0^\xi$. Together with the shift invariance of Corollary~\ref{Corollary:invariance_2} and the fact that, by definition, $x\in\tilde{\bfG}_{\tilde{s}^\xi(x)+t}^\xi $ for all $t\geq 0$, we have 
		\begin{align*}
			\PC\big(\hat{t}^\xi(x) >\tilde{s}^\xi(x)+t\big)
			&=\PC\big(x\notin \PCI^{\xi}_{\tilde{s}^\xi(x)+t}\big)
			\leq
			\PO\big(\origin \notin \PCI^{\xi_x}_t\big).
		\end{align*} Finally, with assumption~\eqref{ConjectureFC}, we conclude \eqref{eq:ExponetialControlIntersectionTime}.
		
		By the fact that $\tilde{\bfG}_t^\xi\subseteq \bfH_t^\xi$, it remains to prove that
		\begin{equation*}
			\PC\big(\exists s\geq0 :  t(1-\varepsilon)\B_{\mu}\subset \tilde{\bfG}_t^{\xi}\cap\PCI^{\xi}_t \,\,\forall t\geq s\big)=1
		\end{equation*}
		for any $\varepsilon>0$. This follows similarly as in in the proof of Theorem~\ref{ShapeTheoremForZeroConfig}. Let $\varepsilon>0$ and assume for the sake of contradiction that there exists an increasing sequence of times $(t_n)_{n\geq 1}$, with $t_n \to\infty$ and $t_{n}(1 - \varepsilon)\B_{\mu}\not\subseteq\tilde{\bfG}_t^{\xi}\cap\PCI^{\xi}_t$. Thus, for any $n\geq 1$ there exists  $x_n\in V$ such that $\hat{t}^{\xi}(x_n)>t_n$ and $\mu(x_n)\leq (1-\varepsilon)t_n$. Combining both inequalities and exploiting that $\mu$ is a norm we get
		\begin{equation*}
			\lim_{n\to\infty}\frac{|\hat{t}^\xi(x_n)-\mu(x_n)|}{||x_n||}\geq \frac{\varepsilon}{1-\varepsilon}C>0
		\end{equation*}
		for some strictly positive constant $C>0$. However, this contradicts \eqref{eq:diff_coupled_times}.
\end{proof}
From the proof above we extract the following observation, which we will need in the next subsection.
\begin{corollary}\label{Corollary:x_couples}
	Under the assumptions of Theorem \ref{Conjecture1}, there exist constants $A,B,C>0$ such that     \begin{equation*}
		\PC\big(x \notin \PCI^\xi_{t+C||x||}\big)\leq A\exp{(-B\sqrt{t})}\quad\text{for all }\xi,~ t\geq 0~\text{and}~x\in V.
	\end{equation*}
\end{corollary}
\begin{proof}
	Let $C>0$ be the constant from Corollary~\ref{corollary:ALL_essential}.
	By \eqref{eq:perm_couple} we have, as above
	\begin{align*}
		\PC\big(x \notin \PCI^\xi_{3t+C||x||}\big)&\leq \PC\big(x \notin \PCI^\xi_{\tilde{s}^\xi(x)+t}\big)+\PC(\tilde{s}^\xi(x)>C||x||+2t)\\
		&\leq \PO\big(\origin \notin \PCI^{\xi_x}_t\big)+\PC(s^\xi(x)>C||x||+t)+\PC(\sigma^\xi(\origin)>t).
	\end{align*}
	Now equation \eqref{ConjectureFC},  Corollary~\ref{corollary:ALL_essential} and Corollary~\ref{corollary:ess_expmoments} give the desired controls.
\end{proof}

As mentioned in the introduction to this section, we conclude the section by briefly showing how Theorem~\ref{Conjecture1} also follows for permanent models.
\begin{lemma}
    Assume that $(\bfeta,\bfxi)$ is a permanent model, i.e.\ $\max_{j\in\loc_\pi} r(j)=0$. Then Theorem~\ref{Conjecture1} holds as well, in the sense that \eqref{ConjectureAML}--\eqref{ConjectureFC} imply \eqref{desired_result}.
\end{lemma}
\begin{proof} We first assume that $r(j)=0$ for         all $j\in\loc$. By definition, the hitting times $(t^{\szero}(x))_{x\in \IZ^d}$ are  subadditive in the sense that
    \begin{equation*}
        t^{\szero}(x+y)\leq t^{\szero}(x)+t^{\szero}(y)\circ T_x\circ \theta_{t^{\szero}(x)}.
    \end{equation*}
    Hence, we can apply Theorem~\ref{theorem:deshayes_as_shape} with $s(x)=t^\xi(x)$, $v(y,x)=t^{\szero}(x)+t^{\szero}(y)\circ T_x\circ \theta_{t^{\szero}(x)}$ and $r(x,y)=0$. Then \eqref{prop:1}--\eqref{prop:2} are fulfilled by definition. The conditions \eqref{prop:3} and \eqref{prop:5} follow from \eqref{ConjectureALL} and \eqref{ConjectureAML}, respectively. Lastly, \eqref{prop:4} follows as in \cite[Lemma 26]{garet2012shape}.
    This gives us the asymptotic shape theorem, i.e.\ \eqref{desired_result}, for $\bfH_t^\xi=\bfH_t^\xi\cap \PCI_t^\xi$, where the convex set $U^\xi$ may depend on the initial configuration $\xi$. It remains to show that this is not the case.
    
    By definition, $\bfH_t^{\szero}\subseteq\bfH_t^\xi$ and we have to verify that $\bfH_t^\xi\subset(1+\varepsilon)tU$ $\IP$-almost surely for all $t$ large enough, where $U:=U^{\szero}=\B_{\mu^{\szero}}$. However, this follows exactly as in the proof of Proposition~\ref{ub_all_bg}.
        
    For the general case, where recoveries might occur in transient background states, i.e.\ $\max_{j\in\loc_\pi} r(j)=0$ but $r(j)\neq0$ for some $j\in \loc$, we consider the process $(\overline{\bfeta},\bfxi)$ with rate functions $\overline{\lambda}\equiv\lambda$ and $\overline{r}\equiv0$.  Clearly, $(\overline{\bfeta},\bfxi)$ is still worst-case-monotone and one can verify that  \eqref{ConjectureAML}--\eqref{ConjectureFC} still hold. Therefore, by our previous reasoning, \eqref{desired_result} holds for $(\overline{\bfeta},\bfxi)$. By construction, $\bfeta_t\subseteq \overline{\bfeta}_t$ and  both processes $(\overline{\bfeta},\bfxi)$ and $(\bfeta,\bfxi)$ are identical if the initial configuration satisfies $\xi\in \loc_\pi^{V\cup E}$. Therefore \eqref{desired_result} holds for any initial background configuration $\xi\in \loc_\pi^{V\cup E}$ and the upper bound holds for all $\xi\in \loc^{V\cup E}$. For $s\geq0$ and $x\in \B_{Ms}$ define the event
    \begin{equation*}
    \tilde{D}_{s,x}:=\{\widetilde{\bfeta}^{x}_t\circ\theta_s\subset \B_{M(t+s)}\subset \Phi_{t+s}~\forall t\geq 0\},
    \end{equation*}
    where $M$ is the constant from \eqref{eq:InfectionContainedInCouplingRegion}. Then we have by translation invariance for all $s\geq0$, $x\in \B_{Ms}$ and $\varepsilon>0$ that
    \begin{equation}\label{eq:up_trans}
        \IP\big(\exists t_0>0:t(1-\varepsilon
        )\B_\mu+x\subseteq\bfH_t^{x,\bfxi_s}\circ\theta_s\big|\tilde{D}_{s,x}\big)=1.
    \end{equation}
    It remains to show that
    \begin{equation*}
        \PC\big(\exists t_0\geq0 :  t(1-\varepsilon)\IB_{\mu}\subset \bfH_t^{\xi}\,\,\forall t\geq t_0\big)=1\quad\text{for all } \varepsilon>0\text{ and }\xi\in\loc^{V\cup E}.
    \end{equation*}

    Thus, fix $\varepsilon>0$ and $\xi$. Note that for fixed $s>0$ and $x\in \B_{Ms}$  we have
    \begin{equation}\label{eq:subset_1}
        (t+s)(1-\varepsilon)\B_\mu\subseteq t(1-\varepsilon/2)\B_\mu +x\quad\text{for $t$ sufficiently large.}
    \end{equation}
     Furthermore, if $x\in\bfeta_s^\xi$ we have
    \begin{equation}\label{eq:subset_2}
        \bfH_t^{x,\bfxi_s}\circ\theta_s\subseteq \bfH_{t+s}^\xi.
    \end{equation}
    Recall the event $D:=\{\exists s\geq 0 :  \widetilde{\bfeta}^{\origin}_t\subset \B_{Mt}\subset \Phi_{t}~\forall t\geq s\}$. By \eqref{eq:InfectionContainedInCouplingRegion}, we have $\IP(D)=1$. Moreover, by definition, on the event $D$ exists some $s\geq0$ and $x\in \B_{Ms}$ such that $\tilde{D}_{s,x}$ occurs. We observe the following $\IP$-almost sure relations
\begin{equation*}
\begin{aligned}
    \{\tau^\xi=\infty\}&=\{\tau^\xi=\infty\}\cap D\\
    &=\bigcup_{s> 0}\bigcup_{x\in \B_{Ms}}\{x\in \bfeta_s^\xi\}\cap\{\tau^{x,\bfxi_s}\circ\theta_{s}=\infty\}\cap\tilde{D}_{s,x}\\
    &=\bigcup_{s> 0}\bigcup_{x\in \B_{Ms}}\{x\in \bfeta_s^\xi\}\cap\{\exists t_0>0: t(1-\varepsilon/2
        )\B_\mu+x\subseteq\bfH_t^{x,\bfxi_s}\circ\theta_s \forall t\geq t_0\}\cap\tilde{D}_{s,x},
\end{aligned}
\end{equation*}
where we used \eqref{eq:up_trans} in the third almost sure equality. Furthermore, since $x\in \B_{Ms} \cap \bfeta_s^\xi$ using \eqref{eq:subset_1} together with \eqref{eq:subset_2} implies that,  
\begin{equation*}
    t(1-\varepsilon/2)\B_\mu+x\subseteq\bfH_t^{x,\bfxi_s}\circ\theta_s \,\, \forall t\geq t_0\quad \Rightarrow \quad (t+s)(1-\varepsilon )\B_\mu\subseteq\bfH_{t+s}^{\xi}\,\, \forall t\geq t_0.
\end{equation*}
Finally, this yields that $\IP$-almost surely
\begin{align*}
    \{\tau^\xi=\infty\}\subset  \{\exists t_0\geq0 :  t(1-\varepsilon)\B_{\mu}\subset \bfH_t^{\xi}\,\,\forall t\geq t_0\}.
\end{align*}
This proves the desired claim.
\end{proof}

\subsection{Proof of Theorem~\ref{thm:ShapeTheoremNonMonoton}}
 This subsection is dedicated to proving Theorem~\ref{thm:ShapeTheoremNonMonoton} where we extend Theorem~\ref{Conjecture1} to the particular setup of a CPIU that is not worst-case monotone. In order to do this we start by showing that we can extend any CPIU $(\bfeta,\bfxi)$ to a worst-case monotone CPIU $(\breve\bfeta,\breve\bfxi)$ on an extended state space.
\begin{proposition}\label{prop:CPIU_worst_case}
	Let $(\bfeta,\bfxi)$ be a CPIU on $\{0,1\}^V\times \loc^{V\cup E}$. Then there exists a worst-case monotone CPIU $(\breve\bfeta,\breve\bfxi)$ on an extended state space $\{0,1\}^V\times \breve \loc^{V\cup E}$ with $\loc\subseteq \breve \loc$ such that almost surely for all $\eta\in \{0,1\}^V$ and $\xi\in \loc^{V\cup E}$ we have
	\begin{equation*}
		(\breve\bfeta_t^{\eta,\xi},\breve\bfxi_t^{\eta,\xi})=(\bfeta_t^{\eta,\xi},\bfxi_t^{\eta,\xi})\quad\text{for all }t\geq 0.
	\end{equation*}
\end{proposition} 
\begin{proof}
	Let $(\bfeta,\bfxi)$ be a CPIU with rate functions $\lambda(\cdot)$ and $r(\cdot)$ and maximal recovery rate $r_{\max}=\max_{j\in \loc}r(j)$. We aim to find a CPIU $(\breve\bfeta,\breve\bfxi)$ on an extended state space $\{0,1\}^V\times \breve{\loc}^{V\cup E}$ which is worst-case monotone and satisfies
	\begin{equation}\label{eq:CPIU_WS}
		(\breve\bfeta_t^{\eta,\xi},\breve\bfxi_t^{\eta,\xi})=(\bfeta_t^{\eta,\xi},\bfxi_t^{\eta,\xi})\quad\text{for all }t\geq 0
	\end{equation}
	almost surely for all $\eta\in \{0,1\}^V$ and $\xi\in \loc^{V\cup E}$.
	
	By relabelling our local state space $\loc$, we can assume $\loc=\{\{0\},\dots,\{N\}\}$ for some $N\in \IN$. Moreover, by definition, $(\bfxi(x))_{x\in V}$ and $(\bfxi(e))_{e\in E}$ are families of independent Markov processes on $\loc$ with transition rate matrices $(q_{i,j}^V)_{i,j\in \loc}$ and  $(q_{i,j}^E)_{i,j\in \loc}$, respectively. Thus, the entire process $\bfxi$ can be constructed with maps of the form
	\begin{align}\label{eq:def_flip_x}
		\mathbf{m}_{i,j}^x(\eta,\xi)(a):=
		\begin{cases}
			(\eta(x),\{j\})& \text{ if } x=a\text{ and }\xi(x)=\{i\}\\
			(\eta(a),\xi(a))& \text{ otherwise},
		\end{cases}
	\end{align}
	and
	\begin{align}\label{eq:def_flip_e}
		\mathbf{m}_{i,j}^e(\eta,\xi)(a):=
		\begin{cases}
			(\eta(e),\{j\})& \text{ if } e=a\text{ and }\xi(e)=\{i\}\\
			(\eta(a),\xi(a))& \text{ otherwise},
		\end{cases}
	\end{align}
	where $x\in V$, $e\in E$ and $i,j\in \loc$. The corresponding jump rates in the graphical construction are then given by $q_{i,j}^V$ and $q_{i,j}^E$, respectively. 
    
	Next, we extend our local state space $\loc$ to an extended local state space $\breve\loc$ with an hierarchical order. Therefore, we introduce $N+1$ levels, where the $k$-th level contains all subsets of $[N]$ with size $k$. Thus, for $k=1,\dots, N+1$, level $k$ has $\binom{N+1}{k}$ elements. In particular, level $N+1$, only contains the element $w:=[N]$ and the level $1$ contains all one-elementary subsets of $[N]$, i.e.\ the original state space $\loc$. Moreover, we define formally 
	\begin{equation*}
		\breve \loc=\{\{\ell_1,\dots,\ell_k\}\subset [N]: 1\leq k\leq N+1\}.
	\end{equation*}
	We now construct $\breve\bfxi$ by extending the maps in \eqref{eq:def_flip_x} and \eqref{eq:def_flip_e} to the extended local state space $\breve \loc$. Therefore, for $i,j\in [N]$ we define a function $s_{i,j}$ on $\breve \loc$ by $s_{i,j}(\{\ell_1,\dots,\ell_k\}):=\big(\{\ell_1,\dots,\ell_k\}\cup\{j\}\big)\setminus\{i\}$ and set
	\begin{align*}
		\breve{\mathbf{m}}_{i,j}^x(\eta,\breve\xi)(a):=
		\begin{cases}
			(\eta(x),s_{i,j}(\breve\xi(x)))& \text{ if } x=a\text{ and }i\in \breve\xi(x)\\
			(\eta(a),\breve\xi(a))& \text{ otherwise},
		\end{cases}
	\end{align*}
	as well as
	\begin{align*}
		\breve{\mathbf{m}}_{i,j}^e(\eta,\breve\xi)(a):=
		\begin{cases}
			(\eta(e),s_{i,j}(\breve\xi(e)))& \text{ if } e=a\text{ and }i\in \breve\xi(e)\\
			(\eta(a),\breve\xi(a))& \text{ otherwise}.
		\end{cases}
	\end{align*}
	Using the same rates $q_{i,j}^V$ and $q_{i,j}^E$ as before, we obtain a process $\breve\bfxi$ with transitions as sketched in Figure~\ref{figure:extendet_state}. Since $\breve{\mathbf{m}}_{i,j}^x$ and $\mathbf{m}_{i,j}^x$, as well as $\breve{\mathbf{m}}_{i,j}^e$ and $\mathbf{m}_{i,j}^e$, are identical on $\loc$, it follows that $\breve\bfxi$ coincides with $\bfxi$ on $\loc$. Moreover, by applying $s_{i,j}$ to $\breve\xi(a)$ with $a\in V\cup E$, the cardinality of $\breve\xi(a)$ either decreases by one or remains unchanged. Thus, by definition of the maps $\breve{\mathbf{m}}_{i,j}^x$ and $\breve{\mathbf{m}}_{i,j}^e$, we observe that the process $\breve\bfxi(a)$ can only jump to equal or lower levels and all local states which are not at level one are transient. Hence, $\breve{\bfxi}$ has the same unique invariant distribution $\pi$ as $\bfxi$. Furthermore, since $\breve\bfxi(a)$ still has a finite state space, it is geometrically ergodic \cite[cf.~Theorem~4.9]{levin2017Markov} and $\breve{\bfxi}$ is reversible if $\bfxi$ is reversible. Thus, Assumption~\ref{AssumptionBackground} is satisfied for $\breve\bfxi$.
	\begin{figure}
			\centering
			\scalebox{0.75}{
				\begin{tikzpicture}
					
					\node[circle, draw, minimum size=8mm] (b0) at (0,0) {$\{0,1,2\}$};
					
					\node[circle, draw, minimum size=8mm] (m0) at (-4,3) {$\{0,1\}$};
					\node[circle, draw, minimum size=8mm] (m1) at (0,3)  {$\{0,2\}$};
					\node[circle, draw, minimum size=8mm] (m2) at (4,3)  {$\{1,2\}$};
					
					\node[circle, draw, minimum size=8mm] (t0) at (-4,7) {$\{0\}$};
					\node[circle, draw, minimum size=8mm] (t1) at (0,7)  {$\{1\}$};
					\node[circle, draw, minimum size=8mm] (t2) at (4,7)  {$\{2\}$};
					
					\draw[->] (b0) to node[left] {$q_{2,1}+q_{2,0}$} (m0);
					\draw[->] (b0) to node[xshift=-1.5mm] {$q_{1,2}+q_{1,0}$} (m1);
					\draw[->] (b0) tonode[right] {$q_{0,1}+q_{0,2}$} (m2);
					
					\draw[->] (m0) to node[left] {$q_{1,0}$} (t0);
					\draw[->] (m0) to node[right,yshift=2mm,xshift=2mm] {$q_{0,1}$} (t1);
					\draw[->] (m0) to[bend right=35] node[below] {$q_{1,2}$} (m1);

					\draw[->] (m1) to node[left,yshift=2mm,xshift=-2mm] {$q_{2,0}$} (t0);
					\draw[->] (m1) to node[right,yshift=2mm,xshift=2mm] {$q_{0,2}$} (t2);
					\draw[->] (m1) to [bend right=35] node[below] {$q_{0,1}$} (m2);
					\draw[->] (m1) to [bend right=35] node[below]{$q_{2,1}$} (m0);
					
					\draw[->] (m2) to node[left,yshift=2mm,xshift=-2mm] {$q_{2,1}$} (t1);
					\draw[->] (m2) to node[right] {$q_{1,2}$} (t2);
					\draw[->] (m2) to[bend right=35] node[below] {$q_{2,0}$} (m1);
					
					\draw[->] (t0) to [bend right=35] node[below]{$q_{0,1}$} (t1);
					\draw[->] (t1) to [bend right=35] node[below]{$q_{1,0}$} (t0);
					\draw[->] (t1) to [bend right=35] node[below]{$q_{1,2}$} (t2);
					\draw[->] (t2) to [bend right=35] node[below]{$q_{2,1}$} (t1);
					\draw[->,dashed] (t2) to [bend right=90] node[below,xshift=1.5mm,yshift=-2mm]{$q_{2,0}$} (4.8,7.8);
					\draw[->,dashed] (m2) to [bend right=90] node[below,xshift=1.5mm,yshift=-2mm]{$q_{2,0}$} (4.8,3.8);
					\draw[->,dashed] (t0) to [bend left=90] node[below,xshift=-1.5mm,yshift=-2mm]{$q_{0,2}$} (-4.8,7.8);
					\draw[->,dashed] (m0) to [bend left=90] node[below,xshift=-1.5mm,yshift=-2mm]{$q_{0,2}$} (-4.8,3.8);
					\draw[draw=blue,fill=blue,fill opacity=0.1] (-5.5,5.7) rectangle (5.5,8);
					\node[red] (w) at (3,0)
					{$
						\begin{aligned}
							r(\cdot) &= r_{\max}\\
							\lambda(\cdot) &= 0
						\end{aligned}
						$};
					\draw[draw=red,fill=red,fill opacity=0.1] (-5.5,-1) rectangle (5.5,4);
					\node[red] (w) at (3,0)
					{$
						\begin{aligned}
							r(\cdot) &= r_{\max}\\
							\lambda(\cdot) &= 0
						\end{aligned}
						$};
					\node (l1) at (6.5,7) {\textbf{level~$1$}};
					\node (l2) at (6.5,3) {\textbf{level~$2$}};
					\node (l3) at (6.5,0) {\textbf{level~$3$}};
			\end{tikzpicture}}
			\caption{Transition graph for $\breve\bfxi(a)$ in case $|\loc|=3$. The blue area contains the states in $\loc$ while the red one the states in $\breve{\loc}\setminus \loc$. }
			\label{figure:extendet_state}
		\end{figure}
	To explain the intuition behind our state space $\breve \loc$, note that if the process $\breve\bfxi(a)$, started initially in state $w$, is currently in some state $\{\ell_1,\dots,\ell_k\}$, then the original process $\bfxi(a)$ is, independent of the initial configuration, in one of the states $\{\ell_1\},\dots,\{\ell_k\}$. In particular, if the process $\breve\bfxi(a)$, started at $w$, reaches level one, i.e.\ $\breve\bfxi_t(a)=\{j\}$ for some $j\in[N]$ and $t\geq 0$, then $a\in \Psi'_{t}$ and $\bfxi_t(a)=\{j\}$ independent of the initial configuration. On the other hand, if $a\in \Psi'_{t}$ for some $t\geq0$ then it follows that $\breve\bfxi_t(a)=\{j\}$ for some $j\in [N]$.
	
	For the process $(\breve\bfeta,\breve\bfxi)$ we only allow  infections between permanently coupled sites and each non permanently coupled site should recover with the maximal recovery rate $r_{\max}$. Therefore, we specify the rate function $\breve\lambda:\breve \loc^3\to [0,\infty)$
	\begin{equation*}
		\breve\lambda(\bfi_1,\bfi_2,\bfi_3)=\begin{cases}
			0\quad&\text{if}~|\bfi_1|+|\bfi_2|+|\bfi_3|>3,\\
			\lambda(\bfi_1,\bfi_2,\bfi_3) &\text{otherwise},
		\end{cases}
	\end{equation*}
	and $\breve r: \breve \loc\to [0,\infty)$
	\begin{equation*}
		\breve r (\bfj)=\begin{cases}
			r_{\max}\quad&\text{if}~ |\bfj|>1,\\
			r(\bfj) &\text{otherwise},
		\end{cases}
	\end{equation*}
	which ensures that our entire process $(\breve\bfeta,\breve\bfxi)$ is worst-case monotone and satisfies \eqref{eq:CPIU_WS}. 
\end{proof}
\begin{proof}[Proof of Theorem~\ref{thm:ShapeTheoremNonMonoton}]
    Let $(\bfeta,\bfxi)$ be a supercritical CPIU which satisfies \eqref{ConjectureAML}--\eqref{ConjectureTEC} and let $(\breve\bfeta,\breve\bfxi)$ be the worst-case monotone CPIU according to Proposition~\ref{prop:CPIU_worst_case}. By Theorem~\ref{thm:CritValueIndependence} it follows that the process $(\breve\bfeta,\breve\bfxi)$ is supercritical for all initial background configurations $\breve\xi$. In order to apply Theorem~\ref{Conjecture1} we have to verify that $(\breve\bfeta,\breve\bfxi)$ still satisfies \eqref{ConjectureAML}--\eqref{ConjectureTEC}.  
    The two conditions \eqref{ConjectureAML} and \eqref{ConjectureTEC} hold by construction. As before, let
    \begin{equation*}
        D_t:=\{\widetilde{\bfeta}^\origin_u\subset\B_{Mu}\subset \btilde\Phi_u \,\forall\, u\geq t\}
    \end{equation*} where $M>0$ is chosen according to Lemma~\ref{lem:InfContainedExponetialSpeed} such that $\IP(  D_t)>1- Ae^{-Bt}$ for some constants $A,B>0$ and all $t\geq 0$. 
    To prove \eqref{ConjectureSC}, we first note that 
    \begin{align*}
        \IP(\{2t\leq \btilde{\tau}^{\zero,\xi}<\infty\} \cap D_t )&\leq \IP(\{2t\leq \btilde{\tau}^{\btilde{\bfeta}^{\xi}_t,\btilde{\bfxi}^{\xi}_t}\circ \theta_t+t<\infty\}\cap D_t )\\
        &\leq \sup_{\xi\in \loc^{V\cup E}}\sup_{\eta\subset \B_{M_t}}\IP(t\leq \tau^{\eta,\xi}<\infty ).
    \end{align*}
    For the second inequality, we used the fact that on $D_t$ it suffices to consider configurations $\xi\in \loc^{V\cup E}$ and for such configurations we can replace $\breve\tau$ through $\tau$ by \eqref{eq:CPIU_WS}.
    By additivity in the infection set, we have for any $\eta \subset \B_{Mt}$ and any $\xi\in \loc^{V\cup E}$ that 
    \begin{align*}
        \IP(t\leq \tau^{\eta,\xi}<\infty )\leq \sum_{x\in \IB_{Mt}}\IP(t\leq \tau^{x,\xi}<\infty )\leq A\exp(-Bt)
    \end{align*}
    for some constants $A,B>0$ since assumption \eqref{ConjectureSC} holds for $\bfeta$ and $|\B_{Mt}|\in O((Mt)^d)$.
    
   We continue with proving \eqref{ConjectureALL}. Therefore, let $c$ be the constant such that $\eqref{ConjectureALL}$ holds for $\bfeta$. Then we have
    \begin{align*}
        &\IP\Big(\Big\{\btilde{t}^{\origin,\xi}(x)\geq \frac{||x||}{c}+\Big(\frac{M}{c}+2\Big)t\Big\} \cap \{\btilde{\tau}^{\origin,\xi}=\infty\}\cap D_t\Big)\\
        \leq &\IP\bigg(\bigcup_{y\in \B_{Mt}}\Big\{\btilde{t}^{y,\btilde{\bfxi}_t}(x)\circ\theta_t \geq \frac{||x||}{c}+\Big(\frac{M}{c}+1\Big)t\Big\} \cap \{\btilde{\tau}^{y,\btilde{\bfxi}^{\xi}_t}\circ \theta_t=\infty\}\cap D_t\bigg)\\
        \leq &\sum_{y\in \B_{Mt}}\sup_{\xi\in \loc^{V\cup E}} \IP\Big(\Big\{t^{y,\xi}(x) \geq \frac{||x||}{c}+\Big(\frac{M}{c}+1\Big)t\Big\} \cap \{\tau^{y,\xi}=\infty\}\cap D_t\Big)\\
        \leq &|\B_{Mt}| \sup_{y\in \B_{Mt}}\sup_{\xi\in \loc^{V\cup E}}\IP\Big(t^{\origin,\xi}(x-y) \geq \frac{||x||}{c}+\Big(\frac{M}{c}+1\Big)t,\tau^{\origin,\xi}=\infty\Big).
    \end{align*}
    Now we have for all $y\in \B_{Mt}$ that $||x||+Mt\geq||x-y||$ which gives the desired bound using property \eqref{ConjectureALL} for the original process $\bfeta$.
    
    We finish with proving a weaker statement than \eqref{ConjectureFC}, namely that there exist some constants $A,B>0$ such that for all $\xi\in\breve{\loc}^{V\cup E}$ we have
    \begin{equation}\label{Conj:tec3_alt}
        \IP(\origin \notin \btilde{\PCI}^{\xi}_{t},\btilde\tau^{\origin,\xi}=\infty)\leq A\exp(-B \sqrt{t})\quad\text{for all } t\geq 0.
    \end{equation}
    However, this condition is also adequate since the weaker bound of $A\exp{(-B\sqrt{t})}$ in \eqref{eq:ExponetialControlIntersectionTime} is sufficient to prove Theorem \ref{Conjecture1}.
    
    For any subset $\{\ell_1,\dots,\ell_k\}\subseteq [N]$, where $\ell_1,\dots,\ell_k$ are ordered according to the natural order, let $p:\breve \loc\to \loc$ be defined as $p(\{\ell_1,\dots,\ell_k\}):=\ell_1$. We extend $p$ to the configuration space $\breve \loc^{V\cup E}$ by setting $p(\xi)(a)= p(\xi(a))$ for $a\in V\cup E$.
 	
 	Let $C$ be the constant from Corollary~\ref{Corollary:x_couples} and $M>0$ be the constant from Lemma~\ref{lem:InfContainedExponetialSpeed} such that the probability of the event $D_t$ decays exponentially. For $t>0$ let $t_0:=\frac{t}{MC}$ and $\xi_{t_0}:=p(\btilde{\bfxi}_{t_0})$. By the definition of $D_{t_0}$, we have
    	\begin{align*}
    		&\{\zero \notin \btilde{\PCI}^{\xi}_{2t+t_0}\}\cap\{\tau^{\origin,\xi}=\infty\}\cap D_{t_0}\\
    		&\subseteq\{\exists x\in \B_{t/C}, \exists s\geq 2t:\tau^{x,\xi_{t_0}}\circ\theta_{t_0}=\infty,\origin\notin \bfeta_s^{x,\xi_{t_0}}\circ\theta_{t_0}, \origin\in \bfeta_s^{V,\xi_{t_0}}\circ\theta_{t_0}\}.
    	\end{align*}
    	Hence, using a union bound and translation invariance gives 
    	\begin{align*}
    		\IP\big(\zero \notin \btilde{\PCI}^{\xi}_{2t+t_0},\tau^{\origin,\xi}=\infty,D_{t_0}\big)
    		&\leq |\B_{t/C}|\sup_{\xi\in\loc^{V\cup E}}\sup_{x\in \B_{t/C}} \IP\big(\origin\notin \PCI^{x,\xi}_{2t},\tau^{x,\xi}=\infty\big)\\
    		&\leq |\B_{t/C}|\sup_{\xi\in\loc^{V\cup E}}\sup_{x\in \B_{t/C}} \PC\big(x\notin \PCI^{\xi}_{t+C||x||}\big).
    	\end{align*}
    	To derive the last inequality, we interchanged the roles of $x$ and $\origin$ and  used that $\PCI^{\xi}_{t+C||x||}\subset \PCI^{\xi}_{2t}$ for $x\in \B_{t/C}$.
    	Corollary~\ref{Corollary:x_couples}, together with the usual cardinality estimate for $\B_{t/C}$,  yields \eqref{Conj:tec3_alt}.
    \end{proof} 
\section{Asymptotic Shape of the CPDRE}\label{sec:ShowingAsympShapeCPDRE}
Throughout this section, we assume that $\loc=[N]$ and $[N]$ is equipped with the canonical total order $\leq$.
Now that we have established Theorem~\ref{Conjecture1} and Theorem~\ref{thm:ShapeTheoremNonMonoton} we are in the position to show the asymptotic shape results for the CPDRE. This section is dedicated to prove Theorem~\ref{thm:CPIUShapeTheorem} and Proposition~\ref{prop:CPDREShapeTheorem} by verifying the estimates \eqref{ConjectureAML}--\eqref{ConjectureFC}. The estimate \eqref{ConjectureAML} is a direct consequence of Lemma~\ref{lem:InfContainedExponetialSpeed}.
\begin{corollary}\label{corollary:AML}
     There exist constants $A,B,M>0$ such that for all $\xi$ and all $x\in V$
	\begin{align}
	\IP(\bfH_t^{\xi}\cap V\not\subset \B_{Mt})&\leq A \exp(-B t).\tag{\ref{ConjectureAML}}
	\end{align}
\end{corollary}
The rest of this section is devoted to showing \eqref{ConjectureSC}-\eqref{ConjectureFC}. For that we establish in Subsection~\ref{sec:BlockConstruction} a coupling of the infection process with an oriented percolation on a macroscopic grid. In Subsection~\ref{sec:RestartingProcedure} we will use this coupling together with a restarting procedure to finally show the remaining estimates.
\subsection{Block Construction}\label{sec:BlockConstruction}
In this subsection a coupling between a CPIU and an oriented percolation on $\IZ^d\times \IN$ is established. This is based on the block construction described in the work of Bezuidenhout and Grimmett \cite{bezuidenhout1990contact}, which couples a contact process on a slab $[a,a]^{d-1}\times \IZ$ with an oriented percolation on $\IZ\times \IN$. We are not the first to adapt and generalise this block construction; see, for example \cite{steif2008critical}, \cite{deshayes2014contact} and \cite{seiler2023contact}. Therefore, we summarise the construction and results, providing further details only if we deem them necessary. We closely follow \cite{seiler2023contact}, or rather the more detailed description in the corresponding dissertation \cite{Diss_Marco}, and \cite{steif2008critical}.
Both are an adaption based on \cite[Chapter I.2]{liggett1999stochastic}.
However, there are two notable differences between our coupling and other generalisations. First, we construct a coupling with (independent) oriented edge percolation rather than with site percolation. Second, our variation of the block construction leads to a coupling between a contact process on $\IZ^d$ and oriented percolation on $\IZ^d\times \IN$, rather than only on a slab. This specific variation allows us to apply some results from \cite{garet2014} and is needed to show several results in Section~\ref{sec:RestartingProcedure}. Therefore, we believe it is necessary to provide a brief summary and highlight the differences from the original construction of \cite{bezuidenhout1990contact} for the sake of completeness.

We first introduce a truncated version $(\prescript{}{L}\bfeta,\prescript{}{L}\bfxi)$ of the CPIU on a finite space-time box for arbitrary but fixed $L\in\IN$. For that let us set
\begin{align*}
	V_L:=[\unaryminus L,L]^d\cap \IZ^d \text{ and } E_L:= \{e: e\cap V_L\in E\}.
\end{align*}
Now we define the process $(\prescript{}{L}\bfeta,\prescript{}{L}\bfxi)$ via a restriction of the random mapping construction which is given in Section~\ref{sec:Construction} , where only changes inside $V_L\cup E_L$ are allowed.
This can be achieved by only considering maps in 
\begin{equation*}
    \cM_L:=\Big\{m\in\cM:\cD(m)\cup \bigcup_{a\in \cD(m)} \cR(m[a]) \subset V_L\cup E_L\Big\}.
\end{equation*}
In words, this means that we only consider infection paths that are restricted to $V_L$. As before, we also write $\prescript{}{L}\seteta_t^{I,\xi}$ to be the set of infected vertices at time $t$ that are reachable via a $\xi$-infection path originating from $I\subseteq V$ and staying in $V_L$.
We now start with some auxiliary results which we need for the block construction.\\
\begin{lemma}Assume the process is supercritical, i.e.\
$\IP(\tau^\origin=\infty)>0$, then
\begin{align*}
    \lim_{n\to\infty}\IP(\seteta^{[- n,n]^d,\szero}\neq \emptyset\, \forall\, t\geq 0)=1.
\end{align*}
\begin{proof}
In the classical case, this is typically shown using the self-duality of the contact process. In our situation, however, this is not always true.
Thus, we need to alternatively use the Poisson construction of the process and Birkhoff's Ergodic Theorem. This approach is sketched in the work of Steif and Warfheimer \cite{steif2008critical}. 
\end{proof}
\end{lemma}
\begin{lemma}\label{lemma:double_limit}
	For any $n,N\geq 1$ we have 
	\begin{equation*}
		\lim\limits_{t\to\infty}\lim\limits_{L\to\infty}\IP(|\prescript{}{L}\seteta_{t}^{[- n,n]^d,\szero}|\geq N)=\IP(\seteta_t^{[- n,n]^d,\szero}\neq \emptyset\, \forall\, t\geq 0)
	\end{equation*}
\end{lemma}
\begin{proof}
     Follows as in the classical case via a martingale argument. See \cite[Part I, Proposition~2.2]{liggett1999stochastic} for the classical case.
\end{proof}
For $L\in\IN$ and $T>0$ let $S(L,T)$ be the union of all lateral faces of the space-time box $[\unaryminus L,L]^d\times[0,T]$ and $S_+(L,T)$ the intersection of $S(L,T)$ with the first orthant and the hyperplane given by $x_1=L$ , i.e.\
\begin{align*}
	S(L,T)&:=\{(x,t)\in\IZ^d\times[0,T]:||x||_\infty=L\}\text{ and }\\
	S_+(L,T)&:=\{(x,t)\in S(L,T):x_1=L,x_i\geq 0\text{ for }2\leq i\leq d\}.
\end{align*}
Given $L,T>0$ and $I\subseteq(\unaryminus L,L)^d$, we define the random variables $N^I(L,T)$ and $N_+^I(L,T)$ representing the maximal number of points in $S(L,T)$ or $S_+(L,T)$, respectively, which are reachable via a $\szero$-infection path from $I$ contained within the space-time cube $[\unaryminus L,L]\times[0,T]$ and have either space distance greater than zero or time distance greater or equal to one.

\begin{lemma}\label{lemma:pos_korrelation}
	For every $n,N,M\geq 1$, $L>n$, $T\geq 0$ we have 
    \begin{align}\label{pos_cor_1}
	    \IP(|\prescript{}{L}\seteta_{T}^{[-n,n]^d,\szero}\cap[0,L)^d|\leq N)\leq \IP(|\prescript{}{L}\seteta_{T}^{[-n,n]^d,\szero}|\leq 2^d N)^{2^{-d}}
	\end{align}
	and 
	$$
	\IP(|N_{+}^{[-n,n]^d}(L,T)|\leq M)\leq \IP(|N_{}^{[-n,n]^d}(L,T)|\leq Md2^d)^{d^{-1}2^{-d}}.	
	$$
\end{lemma}
\begin{proof}
    We start with the first equation and consider the process $\prescript{}{L}\seteta_t$. Note that the event of having $N$ infected individuals at some fixed time $T$ in a fixed orthant is increasing in the occurrences of monotone increasing maps and decreasing in the occurrences of monotone decreasing maps. Since 
    all maps are monotone  by our assumptions, the assumption (2.19) immediately following \cite[Chapter 2, Theorem~2.14]{liggett1985interacting} -- 
    that every jump of the process is only between comparable states -- is satisfied. In particular, discretizing our model yields the result as in the basic case \cite[cf. p.11]{liggett1999stochastic}. Incorporating the background process is not completely trivial, and thus, for the sake of completeness, we sketch the procedure here. First, fix $L,T,n>0$ and  some initial configuration $(\eta,\xi)$.     For every map $m\in \cM_L$ we introduce $n$ Bernoulli random variables  $Y_m^i$, $1\leq i \leq n$ (with $p=e^{ \frac{-r_m T}{n}}$), which indicate if the Poisson clock of the map $m$ rings in the interval $\big[\tfrac{(i-1)T}{n},\tfrac{iT}{n}\big]$, i.e.\
    \begin{align*}
        Y_m^i=\1_{\big\{\big|\{m\}\times \big[\frac{(i-1)T}{n},\frac{iT}{n}\big]\cap \Delta\big|\geq 1\big\}}.
    \end{align*}
    By $m^i$ we denote the random map 
    \begin{align}
        m^i(\eta,\xi)=\begin{cases}
            m(\eta,\xi)\quad&\text{if}\quad Y^i_m=1,\\
            (\eta,\xi) &\text{else}.
        \end{cases}
    \end{align}
    Moreover, we enumerate all maps in $\cM_L=\{m_1,\dots,m_l\}$, which gives a fixed global order on $\cM_L$. This allows us to define the discrete process $(\bfeta^n_i,\bfxi^n_i)_{0\leq i\leq n}$ starting in  $(\bfeta^n_0,\bfxi^n_0)=(\eta,\xi)$ with transition from time $i-1$ to $i$ as follows
    \begin{equation}
        (\bfeta^n_i,\bfxi^n_i)=m^i_l\circ \cdots \circ m_1^i(\bfeta^n_{i-1},\bfxi^n_{i-1}).
    \end{equation}
    Clearly, if in every time step of size $\frac{T}{n}$ there is at most one map not identical to the identity, then the discrete processes $(\bfeta^n_i,\bfxi^n_i)_{0\leq i\leq n}$ and $(\prescript{}{L}\bfeta_{\frac{iT}{n}},\prescript{}{L}\bfxi_\frac{iT}{n})_{0\leq i\leq n}$ coincide.
    Moreover, since we only consider finitely many sites and maps and all maps are applied according to Poisson marks, we can unify all Poisson marks to one Poisson process $N$ on $[0,T]$ with rate $C_{\max}$. For this process we have $\IP(N([\frac{i T}{n},\frac{(i+1)T}{n}])>1~\text{for some   }1\leq i\leq n )\leq \frac{c}{n}$ for some constant $c$ and therefore 
    \begin{equation*}
        \IP((\bfeta^n_i,\bfxi^n_i)\neq(\prescript{}{L}\bfeta_{\frac{iT}{n}},\prescript{}{L}\bfxi_\frac{iT}{n})~\text{for some }0\leq i\leq n)\leq\frac{c}{n}. 
    \end{equation*}
    Note that \eqref{pos_cor_1} holds for the discrete process by applying   \cite[Corollary B.18]{liggett1999stochastic}. Letting $n$ go to $\infty$ yields \eqref{pos_cor_1} for $\prescript{}{L}\bfeta_t$. For details concerning the convergence we also refer to \cite[Section~2]{bezuidenhout1991exponetial}. 
    The second inequality follows in the same way.
\end{proof}

\begin{lemma}\label{lemma:limsup_cor}
	Suppose $L_j\nearrow\infty$ and $T_j\nearrow\infty$. Then for any $N,M\geq 1$ and any finite set $I$ we have
	\begin{equation*}
		\limsup\limits_{j\to \infty}\IP(N^I(L_j,T_j)\leq M)\IP(|\prescript{}{L}\seteta_{T_j}^{I,\szero}|\leq N)\leq \IP(\tau^{I,\szero}<\infty).
	\end{equation*}
\end{lemma}
\begin{proof}
    Follows as in the basic case, that is \cite[Part I, Proposition~2.8]{liggett1999stochastic}, by an martingale argument and using positive correlation again.
\end{proof}
We come to the finite space time conditions and therefore define the events
\begin{align*}
	\cE_1&=\cE_1(n,L,T):=\{\prescript{}{L+n}\seteta_{T+1}^{[-n,n]^d,\szero}\supset x+[\unaryminus n,n]^d\text{ for some }x\in[0,L)^d\},\\
	\cE_2&=\cE_2(n,L,T):=\{\prescript{}{L+2n}\seteta_{t+1}^{[-n,n]^d,\szero}\supset x+[\unaryminus n,n]^d\text{ for some }0\leq t < T\\
    &\hspace{17em}\text{ and }x\in\{L+n\}\times[0,L)^{d-1}\}.
\end{align*}

\begin{condition}[Finite space-time conditions
]\label{fi-st-condition}
	For all $\varepsilon>0$ there exists $n,L\geq 1$ and $T>0$ such that
	\begin{equation*}
	    \IP(\cE_1)>1-\varepsilon\quad\text{and}\quad \IP(\cE_2) > 1-\varepsilon.
	\end{equation*}
\end{condition}

\begin{theorem}\label{theorem:fst_fullfilled}
    If
    $\IP(\tau^{\origin, \szero}=\infty)>0$
    then the finite space-time conditions, i.e.\ Condition~\ref{fi-st-condition}, are satisfied. 
\end{theorem}
\begin{proof}
    As in the basic case, by using the results from above. See \cite[Theorem 7.5]{seiler2023contact}
\end{proof}
We now come to the coupling with an oriented percolation process. As an intermediate result we will establish a coupling with an $M$-dependent oriented percolation model on $\IZ^d\times\IN$. Before providing a rigorous definition, we note that $\Vec{E}$
depends on the dimension $d$ and we therefore write $\Vec{E}^d:=\Vec{E}$ when we want to emphasize the dimension.
\begin{definition}
    Let $d\geq 1$ and $M\geq0 $ be positive integers, $p\in(0,1)$ and $(\Omega,\cF,\IP)$  a probability space with a filtration $(\cG_n)_{n\geq 1}$. An adapted random field $W=(W_n^e)_{n\geq 1, e\in\Vec{E}^d}$ with values in $\{0,1\}^{\IN\times \vec{E}^d}$ is called $M$-dependent oriented (bond) percolation on $\IZ^d\times\IN$ with percolation parameter at least $p$ if 
    \begin{equation*}
        \IP(W_{n+1}^e=1|\cG_n\cup \sigma(W_{n+1}^{e'},d(e,e')>M))\geq p.
    \end{equation*}
    Hereby $d(e,e')$ is the $\ell_1$-distance between the centres of $e$ and $e'$. In case $M=0$ we also call it independent oriented (bond) percolation.
    In accordance to \cite[Definition 1]{garet2014} we denote the class of $M$-dependent oriented percolation on $\IZ^d\times\IN$ with percolation parameter at least $p$ by $\cC_d(M,p)$.
\end{definition}
For $x,y\in \IZ^d$ and $0\leq i<j$ we write $(x,i)\to_W (y,j)$ if there exists a path of adjacent edges $e_{i+1},\dots ,e_j$ with $W_k^{e_k}=1$ for all $k=i+1,\dots, j$ and $e_{i+1}=(x,v)$ as well as $e_j=(w,y)$ for some $v,w\in \IZ^d$. Furthermore, for $n\geq 1$ and $x\in \IZ^d$ we define
\begin{equation*}
    \bfP_n^{W,x}:=\{y\in \IZ^d: (x,0)\to_W (y,n) \}\quad\text{and}\quad
    \tau^{W,x}:=\inf\{n\geq 0: \bfP_n^{W,x}=\emptyset\}.
\end{equation*}
We use our usual shorthand notation $\bfP_n^{W}:=\bfP_n^{W,\origin}$ and $\tau^{W}:=\tau^{W,\origin}$, and we suppress $W$ from the superscript whenever it is clear from the context.
\begin{remark}
    Note, by definition, all visited space-time points lie on the grid $\cL^d$, i.e.\
    \begin{equation*}
        \bigcup_{n=0}^{\infty}\bigcup_{x\in\bfP_n^{W}}(x,n)\subset \cL^d:=\{(x,n)\in \IZ^d\times \IN_0: ||x||_1+n\mod 2=0 \}.
    \end{equation*} 
    \end{remark}
Before we state our coupling result, we recall that due to \cite{stacey_liggett} we can always couple a $M$-dependent percolation with an independent one which is dominated by the $M$-dependent one.
\begin{lemma}[{\cite[Lemma 2]{garet2014}}]\label{lemma:dominance}
    Let $d,M\geq1$. Then there exists a function $g_M: [0,1]\to [0,1]$ with $\lim_{p\to 1}g_M(p)=1$ such that we can couple every $\hat{W}\in \cC_p(M,d)$ with an independent percolation $W$ with percolation parameter $g_M(p)$ satisfying
    \begin{equation*}
        W_n^e \leq \hat{W}_n^e\quad\text{a.s. for all }n\geq 1, e\in \Vec{E}^d.
    \end{equation*}
\end{lemma}
The idea behind the coupling between the infection process and the oriented percolation is to switch from the microscopic grid $\IZ^d$ to a macroscopic one by subdividing the original grid into boxes of side length $2a$ where $a\in \IN$ will be specified later in the proof of Lemma~\ref{lemma:advanced_condition}. More precisely, we identify every macroscopic site $\hat{x}\in \hat{\IZ}^d$ with the box
\begin{equation*}
    B_{\hat{x},a}:=\hat{x}\cdot 2a+[\unaryminus a,a]^d
\end{equation*} 
and the boxes split up our entire microscopic space. Our oriented percolation process  $(W_k^e)_{k\geq 1, e\in \Vec{E}^d}$ will then live on the macroscopic space and the occurrence of $W_k^{(\hat{x},\hat{y})}$ will be determined by the graphical representation of our original process. 

\begin{theorem}\label{theorem:coupling_independent}
    If $\IP(\tau^{\origin}=\infty)>0$ then for every $p<1$ there are choices of $n,a,b\in \IN$ with $n<a$ such that we can couple our process with an independent oriented percolation $(W_k^e)_{k\geq 1, e\in \Vec{E}^d}$ with percolation parameter at least $p$ in the following way: 
    \begin{equation*}
        \hat{x} \in \bfP_k^W\Rightarrow \seteta_t^{[-n,n]^d,\szero}\supset x+[\unaryminus n,n]^{d}
    \end{equation*}
    for some $(x,t)\in B_{\hat{x},a}\times[5kb,(5k+1)b].$
\end{theorem}
To prove the theorem we follow the classical approach from Liggett~\cite[cf.\ p.\ 45 ff.]{liggett1999stochastic} and start by collecting some auxiliary results.

\begin{lemma}\label{lemma:one-condition}
	Assume Condition~\ref{fi-st-condition} is satisfied. Then for every $\varepsilon>0$ there exists $n,L,T$ such that $\IP(\cE_3(n,L,T))>1-\varepsilon$ with
	\begin{align*}
		\cE_3(n,L,T):=\big\{&\prescript{}{2L+3n}\seteta_t^{[-n,n]^d,\szero}\supset x+[\unaryminus n,n]^d\text{ for some }t\in[T, 2T),\\
		&\hspace{5em} x\in[L+n,2L+n]\times[0,2L)^{d-1}\big\}.
	\end{align*}
	
\end{lemma}
\begin{proof} 
The proof is identical to the one presented in Liggett's book for the result \cite[Part I, Proposition 2.20]{liggett1999stochastic} and no modifications are necessary.

\end{proof} 

To simplify the notation a bit we introduce for a macroscopic site $\hat{x}\in \IZ^d$  and $j,a,b\in\IN$ the space-time boxes 
$$
\smallbox(\hat{x},j):=\smallbox(\hat{x},j,a,b):=B_{\hat{x},a}\times [jb,(j+1)b] 
$$ and
$$
\bigbox(\hat{x},j):=\bigbox(\hat{x},j,a,b):=\big([\unaryminus 5a,5a]^{d}+2a\cdot \hat{x}\big)\times[jb,(6+j)b].
$$
Moreover, by $u_i\in \IZ^d$ we denote the unit vector in direction $i\in\{1,\dots,d\}$ and by $\Vec{U}=\{u_1,\dots,u_d,-u_1,\dots,-u_d\}$ we  denote the set of unit vectors in positive and negative directions. For any  site $x\in \IZ^d$ and $u\in \Vec{U}$ we denote by $\ e_x^u=(x,x+u)\in \Vec{E}$ the directed edge from $x$ to $x+u$. Clearly, by definition $\Vec{E}=\{e_x^u:x\in \IZ^d,u\in \Vec{U}\}.$
\begin{lemma}\label{lemma:advanced_condition}
	Suppose Condition~\ref{fi-st-condition} is satisfied. Then for every $\varepsilon>0$ there are choices of $n,a,b$ with $n<a$ such that if $(x,s)\in \smallbox(\hat{x},j,a,b)$ for some $\hat{x}\in\IZ^d$ and $j\in\IN$ then $\IP(\cE^{u_1}(n,a,b,x,s))\geq 1- \varepsilon$ with
	\begin{align*}
			\cE^{u_1}(n,a,b,x,s):=&\big\{\exists(y,t)\in \smallbox(\hat{x}+u_1,j+5,a,b)\text{ such that there are }\\ & \text{$\szero$-infection paths that stay in }\bigbox(\hat{x},j,a,b)\text{ and go from } \\&(x,s)+([\unaryminus n,n]^d\times\{0\})\text{ to every point in }(y,t)+([\unaryminus n,n]^d\times\{0\})\}.
	\end{align*}
\end{lemma}
\begin{proof} The Lemma follows by applying Lemma \ref{lemma:one-condition} multiple-times to move the space-time point $(x,s)$ to $(y,t)$ such that if $(x,s)$ is the centre of an infected cube, then $(y,t)$ will also be the centre of an infected cube. To apply the Lemma \ref{lemma:one-condition} properly one should note that instead of the positive box $[L+n,2L+n]\times[0,2L)^{d-1}$ in the definition of $\cE_3$ one can use any box obtained from the reflections of the positive box on the coordinate planes in $\IZ^d$ by symmetry. Also within the proof we set $a:=2L+n$ and $b:= 2T$. For more details we refer to  \cite[Part I, Proposition 2.22]{liggett1999stochastic}  where the proof is carried out.
\end{proof}
By symmetry we can also deduce immediately:
\begin{corollary}\label{corollary:advanced_condition}
	The statement of Lemma~\ref{lemma:advanced_condition} holds for any direction $u\in \Vec{U}$
    where we have to consider the slightly modified event
	\begin{align*}
\cE^{u}(n,a,b,x,s):=&\big\{\exists(y,t)\in \smallbox(\hat{x}+u,j+5,a,b)\text{ such that there are}\\ &\text{$\szero$-infection paths that stay in }\bigbox(\hat{x},j,a,b)\text{ and go from } \\&(x,s)+([\unaryminus n,n]^d\times\{0\})\text{ to every point in }(y,t)+([\unaryminus n,n]^d\times\{0\})\}.
\end{align*}
\end{corollary}

\begin{proof}[Proof of Theorem \ref{theorem:coupling_independent}] The proof is an adaptation of the proof of \cite[Part I, Theorem~2.23]{liggett1999stochastic} and \cite[Theorem~4.11]{deshayes2014contact}. So let $p<1$ be arbitrary but fixed and choose $\hat{p}$ large enough such that $G_{5d}(\hat{p})\geq p$ holds, where $G_{5d}$ is the function defined in Lemma \ref{lemma:dominance}. Given $\hat{p}$ choose $n,a,b\in\IN$ according to Lemma \ref{lemma:advanced_condition}. We will apply Corollary \ref{corollary:advanced_condition} several times to construct inductively with the help of the graphical construction a $5d$-dependent oriented percolation $(\hat{W}_k^e)_{k\geq 1, e\in \Vec{E}^d}$ and a collection of (random)  space-time points  $\big(Y_{\hat{x},k}\big)_{k\geq 1,\hat{x}\in\IZ^d}$ in the extended space-time $\mathbb{Z}^d\times \mathbb{R}_+\cup\{\dagger\}$ such that the following holds: if $\hat{x}\in \bfP_k^{\hat{W}}$, then $Y_{\hat{x},k}=(x,t)\in \smallbox(\hat{x},5k)$ 
and at time $t$ there exist $\szero$-infection paths to all vertices in $x+[\unaryminus n,n]^d$ starting from $[\unaryminus n,n]^d$ at time zero. Moreover,  $\hat{x}\notin \bfP^{\hat{W}}_k$ if and only if $Y_{\hat{x},k}=\dagger$. In particular, $Y_{\hat{x},k}=\dagger$ indicates that the macroscopic space-time point $(\hat{x},k)$ is not reached by the oriented percolation starting from the origin.

For notational convenience we define for every $u\in \Vec{U}$ an auxiliary function 
\begin{equation*}
    f^{u}:\Omega\times \IZ^d\times \IR_+\to \mathbb{Z}^d\times \mathbb{R}_+\cup\{\dagger\}
\end{equation*} which evaluates to $f^{u}((x,s)):=(y,t)$ if the event $\cE^{u}(x,s):=\cE^{u}(n,a,b,x,s)$ occurs and $(y,t)$ is a space-time point satisfying the requirements of this event (in case there are more space-time points satisfying the requirements just choose one according to an arbitrary but fixed order). In case $\cE^{u}(x,s)$ does not occur, set $f^{u}(x,s)=\dagger$. 
The coupling now works as follows:
	\begin{itemize}
    \item For $k=1$ and $u\in\Vec{U}$    let $Y_{u,1}:=f^{u}(\origin,0)$ and set $\hat{W}_1^{ e_\origin^{u}}=1$ if and only if $\cE^{u}(\origin,0)$ occurs. For all other edges we set $\hat{W}_1^{e}=1$ independently with probability $p$. Moreover, for all $\hat{x}\notin \Vec{U}$ set $Y_{\hat{x},1}:=\dagger$
    
	\item Assume we have constructed our collections until some $k-1$. If $Y_{\hat{x},k-1}=(x,s)\in \smallbox(\hat{x},5(k-1))$ for some $\hat{x}\in\IZ^d$ let $\hat{W}_k^{ e_{\hat{x}}^{u}}:=\1_{\cE^{u}(x,s)}$ for all $u\in\Vec{U}$. For any $\hat{x}$ with $Y_{\hat{x},k-1}=\dagger$ we set $\hat{W}_1^{ e_{\hat{x}}^{u}}=1$ independently with probability $p$ for all $u\in\Vec{U}$. To define $Y_{\hat{x},k}$ for any $\hat{x}\in\IZ^d$, check if there exists some $u\in \Vec{U}$ such that for $\hat{y}=\hat{x}-u$ we have that $Y_{\hat{y},k-1}\neq \dagger$ holds and $\cE^{u}(Y_{\hat{y},k-1})$ occurs. If so, define $Y_{\hat{x},k}:=f^{u}(Y_{\hat{y},k-1})$, otherwise set $Y_{\hat{x},k}:=\dagger$. Note, in case there are more $\hat{y}$ with the desired property, we just choose the first one according to some fixed order on $\Vec{U}$. Furthermore, in case we have set $Y_{\hat{x},k}\neq \dagger$ it follows by our induction hypothesis that $Y_{\hat{y},k-1}\in \smallbox(\hat{y},5(k-1))$ and by the definition of our auxiliary function we have $Y_{\hat{x},k}=f^{u}(Y_{\hat{y},k-1})\in \smallbox(\hat{x},5k) $.
    \end{itemize}
    Note, by definition, the events $\cE^{u}(x,s)$ and $\cE^{u'}(y,t)$ with $(x,s)\in \smallbox(\hat{x},5k)$  and $(y,t)\in \smallbox(\hat{y},5k)$  for some $\hat{x},\hat{y}\in\IZ^d$, $u,u'\in \Vec{U}$ and $k\in\IN$ are independent if $\bigbox(\hat{x},5k)$ and $\bigbox(\hat{y},5k)$ are disjoint, which is the case if$||\hat{x}-\hat{y}||_\infty >5$. Therefore, by construction, our process $(\hat{W}_k^e)_{k\geq 1, e\in \Vec{E}^d}$ is a $5d$-dependent oriented percolation process with parameter at least $\hat{p}$, which satisfies the requirements we made. Applying now Lemma \ref{lemma:dominance} gives us an independent oriented bond percolation $ W$ with parameter $p$ satisfying the desired coupling.
\end{proof}
\begin{remark}\label{rem:1DimensionalslabPercolation}
    Note that for any $p<1$ and any $\ell\leq d$ we can also construct an $\ell$-dimensional independent percolation $\tilde{W}$ with parameter at least $p$
    which satisfies
    \begin{equation*}
        \hat{x} \in \bfP_k^{\tilde{W}}\subset \IZ^\ell\times \{\origin\}^{d-\ell}\Rightarrow \seteta_t^{[-n,n]^d,\szero}\supset x+[\unaryminus n,n]^{d}
    \end{equation*}
    for some $(x,t)\in B_{\hat{x},a}\times[5kb,(5k+1)b]$. This directly follows by Theorem~\ref{theorem:coupling_independent} if we consider the $\ell$-dimensional subfield  $(\tilde{W}_k^e)_{k\geq 1, e\in \Vec{E}^\ell}$ with $1\leq \ell\leq d$ and $\Vec{E}^\ell:=\{e=(x,y)\in \Vec{E}^d:x,y\in \IZ^\ell\times\{\origin\}^{d-\ell}\}$. In particular our result generalises the common coupling in the literature (c.f. \cite[Part I, Theorem 2.23]{liggett1999stochastic}) where the underlying percolation lives in $\IZ\times \IN$ and survival of the percolation implies survival of the process in some basically one-dimensional space-time slab $\IZ\times[\unaryminus 5a,5a]^{d-1}$. 
\end{remark}

A classical first consequence of Theorem~\ref{theorem:coupling_independent} is that the parameter regime for global and local survival is the same. Moreover, we can use this to show that if a monotone CPIU has a positive survival probability, then so does the dual process and vice versa. Recall that this is not a priori clear, since the CPIU is not self-dual in general.
\begin{lemma}\label{lem:EquivalenzOfSurivalOfDual}
    For the CPIU it holds that
    \begin{equation}\label{eq:EquivalenceGlobalAndLocal}
          \IP(\tau^{\zero}=\infty)>0 \Leftrightarrow \liminf_{t\to \infty}\IP(\zero \in \bfeta^{\zero}_{t})>0.
    \end{equation}
    This implies in particular that if $\bfxi$ is reversible, then it holds that
    \begin{equation*}
          \IP(\tau^{\zero}=\infty)>0 \Leftrightarrow  \IP(\widecheck{\tau}^{\zero}=\infty)>0
    \end{equation*}
\end{lemma}
\begin{proof}
    The first statement can be shown analogously as in \cite[Proposition~7.15]{seiler2023contact}. Note that the referenced proof uses a coupling of the infection process with an oriented percolation on $\IZ \times[\unaryminus 5a,5a]^{d-1}$, which can also be derived from our coupling, see Remark~\ref{rem:1DimensionalslabPercolation}. Next we see that $\liminf_{t\to \infty}\IP(\zero \in \bfeta^{\zero}_{t})>0$ implies that $\nu\neq \delta_{\emptyset}\otimes \pi$, and thus the second claim follows as a consequence of Corollary~\ref{cor:FiniteVSInfinteSurv}.
\end{proof}

In preparation for further results we close this section with some properties concerning oriented percolation which are known or follow from known results.
For $x\in \IZ^d$ and $i\geq0$ we write $(x,i)\to_W \infty$ if for every $k> i$ there exists some $y$  with $(x,i)\to_W (y,k)$. Furthermore, for $\rho\in (0,1)$ and $n\geq 1$ let
\begin{align*}
    \hat{\gamma}(\rho,x)&:=\inf\big\{n\geq 0: \forall k\geq n : |\{i\in\{1,\dots, k\big\}: (0,\origin)\to_W (i,x)\to_W \infty\}|\geq \rho k \},\\
    R^{W}_{n}(x)&:=\inf\{k>R^{W}_{n-1}(x): (0,\origin)\longrightarrow_{W}(x,k)\}\quad \text{with}\quad R^{W}_{0}(x):=0.\\
    \hat{R}^{W}_{n}(x)&:=\inf\{k>\hat{R}^{W}_{n-1}(x): (0,\origin)\longrightarrow_{W}(x,k)\longrightarrow_{W}\infty\}\quad \text{with}\quad \hat{R}^{W}_{0}(x):=0.
\end{align*}

\begin{corollary}\label{corollary:macro_k_return}
    Let $d\geq 1$. There exists $p<1$ and positive constants $A, B,\alpha,\beta,\beta^*>0$ such that for every independent oriented percolation $W$ on $\IZ^d\times\IN$ with  at least $p$ it holds that
    \begin{equation}\label{sc_oriented_percolation}
    \IE[\1_{\{\tau^W<\infty\}}\exp(\alpha\tau^W)]\leq 1
\end{equation}
and
    \begin{align}
        \IP(\tau^W=\infty,R^{W}_{n}(x)\geq \beta ||x||+\beta^*n)&\leq A\exp{(-Bn)} \quad \forall n\in \IN,~ \forall x\in \IZ^d,\label{eq:R_k}\\
        \IP(\tau^W=\infty,\hat{R}^{W}_{1}(x)\geq \beta ||x||+n)&\leq A\exp{(-Bn)} \quad \forall n\in \IN,~ \forall x\in \IZ^d\setminus\{\origin\}.\label{eq:R_1'}
    \end{align}
\end{corollary}
\begin{proof} The fact that for $p<1$ large enough there exists some $\alpha>0$ such that \eqref{sc_oriented_percolation} holds is just \cite[Corollary 3.1]{garet2014}. Moreover, in  \cite[ Lemma 3.5]{garet2014} it is shown that there exists positive constants $A,B,\beta, \rho>0$ such that
    \begin{equation}\label{eq:survival_gamma}
        \IP(\tau^{W}=\infty ,\hat{\gamma}(\rho,x)\geq \beta ||x||+n)\leq A\exp({-Bn}) \quad \forall n\in \IN,~ \forall x\in \IZ^d.
    \end{equation}
   For the control of $R_n^W(x)$, which is just the $n$-th hitting time of $x$, we define the auxiliary variable
   \begin{equation*}
       \gamma'(\rho,x):=\inf\big\{n\geq 0: \forall k\geq n : |\{i\in\{1,\dots, k\big\}: (0,\origin)\to_W (i,x)\}|\geq \rho k \}.
   \end{equation*}
   Note, by definition $\gamma'(\rho,x)\leq \hat{\gamma}(\rho,x)$ and thus \eqref{eq:survival_gamma} also holds for $\gamma'$.
    For $\frac{n}{\rho}\geq \gamma'(\rho,x)$ we have by definition of $\gamma'$ that we hit $x$ from $\origin$ at least $n$ times until time $\frac{n}{\rho}$. Moreover, up to time $\frac{\gamma'}{\rho}$ we hit $x$ at least $\gamma'$ times. Thus
    \begin{equation*}
        R_n^W(x)\leq\frac{1}{\rho}\max(\gamma'(\rho,x),n).
    \end{equation*}
    Together with equation \eqref{eq:survival_gamma} for $\gamma'$ we get for any $\beta'>1$ that
    \begin{align*}
        \IP(\tau^W=\infty,R_n^W(x)\geq \frac{\beta}{\rho}||x|| + \frac{\beta'}{\rho} n )
        &\leq \IP(\tau^W=\infty, \gamma'(\rho,x)\geq \beta||x|| + \beta' n )
        \leq A e^{-B\beta' n }.
    \end{align*}
    Changing the constants appropriately gives \eqref{eq:R_k}. The last claim also follows directly from \eqref{eq:survival_gamma}. Just observe that by definition $\hat{R}_1^W(x)\leq \frac{\hat{\gamma}(\rho,x)}{\rho}$ for $x\neq \origin$ and therefore
    \begin{equation*}
        \IP(\tau^W=\infty,\hat{R}^{W}_{1}(x)\geq \frac{\beta}{\rho} ||x||+n)\leq \IP(\tau^W=\infty,\hat{\gamma}(\rho,x)\geq {\beta} ||x||+{\rho} n)\leq Ae^{-B\rho n}.\qedhere
    \end{equation*}
\end{proof}
We give another result concerning the density 
of independent oriented percolation on $\IZ\times \IN$.
\begin{lemma}\label{lemma:_density_up_inv}
    Given a threshold $\beta\in(0,1)$, we can choose $p<1$ large enough such that there exist constants $A,B>0$ such that for every independent oriented bond percolation $W$ on $\IZ\times \IN$  with parameter at least $p$ it holds that
    \begin{equation*}
        \IP(|\bfP_{n}^{W,2\IZ}\cap \B_{r}|<\beta r)\leq A\exp({-Br})\quad\text{for all }n\geq 0,r\geq0.
    \end{equation*}
\end{lemma}
\begin{proof} To keep the proof short, we only show the statement for even $n$. However, for odd $n$ the statement follows similarly. Note that for all $m\geq 0$ we can couple two percolations $W$ and $W'$ (with the same law) such that $\bfP_{2n+2m}^{W',2\IZ}\subset\bfP_{2n}^{W,2\IZ}$ holds for all $n\geq0$ (by starting $W$ at time $2m$). This implies 
    \begin{equation*}
        \IP(|\bfP_{2n}^{W,2\IZ}\cap \B_{r}|<\beta r)\leq \IP(|\bfP_{\infty}^{W,2\IZ}\cap \B_{r}|<\beta r)
    \end{equation*}
    where 
    \begin{equation*}
        \bfP_{\infty}^{W,2\IZ}=\{z\in 2\IZ:(z,0)\to_W (y,n)\text{ for all }n\geq 0\text{ and some } y\in \IZ\}.
    \end{equation*}
    With the contour argument given in \cite[(5.6) p.\ 592]{durrett1988large} (adapted to bond percolation) it follows that there exist constants $A,B>0$ such that
    \begin{equation*}
        \IP(|\bfP_{\infty}^{W,2\IZ}\cap \B_{k}|<\beta (k+1)) \leq A\exp(-Bk)
    \end{equation*}
    holds for all $k\geq 0$, if we choose $p$ sufficiently large. This directly implies our result for integer $k$. The extension of the inequality to  continuous radii $r$ is straightforward.
\end{proof}
We conclude this section with a statement concerning the density in a one-dimensional slab around the origin for an independent oriented percolation on $\IZ^d\times \IN$ for some general dimension $d\geq1$, starting from an arbitrary point sufficiently close to the origin.
\begin{lemma}\label{Lemma:d-dim-percolation_density_x}
    For $d\geq1$ and a given threshold $\beta\in(0,1)$ there exists some $p<1$ and constants $A,B,c,\alpha>0$ depending on $p$ such that for any independent oriented percolation $W$  on $\IZ^d\times \IN$ with percolation parameter at least $p$ we have for all $c'\leq c$ that
    \begin{align*}
        \IP(|\bfP_n^{x}\cap \B_{c'n}\cap (\IZ\times\{0\}^{d-1})|<\beta c'n|\tau^W=\infty)\leq A\exp({-Bn})\quad\forall n\geq0,~\forall x\in \B_{\alpha n} .
    \end{align*}
\end{lemma}
\begin{proof}
    We begin by introducing the coupled region for independent percolation, namely
    \begin{equation*}
        K_n:=\Big(\bfP_n^{W,\origin}\triangle\bfP_n^{W,2\IZ^d}\Big)^c.
    \end{equation*}
    It is known (see, e.g.\ \cite[(c), p. 117]{durrett1987} or \cite[Lemma 1]{garet2014}) that there exist constants $\hat{c},\hat{A},\hat{B}>0$ such that
    \begin{equation*}
        \IP(\B_{\hat{c}n}\not\subseteq K_n|\tau^W=\infty)\leq \hat{A}\exp(-\hat{B}n) \quad\forall n\geq0.   
    \end{equation*}
    Let $p<1$ be large enough such that Lemma~\ref{lemma:_density_up_inv} holds for $\beta$ and equation \eqref{eq:R_1'} holds for some $\tilde{\beta}$. For convenience, let $\B_{c'n}^1:=\B_{c'n}\cap (\IZ\times\{0\}^{d-1})$. Moreover, let $W^1\in \cC_1(M,p)$ be an embedded one-dimensional independent percolation restricted to $\IZ\times \{0\}^{d-1}$. We first show a slightly stronger statement for $x=\origin$, namely 
    \begin{equation}\label{eq:density_origin}
        \IP(|\bfP_m^{W,\origin}\cap \B_{c'n}^1|<\beta c'n|\tau^W=\infty)\leq A\exp(-Bn)
    \end{equation}
    for all $m\geq n$, $0<c'\leq \hat{c}<1$. This follows by the following observation
    \begin{align*}
        \IP(|\bfP_m^{W,\origin}\cap \B_{c'n}^1|<\beta c'n|\tau^W=\infty)&\leq \IP(|\bfP_m^{W,\origin}\cap \B_{c'n}^1|<\beta c'n, \B_{c'n}\subseteq K_m|\tau^W=\infty)\\
        &+\IP(\B_{c'n}\not\subseteq K_m|\tau^W=\infty)\\
        &\leq\frac{1}{\rho} \IP(|\bfP_m^{W,2\IZ^d}\cap \B_{c'n}^1|<\beta c'n)+\hat{A}\exp(-\hat{B}m)\\
        &\leq \frac{1}{\rho}\IP(|\bfP_m^{W^1,2\IZ}\cap \B_{c'n}^1|<\beta c'n)+\hat{A}\exp(-\hat{B}n).
    \end{align*}
    The first term can be controlled by Lemma~\ref{lemma:_density_up_inv} since we have chosen $p$ sufficiently large. 
    We now choose  $\alpha=\frac{1}{2\tilde{\beta}}$ and show that this choice is sufficient. For general $x\in \B_{\alpha n}$, let
    \begin{equation*}
        \sigma^x(y):=\hat{R}^W_1(y-x)\circ T_x
    \end{equation*} be the first \textit{essential hitting time} of $y$ if the percolation process starts in $x$. In particular, on the event $\{\tau^x=\infty\}$ the time $\sigma^x(y)$ is finite and  the percolation starting from $(y,\sigma^x(y))$ survives. Moreover, since $x\in \B_{\alpha n}$, we have $\tilde{\beta}||x||\leq \frac{n}{2}$ by definition of $\alpha$ and therefore 
    \begin{align*}
        \IP(\sigma^x(0)\geq n,\tau^{W,x}=\infty)&=\IP(\hat{R}^W_1(0-x)\circ T_x\geq n,\tau^{W,x}=\infty)\\&=\IP(\hat{R}^W_1(x)\geq n,\tau^{W,\origin}=\infty)\\
        &\leq\rho\cdot \IP(\hat{R}^W_1(x)\geq \tilde{\beta}||x||+ n/2|\tau^{W,\origin}=\infty).
    \end{align*}
    In particular, the last term can be controlled by \eqref{eq:R_1'}. Furthermore, by definition, $\bfP_k^{W,y}\circ \theta_{\sigma^x(y)}\subseteq \bfP^{W,x}_{k+\sigma^x(y)}$. Rewriting the last set-inclusion for $n>\sigma^x(y)$ gives $\bfP_{n-\sigma^x(y)}^{W,y}\circ \theta_{\sigma^x(y)}\subseteq \bfP^{W,x}_{n}$ on $\{\tau^x=\infty\}$, which yields     
    \begin{align*}
        \IP&(|\bfP_{2n}^{W,x}\cap \B_{c^*n}^1|<\beta c^*n|\tau^{W,x}=\infty)\leq \IP(\sigma^x(0)\geq n|\tau^{W,x}=\infty)
        \\&\quad\quad+\IP(|\bfP_{2n-\sigma^x(0)}^{W,\origin}\circ\theta_{\sigma^x(0)} \cap \B_{c^*n}^1|<\beta c^*n, \sigma^x(0)< n|\tau^{W,x}=\infty).
    \end{align*}
By \eqref{eq:density_origin} the second term can be controlled for all $c^*\leq \hat{c}$. Setting $c'=\frac{c^*}{2}$ gives the result for all even $n$ and $c'\leq \frac{\hat{c}}{2}\eqqcolon c$. For odd $n$ the  result follows in a similar way.
\end{proof}

\subsection{Restarting Procedure}\label{sec:RestartingProcedure}

In this chapter we show the growth estimates \eqref{ConjectureAML}--\eqref{ConjectureFC}, which we need to apply Theorem~\ref{Conjecture1}. For that we will make use of the coupling between the infection process and an independent oriented macroscopic  percolation living on $\IZ^d\times \IN$, see Theorem~\ref{theorem:coupling_independent}.

More precisely, we adapt the restart procedure sketched by Durrett~\cite{durrett1991} and given in \cite{deshayes2014contact} and do the following: if the infection process $\eta^{\zero,\xi}$ fully infects a seed $x+[\unaryminus n,n]^d$, then we couple the process with an independent macroscopic percolation $W$. By definition of the coupling we know that if the percolation $W$ survives so does the infection process. If it does not, but the infection process is still alive, then we restart the whole procedure at a suitable time point when again a seed $x'+[\unaryminus n,n]^d$ is fully infected such that we can again couple with a macroscopic percolation. 

At the end of the procedure we will derive a random space time point $(\sigma,Y)$ such that a macroscopic percolation $W$ starting at time $\sigma$ from point $Y$ will survive. Moreover, if our original process survives, then the whole cube $Y+[\unaryminus n,n]^d$ is infected at time $\sigma$.

In fact, this procedure does not only work for a monotone CPIU, but for a more general CPDRE. Therefore, we assume for the rest of section the following.
\begin{assumption}\label{ass:coupling_exists}
    Let $(\bfeta,\bfxi)$ be a CPDRE. Suppose that there exists a supercritical, monotone CPIU $(\underline{\bfeta},\underline{\bfxi})$ such that $(\underline{\bfeta}_t,\underline{\bfxi}_t)\leq (\bfeta_t,\bfxi_t)$ for all $t\geq 0$.
\end{assumption}
This allows us to extend the restarting procedure to the CPDRE, although it does not apply to the entire parameter regime of survival.

\begin{remark}\label{remark:dual_cpiu}
    We denote the dual process of the monotone CPIU  $(\underline{\bfeta},\underline{\bfxi})$ by $(\widecheck{\underline{\bfeta}}_t,\widecheck{\underline{\bfxi}}_t)$. By Lemma~\ref{lem:EquivalenzOfSurivalOfDual} this process is also supercritical. Moreover, we also have that our dual processes are coupled, i.e.\ $(\widecheck{\underline{\bfeta}}_t,\widecheck{\underline{\bfxi}}_t)\leq (\widecheck{\bfeta}_t,\widecheck{\bfxi}_t)$ for all $t\leq \frac{3t}{2}$. 
\end{remark}
Given $d\geq 1$ let $p<1$ be sufficiently large such that Corollary \ref{corollary:macro_k_return} holds. 
Moreover, let $a,b,n$ be the constants depending on $p$ given in Theorem \ref{theorem:coupling_independent} to establish a coupling between the supercritical monotone CPIU $(\underline{\bfeta},\underline{\bfxi})$ and an independent oriented percolation. Similarly to \cite{deshayes2014contact}, we are looking for a space time point $(x,t)$ such that a box $x+[\unaryminus n,n]^d$ is fully infected by $\underline{\bfeta}$ at time $t$ and we can start a coupling with a macroscopic percolation process. Note, that there exists a $\delta>0$ such that
\begin{equation*}
    \IP\big(\exists x\in \IZ^d: \underline{\bfC}^{\zero,\szero}_1\supset x+[\unaryminus n,n]^d\big)>\delta.
\end{equation*}
We will now successively define random times $(N_\ell,M_\ell)_{\ell\geq 1}$, where $N_{\ell}$ indicates the time of the $\ell$-th restart and $M_{\ell}$ the time until the macroscopic percolation dies out. Assume that we have already defined $N_{\ell-1}$ and $M_{\ell-1}$ with $M_{\ell-1}<\infty$ for some $\ell\geq 2$, then we set $U_{\ell-1}:=N_{\ell-1}+1+6bM_{\ell-1}+U_{\ell-2}$. Note that we set $U_0=0$ if we just started the procedure. Now we describe the restart procedure. First, we restart the process $\underline{\bfeta}$ at time $U_{\ell-1}$ in the worst background-state, i.e.\ $\szero$, with only one infected vertex which we choose in the following way. In case, our original process $\bfeta$ is alive at time $U_{\ell-1}$, we just choose one infected vertex for the restart (according to some arbitrary but fixed order). Otherwise we restart the process with only the origin infected. Formally, we define the process $\underline{\bfeta}_t^{(\ell)}$ via the same graphical representation as our original process starting at time $U_{\ell-1}$. If $\bfC^{\zero,\xi}_{U_{\ell-1}}\neq \emptyset$ we set $\underline{\bfC}_t^{(\ell)}:= \underline{\bfC}_t^{x,\szero}\circ \theta_{U_{\ell-1}}$  for some $x\in \bfC^{\zero,\szero}_{U_{\ell-1}}$ otherwise define $\underline{\bfC}_t^{(\ell)}:= \underline{\bfC}_t^{\origin,\szero}\circ \theta_{U_{\ell-1}}$. Next, we define the stopping time
\begin{equation*}
    N_{\ell}:=\inf\big\{k\geq 0: \underline{\bfC}_{k+1}^{(\ell)}= \emptyset \text{ or } \exists x \text{ s.t. } \underline{\bfC}_{k+1}^{(\ell)}\supset x+[\unaryminus n,n]^d \big\}.
\end{equation*}
Clearly, $N_\ell$ has sub-geometric tail by the following observation
\begin{align*}
    \IP(N_\ell=m)&= \IP\big(\{\underline{\bfC}_{m+1}^{(\ell)}= \emptyset \text{ or } \exists x \text{ s.t. } \underline{\bfC}_{m+1}^{(\ell)}\supset x+[\unaryminus n,n]^d\}\cap \{N_\ell \geq m \}\big)\\
    &=\IP\big(\underline{\bfC}_{m+1}^{(\ell)}= \emptyset \text{ or } \exists x \text{ s.t. } \underline{\bfC}_{m+1}^{(\ell)}\supset x+[\unaryminus n,n]^d \big| N_\ell \geq m \big)\cdot \IP(N_\ell \geq m )\\
    &\geq \IP\big(\exists x \text{ s.t. } \underline{\bfC}_{m+1}^{(\ell)}\supset x+[\unaryminus n,n]^d \big| N_\ell \geq m \big)\cdot \IP(N_\ell \geq m )\\
    &\geq \IP\big(\exists x\in \IZ^d: \underline{\bfC}^{\zero,\szero}_1\supset x+[\unaryminus n,n]^d\big)\cdot \IP(N_\ell \geq m )>\delta \cdot \IP(N_\ell \geq m ).
\end{align*}
\begin{remark}\label{remark:first_issue}
    Note for the second inequality we used the Markov-Property and the fact that $\{N_\ell \geq m \}$ implies $\{\underline{\bfC}_m^{(\ell)}\neq \emptyset\}$. However, we can not deduce the existence of some infected vertex $x\in \underline{\bfC}_m^{(\ell)}$ which lies in the slab $\IZ\times [\unaryminus a,a]^{d-1}$. This is one reason why we consider the macroscopic percolation on the whole graph $\IZ^d$ and not only on a slab as in \cite{deshayes2014contact}.
\end{remark}

Now we distinguish between the two possible outcomes. In case $\underline{\bfC}_{N_{\ell}+1}^{(\ell)}= \emptyset$ we set $M_{\ell}=0$, since we did not reach a configuration suitable for a coupling with the macroscopic percolation. 
On the other hand, if we find some $x$ with $x+[\unaryminus n,n]^d\subset \underline{\bfC}_{N_{\ell}+1}^{(\ell)}\subset\bfC_{N_{\ell}+1+U_{\ell -1}} $ we couple the infection process with a macroscopic percolation $\big(W^{(\ell)}_k\big)_{k\geq 0}$ with $W^{(\ell)}_k:=W_k\circ T_x\circ\theta_{N_{\ell}+1+U_{\ell-1}}$ starting at time $N_{\ell}+1+U_{\ell-1}$ from the new microscopic origin $x$.  
Let $M_{\ell}$ be the extinction time of $W^{(\ell)}$,  
 which means $M_{\ell}\in\IN\cup\{\infty\}$. Then,
\begin{enumerate}
    \item If $\big(W^{(\ell)}_k\big)_{k\geq 0}$ percolates, i.e.\ $M_{\ell}=\infty$, we set $Y=x$ and the procedure stops.
    \item If $\big(W^{(\ell)}_k\big)_{k\geq 0}$ does not percolate, i.e.\ $M_{\ell}<\infty$,  we start the next iteration of our procedure.
\end{enumerate}

Note that, if $\bfC^{\zero,\szero}_{U_{\ell-1}}=\emptyset$, i.e.\ the original infection process dies out, we  may stop the procedure. However, in order to guarantee that the $(N_\ell)_{\ell \geq 1}$ are identically distributed and independent, we need to artificially resume the construction until we find some $\ell$ with $M_\ell=\infty$ as above. Furthermore, our construction will ensure that we always find a well defined starting point $(\sigma,Y)$ for our macroscopic percolation.

Let  
$L:=\inf\{\ell>0: M_\ell=\infty\}$ and note that $L$ is the number of trials until we successfully infected a spatially shifted box of the form $[\unaryminus n,n]^d$ and the the macroscopic model percolates. Obviously $L$ is a geometric distributed random variable by construction. Then, we define
\begin{equation*}
    \sigma:=U_{L-1}+N_L+1=\sum_{i=1}^{L-1}(N_i+1+6b M_i)+N_{L}+1
\end{equation*}
which is the (microscopic) time until we start a successful macroscopic percolation model. Note that on the event $\{\tau^{\zero,\xi}<\infty\}$ it holds that $\sigma>\tau^{\zero,\xi}$. Moreover, the variables $N_\ell$ and $\sigma$ depend on $\xi$ but not there distributions. For this reason we avoided superscripts.\\

Since $L$ is a geometric distribution there exists a constant $c$ such that $\IE[\exp(cL)]<\infty$. Furthermore, the probability of the event $\{t<N_1\}$  decays exponentially fast which implies that some exponential moments exist. The same holds for $M_1$ on the event $\{M_1<\infty\} $ since Corollary~\ref{corollary:macro_k_return} implies that inequality \eqref{sc_oriented_percolation} holds for some $\alpha>0$ and $M_1=\tau^{W^{(1)}}$.
Thus, we can choose a constant $c_0>0$ small enough such that 
\begin{equation*}
    \IE[\exp(c_0 (N_1+1+6bM_1\1_{\{M_1<\infty\}}))]\leq e^{c}.
\end{equation*}
With this choice we see that $\sigma$ has exponential moments, since
\begin{equation} \label{eq:exp_bounds_sigma}
\begin{split}
    \IE[\exp(c_0 \sigma)]&=\IE\bigg[\IE[\exp(c_0 (N_L+1))|L]\prod_{i=1}^{L-1} \IE[\exp(c_0 (N_i+1+6bM_i))|L]\bigg]\\
    &\leq \IE\big[\IE[\exp(c_0 (N_1+1+6bM_1\1_{\{M_1<\infty\}}))]^{L}\big]\leq  \IE[\exp(cL)]<\infty.
    \end{split}
\end{equation}
We conclude with the following observations:
\begin{lemma}\label{lemma:Y_sigma}
    For the random variables $\sigma$ and $Y$ it holds:
    \begin{enumerate}
        \item On the event $\{\tau^{\origin,\xi}<\infty\}$ we have $\sigma>\tau^{\origin,\xi}$.
        \item On the event $\{\tau^{\origin,\xi}=\infty\}$ we have $Y+[\unaryminus n,n]\subset \bfC_\sigma^{\origin,\xi}$ and the macroscopic percolation $\bW:=\big(W^{(L)}_k\big)_{k\geq 0}$ starting from $Y$ at time $\sigma$ survives.
        \item There exist constants $A$ and $B$ such that for every $t>0$ it holds that
        \begin{equation*}
            \IP(\sigma >t)\leq A\exp(-Bt)\quad \text{and}\quad
            \IP(||Y||>t)\leq A\exp(-Bt).
        \end{equation*}
        
    \end{enumerate}
\end{lemma}
\begin{proof}
    The first two points follow directly by our construction. The fact that the tail of $\sigma$ decays exponentially follows from the fact that $\sigma$ has exponential moments, as shown in \eqref{eq:exp_bounds_sigma}. For the second inequality we exploit the at most linear growth, i.e.\ Corollary~\ref{corollary:AML}, to deduce for $t>0$
    \begin{equation*}
        \begin{aligned}
       \IP(||Y||>t)&=\IP\big(||Y||>t,\sigma> c t\big)+\IP\big(||Y||>t, \sigma\leq ct\big)\\
       &\leq \IP\big(\sigma> c t\big)+\IP\big(||Y||>t, \sigma\leq ct\big)\\
       &\leq \IP\big(\sigma> c t\big)+\IP\big(\exists y : ||y||>t, y\in \bfH_{ct}^\xi\big)\\
       &\leq A\exp\big(-Bct\big)+\IP(\bfH_{ct}^\xi\not\subseteq B_t))\leq A^*\exp(-B^* t),
        \end{aligned}
    \end{equation*}
    where $c=\frac{1}{M}$, $M$ is the constant from Corollary~\ref{corollary:AML} and $A^*$ and $B^*$ are some new constants.
\end{proof}
\begin{corollary}\label{corollary:SC}
     There exist constants $A,B,M,c>0$ such that for all $\xi$ and all $x\in V$
	\begin{align}
	\IP(t< \tau^{\origin,\xi}<\infty)&\leq A \exp(-B t),\tag{\ref{ConjectureSC} }
	\end{align}
\end{corollary}
\begin{proof}
    Note \eqref{ConjectureSC} is a direct consequences of the first and third point of  Lemma~\ref{lemma:Y_sigma}.
\end{proof}
 One thing which remains to show is the at least linear growth, that is \eqref{ConjectureALL}.  
\begin{proposition}\label{prop:ALL}
There exist constants $A,B,\beta>0$ such that
\begin{equation*}
    \sup_{\xi}\IP\big(t^{\origin,\xi}(x)\geq \beta ||x||+t,\tau^{\origin,\xi}=\infty\big)\leq A \exp(-B t)
\end{equation*}    
\end{proposition}
\begin{proof}
    Let $x$ and $\xi$ be arbitrary but fixed and $p<1$ sufficiently large such that Corollary~\ref{corollary:macro_k_return} holds.  On the event that the infection survives, i.e.\ $\{\tau^{\zero,\xi}=\infty\}$, there exists a macroscopic percolation $W\circ\theta_{\sigma}\circ T_{Y}$ with parameter at least $p$, which percolates and is dominated by the infection process. Let $\hat{x}\in \IZ^d$ be the macroscopic site containing $x$ when the centre of the macroscopic percolation is at $Y$, i.e.\ $\hat{x}$ is the unique site such that $x-Y\in (\unaryminus a,a]^d+2a\hat{x}$, and let $R_k^W:=R_k^W(\hat{x})\circ \theta_\sigma \circ T_{Y}$ be the $k$-th hitting time of $\hat{x}$ by the percolation $W$.

    By the coupling of the percolation with our original process we know that on the event $\{\tau^{\zero,\xi}=\infty\}$ there exists a box $[\unaryminus n,n]^d+z_k\subset [\unaryminus a,a]^d+2a\hat{x}+Y$ and a time $s_k\in [0,b]+R_k^W\cdot 5b +\sigma $ such that
    $[\unaryminus n,n]^d+z_k\subset \bfC^{\zero,\xi}_{s_k}$ if $R_k^W<\infty$. For convenience we define $s_k:=\infty$ if $R_k^W=\infty$ which yields $s_k<\infty$ if and only if $R_k^W<\infty$.
    
    Note, the exponential control for the macroscopic first hitting time $R_1^W$ together with the exponential controls in Lemma~\ref{lemma:Y_sigma} imply an exponential control for the first hitting time of the macroscopic site $\hat{x}$ by our original process. However, since hitting the macroscopic site $\hat{x}$ does not imply that we hit the microscopic site $x$ with our original process, we have to refine the argument. 
    
    From the exponential controls of $R_k^W$ with $k\geq 1$ we deduce exponential controls for  the  $s_k$'s with $k\geq 1$.  By definition $s_k\leq 5b R_k^W +\sigma+b$ on the event $\{\tau^{\zero,\xi}=\infty\}$ which yields for $\beta,\beta'$ that
    \begin{align*}
        \IP(\tau^{\zero,\xi}=\infty&,s_k\geq \beta||x|| + \beta'k )\\
        &\leq\IP(5b R_k^W +\sigma+b\geq \beta||x|| + \beta'k )\\
        &\leq \IP\bigg(\sigma\geq \frac{\beta||x|| +\beta'k-b}{2}\bigg)+\IP\bigg(R_k^W(\hat{x})\circ\theta_\sigma\circ T_Y \geq \frac{\beta||x|| +\beta'k-b}{10b}\bigg).
    \end{align*}
    For the first term we have exponential control in $k$ by Lemma \ref{lemma:Y_sigma}, for the second one note that
    \begin{align*}
        \IP\bigg(R_k^W(\hat{x})\circ\theta_\sigma\circ T_Y &\geq \frac{\beta||x|| +\beta'k-b}{10b}\bigg)\\
        &=\IP\bigg(\tau^W=\infty,R_k^W(\hat{x}) \geq \frac{\beta||x|| +\beta'k-b}{10b}\bigg)\\
        &\leq \IP\bigg(||Y||>\frac{\beta'k}{2\beta}\bigg)+\IP\bigg(\tau^W=\infty,R_k^W(\hat{x}) \geq \frac{\beta||\hat{x}|| +\frac{\beta'k}{2}-b}
        {10b}\bigg).
    \end{align*}
    Hereby, the first equation follows by the translation invariance of our underlying graphical construction and the definitions of $Y$ and $\sigma$. For the inequality we use that by definition $||x-Y||\geq ||\hat{x}||$ and therefore
    \begin{equation*}
        ||x||\geq ||x-Y|| -||Y||\geq ||\hat{x}|| -\frac{\beta'k}{2\beta}\quad \text{if}~~||Y||\leq \frac{\beta'k}{2\beta}.
    \end{equation*}
    Again the first summand can be controlled by Lemma \ref{lemma:Y_sigma} and the second one can be controlled by Corollary \ref{corollary:macro_k_return} if we choose $\beta$ and $\beta'$ appropriately. In particular, there exist constants $\beta,\beta',A$ and $B$ such that
    \begin{equation}\label{eq:control_sk}
        \IP(\tau^{\zero,\xi}=\infty,s_k\geq \beta||x|| + \beta'k )\leq A \exp({-Bk}).
    \end{equation}
    
    Now, given $R_i^W<\infty$ or equivalently $s_i<\infty$, let $B_i$ be the event that from the infection seed $[\unaryminus n,n]^d+z_i$ there is a $\szero$-infection path to $x$ within one unit of time, i.e.\
    \begin{equation*}
        B_i:=\{\exists y\in[\unaryminus n,n]^d+ z_i \text{ s.t. } (y,s_i)\stackrel{\szero}{\longrightarrow}(x,s_i+1)\}.
    \end{equation*}
    Since $z_i$ and $x$ lie in the same macroscopic box of length $2a$ the distance between $z_i$ and $x$ is bounded by $2ad$ and therefore $\IP(B_i)$ is uniformly bounded away from zero. In particular, there exists a $c>0$ such that $\IP(B_i)\geq c>0$ for all $i\geq 1$ with $R_i^W<\infty$. Let
    \begin{equation*}
        A_k:=\bigcap_{i=1}^{k}\{s_i<\infty\}\cap B_i^c
    \end{equation*}
    be the event, that after the first $k$ macroscopic infections of the site $\hat{x}$ there is no infection path from the corresponding infection seeds $([\unaryminus n,n]^d+ z_i)_{1\leq i\leq k}$ to $x$ within time one. Clearly, 
        \begin{equation}\label{eq:ak}
        \{t^{\origin,\xi}(x)>s_k\}\subseteq\{t^{\origin,\xi}(x)>s_{k-1}+1\}\subseteq A_{k-1}.
    \end{equation}
    Moreover, by our uniform bound on $\IP(B_i)$
    \begin{equation*}
        \IP(A_k)=\IE[\1_{\{A_{k-1}\cap\{ s_k<\infty\}\}}\IE[\1_{B_k^c}|\cF_{s_k}]]\leq (1-c)\IE[\1_{A_{k-1}}]\leq (1-c)^k.
    \end{equation*}
   Together with \eqref{eq:control_sk} and \eqref{eq:ak} we get
    \begin{align*}
    \IP\big(t^{\origin,\xi}(x)\geq \beta ||x||+\beta'k,\tau^{\origin,\xi}=\infty\big)&\leq \IP(t^{\origin,\xi}(x)>s_k) +\IP(s_k\geq \beta ||x||+\beta'k,\tau^{\origin,\xi}=\infty)\\
    &\leq (1-c)^{k-1}+A\exp({-Bk}).
    \end{align*}
    This finishes the proof since the bound does not depend on $\xi$.
\end{proof}
\begin{remark}
The proof above is motivated by \cite[Theorem 2.5]{deshayes2014contact}. However, we do not couple our process with one-dimensional slab percolation for two reasons. First, there is a problem with the application of the Markov property in the restart procedure above, as described in Remark~\ref{remark:first_issue}. Second, to hit a site $x=(x_1,\dots,x_d)$, \cite[Theorem 2.5]{deshayes2014contact} repeats the coupling with a one-dimensional percolation along the path
\begin{equation*}
(x_1,0,\dots,0)\to(x_1,x_2,0,\dots,0)\to\cdots\to(x_1,\dots,x_d).
\end{equation*}
This requires the microscopic sites along the path to be essentially hit, rather than merely hit, to guarantee the survival of the restarted process. However, these essential hitting times are no longer stopping times, so it was not clear to us how to apply the strong Markov property in this case.
\end{remark}

It remains to show that the process couples exponentially fast, i.e.\ \eqref{ConjectureFC} holds. For that we first need to show the following auxiliary result.
\begin{lemma}\label{lem:forwad_backward_intersection}
    Let $\IP(\tau^\origin=\infty)>0$, then there exist constants $A,B,\alpha>0$ such that for  any $\xi$ and $x\in \B_{\alpha t}$ we have that
    \begin{equation*}
        \IP(\bfC_{t}^{\origin,\xi}\cap \widehat{\bfC}_{t}^{x,\xi,2t}=\emptyset,\widehat{\bfC}_{t}^{x,\xi,2t}\neq \emptyset, \bfC_t^{\zero,\xi}\neq \emptyset)\leq A\exp(-Bt)\quad\text{for all }t\geq 0.
    \end{equation*}
\end{lemma}
\begin{proof} 
Let 
\begin{equation*}
    E_0:=\Big\{\widehat{\bfeta}_{u}^{x,\xi,2t}=\widecheck{\bfeta}_{u}^{x,t/2,2t}\,\, \forall u\leq t\Big\}
\end{equation*} be the event that the two dual processes $\widehat{\bfeta}_{u}^{x,\xi,2t}$ and $\widecheck{\bfeta}_{u}^{x,t/2,2t}$  coincide from time $2t$ (backwards in time) until time $t$. By Lemma~\ref{lem:ControlDualProcess} the probability of $E^c_0$ decays exponentially in $t$ and the constants do not depend on $\xi$, thus it suffices to control
\begin{equation*}
    E:=\{\bfC_{t}^{\origin,\xi}\cap \widecheck{\bfC}_{t}^{x,t/2,2t}=\emptyset,\widecheck{\bfC}_{t}^{x,t/2,2t}\neq \emptyset, \bfC_t^{\zero,\xi}\neq \emptyset\}.
\end{equation*}
For the given threshold $\beta=\frac{3}{4}$, let $p^*<1$ and $\alpha^*$,$c^*>0$ be the constants given by Lemma~\ref{Lemma:d-dim-percolation_density_x}. For this $p^*$ there exists according to Theorem~\ref{theorem:coupling_independent} some further constants $n,a,b$ such that we can couple our underlying monotone CPIU $\underline{\bfeta}$ with a $d$-dimensional independent percolation. According to Remark~\ref{remark:dual_cpiu} the backwards process $\widecheck{\underline{\bfC}}^{x,t/2,2t}$ is also a supercritical CPIU  and we can also couple this process according to Theorem~\ref{theorem:coupling_independent} where we get some constants $\hat{n},\hat{a},\hat{b}$. Let us now fix $t\geq132\hat{b}+22$ and some $x\in \B_{\alpha t}$ with $\alpha=\frac{\hat{\alpha}}{22\hat{b}}$ (the awkward choice will become clear later).

We start our restart procedure from time $0$ and the origin to find a space-time-starting point $(\sigma,Y)$ for an independent percolation $W$ with percolation parameter at least $p$ which gets dominated by the forward time process. We can do the same for our backward process $\widecheck{\bfC}_{u}^{x,t/2,2t}$ starting our restart procedure at site  $x$ going backwards in time from $2t$, which yields a space-time-starting point $(\hat{\sigma}^x,\hat{Y}^x)$ from where an independent percolation $\hat{W}$ with percolation parameter at least $p$ starts. 
Let us define the events 
\begin{equation*}
    E_1:=\Big\{\sigma\leq \frac{t}{2}-1-b ,||Y||\leq \Big\lfloor\frac{t}{10b}\Big\rfloor\cdot \hat{\alpha} \Big\}~~\text{and}~~
    \hat{E}_1:=\Big\{\hat{\sigma}^x\leq \frac{t}{2}-1-\hat{b},|| \hat{Y}^x-x||\leq \Big\lfloor\frac{t}{10\hat{b}}\Big\rfloor\cdot \frac{\hat{\alpha}}{2} \Big\}
\end{equation*}that we find the space time points $(Y,\sigma)$ and $(Y^x,\sigma^x)$, respectively, reasonable fast and not to far from the starting points of the restart procedures. Clearly, by Lemma~\ref{lemma:Y_sigma} the complements of both events can be controlled uniformly for all $\xi$ and hence it suffices to control $\IP(E\cap E_1\cap\hat{E}_1 )$. 

Let $k:=\Big\lfloor\frac{t-1-\sigma-b}{5b}\Big\rfloor$ and $\hat{k}:=\Big\lfloor\frac{t-1-\sigma^x-\hat{b}}{5\hat{b}}\Big\rfloor$ be the remaining (random) macroscopic times for the forward, respectively, the backward percolation before they reach the microscopic times $t-1-b$ or $t+1+\hat{b}$, respectively. By definition we have 
\begin{align*}
    \Big\lfloor\frac{t}{10b}\Big\rfloor \leq k\leq \Big\lfloor\frac{t-1-b}{5b}\Big\rfloor\text{  and   }||Y||\leq\alpha k \text{   on   }E_1.
\end{align*}
Similarly, on $\hat{E}_1$ we have $ \Big\lfloor\frac{t}{10\hat{b}}\Big\rfloor \leq \hat{k}\leq \Big\lfloor\frac{t-1-\hat{b}}{5\hat{b}}\Big\rfloor$ and 
\begin{equation*}
    ||Y^x||\leq ||Y^x-x||+||x||\leq \frac{\hat{\alpha}}{2}\hat{k}+\alpha t\leq \hat{\alpha} \hat{k}
\end{equation*}  by our choice of $\alpha$ and $t$. Moreover, our restart procedure guarantees that the percolations starting from $(\sigma,Y)$ and $(\sigma^x,Y^x)$, respectively, survive. 

In order to make sure that both percolation processes cover with high probability sufficiently many sites in the same space, let us compare $ka$ and $\hat{k}\hat{a}$. Without loss of generality $ka\leq \hat{k}\hat{a}$ and we choose $c=c^*$ and $\hat{c}\leq c^*$ such that $l(t)=kac=\hat{k}\hat{a}\hat{c}$. Consider now the events
\begin{equation*}
    E_2:=\Big\{\Big|\bfP_{k}^{W,Y}\circ\theta_\sigma\cap \B_{ck}^1\Big|\geq\beta ck\Big\}~~\text{and}~~
    \hat{E}_2:=\Big\{\Big|\bfP_{\hat{k}}^{\hat{W},\hat{Y^x}}\circ\theta_{-\sigma^x}\cap \B_{\hat{c}\hat{k}}^1\Big|\geq\beta \hat{c}\hat{k}\Big\}.
\end{equation*}
By our construction we can apply Lemma~\ref{Lemma:d-dim-percolation_density_x} to show that $\IP(E_1\cap E_2^c)$ and $\IP(\hat{E}_1\cap \hat{E}_2^c)$ have exponential decaying probability in $k$, respectively in $\hat{k}$, and thus in $t$. 

It remains to control the probability of the event $E^*:=E\cap E_1 \cap \hat{E}_1\cap E_2 \cap \hat{E}_2 $. Note, by definition  there exists at least $\beta c k$ many infected macroscopic sites $\hat{x}_i\in\B^1_{ck}$ on the event $E_1\cap E_2$, which guarantees $\beta c k$ many infected cubes $y_i+[\unaryminus n,n]^d$ with space time centres 
\begin{equation*}
    (y_i,t_i)\in \B_{\hat{x}_i,a}\times [5bk, 5bk+1]\subset ([\unaryminus \ell,\ell]\times[\unaryminus a,a]^{d-1})\times [t-1-6b,t-1].
\end{equation*}
In particular, the set $[\unaryminus \ell,\ell]\times[\unaryminus a,a]^{d-1}$ is the union of $\frac{\ell}{a}$ many disjoint blocks of the form $\hat{x}_i+((2a,2a]\times[\unaryminus a,a]^{d-1})$ and at least $\frac{3}{4}$ of them contain a centre $y_i$.
Equivalently,  for the backward process we have $\beta \hat{c} \hat{k}$ many infected cubes $\hat{y}_i+[\unaryminus \hat{n},\hat{n}]^d$ with space time centres 
\begin{equation*}
    (\hat{y}_i,\hat{t}_i)\in  ([\unaryminus \ell,\ell]\times[\unaryminus \hat{a},\hat{a}]^{d-1})\times [t+1,t+1+6\hat{b}]
\end{equation*}
on the event $\hat{E}_1\cap \hat{E}_2$ and at least $\frac{3}{4}$ of the $\frac{\ell}{\hat{a}}$ many subsets $\hat{x}_i+((\unaryminus 2\hat{a},2\hat{a}]\times[\unaryminus \hat{a},\hat{a}]^{d-1})$ contain a centre $\hat{y}_i$.

Let $a^*=\max\{a,\hat{a}\}$. Then we find at least $\frac{\ell}{2a^*}$ many disjoint space-time boxes $\hat{x}_k+((\unaryminus 2a^*,2a^*]\times[\unaryminus a^*,a^*]^{d-1})\times [t-1-6b,t+1,6\hat{b}]$ containing some $y_{i_k}$ and $\hat{y}_{j_k}$. The probability that there exist an infection path from $y_{i_k}$ to $\hat{y}_{j_k}$ which stays in the above space-time box can clearly be lower bounded by some positive $\delta>0$ for all $k$. Furthermore on $E$ there must not exist any of these paths. Hence by the fact that $\frac{\ell}{2a^*}\geq \gamma t$ with $\gamma=\min\big\{\frac{c}{20b},\frac{\hat{c}}{20\hat{b}}\big\}$ we finally get
\begin{equation*}
    \IP(E)\leq (1-\delta)^{ \gamma t}
\end{equation*} which finishes the proof.
\end{proof}

\begin{proposition}\label{prop:exp_coupling} There exist constants $A,B>0$ such that
    \begin{equation*}
   \sup_{\xi} \IP(\origin\notin \PCI^{\xi}_{t},\tau^{\origin,\xi}=\infty)\leq A \exp(-B t)\quad\text{for all }t\geq0.
\end{equation*}
\end{proposition}
\begin{proof}
     If $\IP(\tau^\origin=\infty)=0$ there is nothing to show. Otherwise let $A,B,\alpha>0$ be the constants from Lemma~\ref{lem:forwad_backward_intersection}. Then
\begin{align*}
    \PC(0\notin \PCI^{\xi}_t)&=\PC(\exists s\geq t: 0\notin \bfK^{\xi}_s)\\
    &\leq \sum_{k=\lfloor t\rfloor}^{\infty} \PC( \B_{\alpha k}\not\subset \bfK^{\xi}_k)+\PC( \B_{\alpha k}\subset \bfK^{\xi}_k,\exists s\in[k,k+1) \text{ s.t. } \zero\notin \bfK_s^{\xi}).
\end{align*}
 We first handle the second summand in the same manner as in \cite[Proof of $(11)$, p. 1400]{garet2014}.
\begin{align*}
    \PC( \B_{\alpha k}\subset \bfK^{\xi}_k,\exists s\in[k,k+1) \text{ s.t. } 0\notin \bfK_t^{\xi})\leq \frac{1}{\IP(\tau^{\origin,\xi}=\infty)}\IP\Big(\tilde{\bfC}_1^{\origin}\not\subset \B_{\alpha k}\Big)
\end{align*}
where $\tilde{\bfC}_t$ is the maximal infection process and the exponential control in $k$ is given by Lemma~\ref{lem:InfContainedExponetialSpeed}. 
Hence it suffices to control the first term and to show that
\begin{equation*}
    \IP( \B_{\alpha k}\not\subset \bfK^{\xi}_k,\tau^{\zero,\xi}=\infty)\leq A \exp(- B k)\quad\text{for all } k\geq 0.
\end{equation*}
Since we are in $\IZ^d$ the ball $\B_{\alpha k}$ only contains polynomial many vertices, and thus we only need to show that for any $t\geq 0$ and $x\in \B_{\alpha t}$ we have
\begin{equation*}
    \IP(x\in \bfC_t^{\underline{1},\xi} \backslash\bfC_t^{\origin,\xi} , \bfC_t^{\zero,\xi}\neq \emptyset)\leq A \exp(-B t).
\end{equation*}
First of all, if $\bfeta_t^{\origin,\xi}(x)\neq \bfeta_t^{\underline{1},\xi}(x)$, then it follows by monotonicity that $\bfeta_t^{\origin,\xi}(x)=0$ and $\bfeta_t^{\underline{1},\xi}(x)=1$. Therefore, by using \eqref{eq:ConditinalDuality} we have that
\begin{align*}
    &\IP(\bfC_t^{\origin,\xi}\cap \{x\}=\emptyset,\bfC_t^{\underline{1},\xi}\cap\{x\}\neq\emptyset, \bfC_t^{\zero,\xi}\neq \emptyset)\\
    =&
    \IP(\bfC_{t-s}^{\origin,\xi}\cap \widehat{\bfC}_{s}^{x,\xi,t}=\emptyset,\bfC_{t-s}^{\underline{1},\xi}\cap \widehat{\bfC}_{s}^{x,\xi,t}\neq \emptyset, \bfC_t^{\zero,\xi}\neq \emptyset)\\
    \leq &\IP(\bfC_{t-s}^{\origin,\xi}\cap \widehat{\bfC}_{s}^{x,\xi,t}=\emptyset,\widehat{\bfC}_{s}^{x,\xi,t}\neq \emptyset, \bfC_{t-s}^{\zero,\xi}\neq \emptyset),
\end{align*}
for all $s\leq t$. Replacing $t$ by $2t$ and $s$ by $t$ we can deduce the desired control by Lemma~\ref{lem:forwad_backward_intersection}.
\end{proof}
\subsection{Verifying the Assumptions}
In this subsection, we build on our prior work and verify the assumptions \eqref{ConjectureAML}--\eqref{ConjectureTEC} in order to apply Theorem~\ref{Conjecture1} and Theorem~\ref{thm:ShapeTheoremNonMonoton} , thereby providing the proofs of Theorem~\ref{thm:CPIUShapeTheorem} and Proposition~\ref{prop:CPDREShapeTheorem}. We start by showing Proposition~\ref{prop:CPDREShapeTheorem} since Theorem~\ref{thm:CPIUShapeTheorem} follows then immediately.
\begin{proof}[Proof of Proposition~\ref{prop:CPDREShapeTheorem}]
    Let $(\bfeta,\bfxi)$ be a worst-case monotone CPDRE which satisfies \eqref{ConjectureTEC} or a CPDP. Moreover, let $(\underline{\bfeta},\underline{\bfxi})$ be a supercritical monotone CPIU that gets dominated by $(\bfeta,\bfxi)$.

    By Corollary~\ref{corollary:AML} the process $(\bfeta,\bfxi)$ satisfies \eqref{ConjectureAML} and in case $(\bfeta,\bfxi)$ is a CPDP the condition \eqref{ConjectureTEC} trivially holds. Furthermore, our coupling  implies that all our results from Section~\ref{sec:RestartingProcedure} can be applied to $(\bfeta,\bfxi)$ since Assumption~\ref{ass:coupling_exists} is satisfied. In particular, the remaining conditions \eqref{ConjectureSC}, \eqref{ConjectureALL} and \eqref{ConjectureFC} follow by Corollary~\ref{corollary:SC}, Proposition~\ref{prop:ALL} and Proposition~\ref{prop:exp_coupling}, respectively. Having established all assumptions \eqref{ConjectureAML}-\eqref{ConjectureTEC} Theorem~\ref{thm:ShapeTheoremNonMonoton} gives us that \eqref{desired_result} holds in case $(\bfeta,\bfxi)$ is a CPDP. Otherwise, if $(\bfeta,\bfxi)$ is a a worst-case monotone CPDRE Theorem~\ref{Conjecture1} implies that the desired result \eqref{desired_result} holds.
\end{proof}
\begin{proof}[Proof of Theorem~\ref{thm:CPIUShapeTheorem}]
    We have to show that \eqref{desired_result} holds for a monotone supercritical CPIU $(\bfeta,\bfxi)$. Clearly, a monotone CPIU is also worst-case monotone and satisfies \eqref{ConjectureTEC}. Therefore Proposition~\ref{prop:CPDREShapeTheorem} applies.
\end{proof}

\textbf{Acknowledgements:} We would like to thank Noemi Kurt for carefully reading the manuscript and her valuable remarks and suggestions regarding the research presented here. MS was supported by the German Research Foundation (DFG) Project ID: 531542160 within the priority programme SPP 2265.

\bibliographystyle{abbrv}
\bibliography{references}

\end{document}